\documentclass{amsart}
\usepackage[utf8]{inputenc}
\usepackage{amssymb}
\usepackage{hyperref}
\usepackage[final]{showkeys} 

\input xy
\xyoption{all}

\theoremstyle{definition}
\newtheorem{mydef}{Definition}[section]
\newtheorem{lem}[mydef]{Lemma}
\newtheorem{thm}[mydef]{Theorem}
\newtheorem{conjecture}[mydef]{Conjecture}
\newtheorem{cor}[mydef]{Corollary}

\newtheorem{hypothesis}[mydef]{Hypothesis}

\newtheorem{defin}[mydef]{Definition}
\newtheorem{example}[mydef]{Example}
\newtheorem{remark}[mydef]{Remark}

\newtheorem{notation}[mydef]{Notation}
\newtheorem{fact}[mydef]{Fact}

\newcommand{\fct}[2]{{}^{#1}#2}

\newcommand{\ba}{\bar{a}}
\newcommand{\bb}{\bar{b}}

\newcommand{\bx}{\bar{x}}
\newcommand{\by}{\bar{y}}

\newcommand{\Ksatpp}[2]{{#1}^{#2\text{-sat}}}
\newcommand{\Ksatp}[1]{\Ksatpp{\K}{#1}}

\newcommand{\sea}{\mathfrak{C}}

\newcommand{\cf}[1]{\text{cf} (#1)}
\newcommand{\seq}[1]{\langle #1 \rangle}
\newcommand{\rest}{\upharpoonright}

\newcommand{\s}{\mathfrak{s}}

\newcommand{\id}{\text{id}}

\newcommand{\leap}[1]{\le_{#1}}
\newcommand{\ltap}[1]{<_{#1}}

\newcommand{\geap}[1]{\ge_{#1}}

\newcommand{\lta}{\ltap{\K}}
\newcommand{\lea}{\leap{\K}}

\newcommand{\gea}{\geap{\K}}

\def\lee{\preceq}

\newcommand{\K}{\mathbf{K}}

\newcommand{\Kslpp}[2]{{#1}^{{#2}\text{-sl}}}
\newcommand{\Kslp}[1]{\Kslpp{\K}{#1}}

\newbox\noforkbox \newdimen\forklinewidth
\forklinewidth=0.3pt \setbox0\hbox{$\textstyle\smile$}
\setbox1\hbox to \wd0{\hfil\vrule width \forklinewidth depth-2pt
 height 10pt \hfil}
\wd1=0 cm \setbox\noforkbox\hbox{\lower 2pt\box1\lower
2pt\box0\relax}
\def\unionstick{\mathop{\copy\noforkbox}\limits}
\newcommand{\nf}{\unionstick}
\newcommand{\nfs}[4]{#2 \nf_{#1}^{#4} #3}

\def\1nf{\unionstick^{(1)}}

\def\2nf{\unionstick^{(2)}}
\def\3nf{\unionstick^{(3)}}

\newcommand{\gtp}{\mathbf{tp}}

\newcommand{\gS}{\mathbf{S}}
\newcommand{\gSna}{\mathbf{S}^{\text{na}}}

\newcommand{\Sna}{\gS^\text{na}}
\newcommand{\Sbs}{\gS^\text{bs}}

\newcommand{\Ii}{\mathbb{I}}

\newcommand{\hanf}[1]{h (#1)}
\newcommand{\ehanf}[1]{\beth_{\left(2^{#1}\right)^+}}

\newcommand{\EM}{\operatorname{EM}}
\newcommand{\Ll}{\mathbb{L}}
\newcommand{\otp}{\operatorname{otp}}

\newcommand{\WGCH}{\operatorname{WGCH}}

\newcommand{\GCH}{\operatorname{GCH}}
\newcommand{\GCHWD}{\operatorname{GCHWD}}

\newcommand{\Card}{\operatorname{CARD}}

\newcommand{\NF}{\operatorname{NF}}

\newcommand{\tlt}{\triangleleft}
\newcommand{\tleq}{\trianglelefteq}

\newcommand{\Cat}{\operatorname{Cat}}

\newcommand{\goodp}{\text{good}^+}

\newcommand{\LS}{\text{LS}}

\newcommand{\Tow}{\mathcal{T}}
\newcommand{\R}{\mathcal{R}}

\title[Categoricity spectrum of large AECs]{The categoricity spectrum of large abstract elementary classes}
\date{\today }

\author{Sebastien Vasey}
\email{sebv@math.harvard.edu}
\urladdr{http://math.harvard.edu/\textasciitilde sebv/}
\address{Department of Mathematics \\ Harvard University \\ Cambridge, Massachusetts, USA}

\begin{document}

\begin{abstract}
  The categoricity spectrum of a class of structures is the collection of cardinals in which the class has a single model up to isomorphism. Assuming that cardinal exponentiation is injective (a weakening of the generalized continuum hypothesis, GCH), we give a complete list of the possible categoricity spectrums of an abstract elementary class with amalgamation and arbitrarily large models. Specifically, the categoricity spectrum is either empty, an end segment starting below the Hanf number, or a closed interval consisting of finite successors of the Löwenheim-Skolem-Tarski number (there are examples of each type). We also prove (assuming a strengthening of the GCH) that the categoricity spectrum of an abstract elementary class with no maximal models is either bounded or contains an end segment. This answers several longstanding questions around Shelah's categoricity conjecture.
\end{abstract}

\maketitle

\tableofcontents

\section{Introduction}

\subsection{Motivation}

A recurring question in mathematics is whether a list of properties characterize a certain structure. To make this precise, let us restrict ourselves here to structures in the model-theoretic sense: a universe with operations and relations on it. We say that a class $K$ of such structures is \emph{categorical} if it contains exactly one member up to isomorphism. The Löwenheim-Skolem theorem, a basic result of model theory, says that if a set $T$ of sentences in first-order logic\footnote{That is, sentences like $\forall x \exists y: x \cdot y = 1 \land y \cdot x = 1$: quantification is over elements and only finite conjunctions and disjunctions are allowed.} has an infinite model, then it has models of all infinite sizes (at least $|T|$). In particular, the class of models of $T$ cannot be categorical.

The following weakening of the definition of categoricity was proposed as a remedy by Łoś \cite{los-conjecture}: for $\lambda$ a cardinal, we say that a class $K$ of structures is \emph{categorical in $\lambda$} if it has a single model \emph{of cardinality $\lambda$} up to isomorphism. We say that a theory (i.e.\ a set of sentences) $T$ is \emph{categorical in $\lambda$} if its class of models is categorical in $\lambda$. The following central result of modern model theory describes the behavior of categoricity  for countable first-order theories:

\begin{fact}[Morley's categoricity theorem, \cite{morley-cip}]
  If a countable first-order theory is categorical in \emph{some} uncountable cardinal, then it is categorical in \emph{all} uncountable cardinals.
\end{fact}

Two examples of classes that are categorical in all uncountable cardinals are the class of vector spaces (over a fixed countable field) and the class of algebraically closed fields (of a fixed characteristic). The reason for categoricity in both classes is that they have well-understood notions of independence (linear independence in  vector spaces, algebraic independence in fields). The proof of Morley's theorem (as well as further developments) tells us that this is not an accident: \emph{any} class of models of an uncountably categorical countable first-order theory will have a nice notion of independence. Such a notion was baptized \emph{forking} by Shelah \cite{shelahfobook} and is now central, both for pure model theory and for applications to other fields of mathematics (on the latter, see e.g.\ Hrushovski's proof of the Mordell-Lang conjecture for function fields \cite{mordell-lang-hrushovski}).

\subsection{Shelah's categoricity conjecture}

One can see Morley's theorem (and its earlier conjecture by Łoś \cite{los-conjecture}) as the catalyst that led to the development of forking for first-order theories. It is natural to ask whether Morley's theorem (and hence forking) can be generalized to other, not necessarily first-order, classes of structures. For uncountable first-order theories, Shelah \cite{sh31} proved that a first-order theory $T$ categorical in \emph{some} $\mu > |T|$ is categorical in \emph{all} $\mu' > |T|$. The next step was to look at infinitary logics such as $\Ll_{\omega_1, \omega}$ (where countably infinite conjunctions and disjunctions are allowed). The situation here is much more complicated, since the compactness theorem fails, but nevertheless Shelah \cite[Conjecture 2]{sh87a} conjectured the following version of Morley's categoricity theorem\footnote{The conjecture also appears as open problem D.3(a) in \cite{shelahfobook}.}:

\begin{conjecture}[Shelah's categoricity conjecture for $\Ll_{\omega_1, \omega}$]\label{categ-syn}
  If an $\Ll_{\omega_1, \omega}$-sentence is categorical in \emph{some} $\mu \ge \beth_{\omega_1}$, then it is categorical in \emph{all} $\mu' \ge \beth_{\omega_1}$.
\end{conjecture}

Here, $\omega_1$ is the first uncountable ordinal, and for an ordinal $\alpha$, $\beth_{\alpha}$ is the cardinal obtained by iterating cardinal exponentiation $\alpha$-many times, starting with $\aleph_0$ (see Section \ref{set-thy-sec} for a precise definition). The spirit of the conjecture is that categoricity \emph{somewhere} high-enough should imply categoricity \emph{everywhere} high-enough (and indeed, even that eventual version is open). A construction of Morley (see Example \ref{categ-examples} here) shows that the ``high-enough'' threshold must in this case be at least $\beth_{\omega_1}$.

As the compactness theorem fails, and moreover there is a plethora of other ``reasonable'' logics to work with, it turns out to be convenient to work semantically rather than syntactically. The framework of \emph{abstract elementary classes} (AECs), introduced by Shelah in the late seventies \cite{sh88}, encompasses ``reasonable'' infinitary logics such as $\Ll_{\omega_1, \omega}$ as well as natural classes of algebraic examples (such as Zilber's pseudoexponential fields \cite{zilber-pseudoexp} or certain classes of modules \cite{bet}). Roughly, an AEC is a partially ordered class $\K = (K, \lea)$ of structures satisfying some closure properties. For example, it must be closed under unions of chains and any member of the class must have a substructure which is also in the class and has size at most a fixed cardinal, the Löwenheim-Skolem-Tarski number of the class, written $\LS (\K)$ (this is a replacement for the size of the theory). The reader should see \ref{aec-def} here for a precise definition, and the introductions of \cite{shelahaecbook}, \cite{baldwinbook09}, or \cite{grossberg2002} for more motivation on AECs. Note that AECs can also be characterized as certain kind of accessible categories, see \cite{beke-rosicky, ct-accessible-jsl}.

Shelah has given the following version of Conjecture \ref{categ-syn} for AECs (see \cite[6.14]{sh702}, \cite[N.4.3]{shelahaecbook}):

\begin{conjecture}[Shelah's categoricity conjecture for AECs]\label{categ-aec}
  If an AEC $\K$ is categorical in \emph{some} $\mu \ge \ehanf{\LS (\K)}$, then $\K$ is categorical in \emph{all} $\mu' \ge \ehanf{\LS (\K)}$.
\end{conjecture}

\subsection{Categoricity in AECs with amalgamation}

One goal of the present paper is to prove (assuming a set-theoretic hypothesis) the statement of Conjecture \ref{categ-aec} for AECs which satisfy the amalgamation property (i.e.\ where every span can be completed to a commuting square). 

Before going further, let us explain our justification for assuming the amalgamation property. First, amalgamation holds in case the AEC comes from a first-order theory (by compactness). Second, it is known \cite{makkaishelah,kosh362, tamelc-jsl} that amalgamation holds (eventually) in any AEC that is categorical above a large cardinal. In fact, Grossberg conjectured \cite[2.3]{grossberg2002} that eventual amalgamation should follow from high-enough categoricity, even without large cardinals. Finally, Shelah showed (assuming a set-theoretic hypothesis) that in an AEC categorical in $\lambda$ and $\lambda^+$, amalgamation holds for models of cardinality $\lambda$ \cite[I.3.8]{shelahaecbook}. Thus in an eventually categorical AEC, eventual amalgamation must hold, so amalgamation is a consequence of the eventual categoricity result we want to prove and hence it seems reasonable to assume it as a starting point.

The first milestone result for general AECs with amalgamation is due to Shelah \cite{sh394} who showed that categoricity in \emph{some successor} $\mu \ge H_2$ implies categoricity in \emph{all} $\mu' \in [H_2, \mu]$. We have set $H_1 := \ehanf{\LS (\K)}$ and $H_2 := \ehanf{H_1}$. Until very recently, it was open whether the successor assumption could be removed, whether one could also go up (i.e.\ get categoricity in all $\mu' > \mu$), and whether the threshold $H_2$ could be lowered to the cardinal $\ehanf{\LS (\K)}$ featuring in Conjecture \ref{categ-aec}. These questions feature in several prominent lists of open problems, for example in \cite[6.14]{sh702}, \cite[Section 9]{grossberg2002}, \cite{baldwin-aec-survey-2006},  \cite[Appendix D]{baldwinbook09}, or the introduction of \cite{shelahaecbook}.

In his book on AECs, Shelah \cite[IV.7.12]{shelahaecbook} asserts that he can remove the successor assumption and also get an upward transfer (this is done at the cost of a set-theoretic hypothesis, see below). The $H_2$ threshold is also slightly lowered, although not quite all the way down to $H_1$. However, Shelah's proof relies on a claim that was not proven at the time the book was written (see also \cite[Section 11]{downward-categ-tame-apal} for an exposition of this proof, modulo the claim). In very recent joint work of the author with Shelah \cite{multidim-v2}, the missing claim was finally proven. It nevertheless remained open whether $H_2$ could be lowered all the way down to $H_1$. This is one contribution of the present paper.

\subsection{Categoricity in AECs with no maximal models}

We also investigate a much weaker framework than AECs with amalgamation: AECs with no maximal models. This was first studied by Shelah and Villaveces \cite{shvi635}, with later contributions by VanDieren \cite{vandierennomax, nomaxerrata}. In these papers, some superstability-like properties were shown to follow from categoricity (assuming again set-theoretic hypotheses). Still, nothing could be said on the categoricity spectrum (i.e.\ the class of categoricity cardinals above $\LS (\K)$, see Definition \ref{cat-spec-def}) and for a long time no further progress was made. Another contribution of the present paper is that the categoricity spectrum in AECs with no maximal models is either bounded or contains an end segment (Corollary \ref{spec-nmm}). This implies the eventual categoricity conjecture in this framework (Corollary \ref{event-nmm}), although we are unable to give an explicit bound on the threshold.

\subsection{Main results}

Let us discuss our results more precisely. Assuming a set-theoretic hypothesis, we list \emph{all} possibilities for the categoricity spectrum  of an AEC $\K$ with amalgamation and arbitrarily large models. There are three possibilities (Corollary \ref{main-cor}): the categoricity spectrum is either empty, an end segment, or a closed interval. Furthermore in case it is an end segment, the first categoricity cardinal must be below $H_1$, and in case it is an interval, there must be finitely-many cardinals between the endpoints and $\LS (\K)$. It was known that each of those three possibilities could happen (see Example \ref{categ-examples}). In particular, for each $n < \omega$ there is an example (due to Hart and Shelah \cite{hs-example} and further analyzed by Baldwin and Kolesnikov \cite{bk-hs}) of an AEC $\K$ with $\LS (\K) = \aleph_0$ which has amalgamation, arbitrarily large models, and is categorical in $\aleph_0, \aleph_1, \ldots, \aleph_n$ but not anywhere above. The present work shows (again assuming amalgamation and arbitrarily large models) that categoricity at $\aleph_{\omega}$ or above must imply eventual categoricity, hence that the Hart-Shelah example is in a sense the only obstruction to transferring categoricity upward.

In passing (and still assuming some set theory), we also show as a consequence of our argument that \emph{tameness}, a locality property introduced by Grossberg and VanDieren in \cite{tamenessone} and used to prove an upward categoricity transfer in \cite{tamenesstwo, tamenessthree} follows from categoricity above $H_1$ in an AEC with amalgamation (Corollary \ref{gv-cor}). This was conjectured in \cite[1.5]{tamenessthree}. 

All of the work described in the previous two paragraphs is done assuming a weakening of the generalized continuum hypothesis (GCH). Recall that GCH is the statement that $2^{\lambda} = \lambda^+$ for any infinite cardinal $\lambda$. We assume here only the \emph{weak generalized continuum hypothesis} (WGCH): $2^{\lambda} < 2^{\lambda^+}$ for any infinite cardinal $\lambda$. It says that cardinal exponentiation is strictly monotonic. The use of this hypothesis is quite prominent in Shelah's work on AECs \cite{shelahaecbook, shelahaecbook2} and can be traced back to a combinatorial principle, the \emph{weak diamond}, proven by Devlin and Shelah \cite{dvsh65} to follow from $2^{\aleph_0} < 2^{\aleph_1}$.

Regarding AECs with no maximal models, we prove assuming the GCH \emph{and} a strengthening of the weak diamond (see Section \ref{set-thy-sec}; this hypothesis is already used in \cite{shvi635}) that the categoricity spectrum is either bounded or contains an end segment. More precisely, we give a list of five possibilities for the categoricity spectrum above a countable limit of fixed points of the $\beth$ operation (Corollary \ref{spec-nmm}). In this case, we do not know whether all of these possibilities can happen.

\subsection{Methods}

The main technical tool of the present paper is the theory of \emph{good frames}, the core object of study in Shelah's two-volume book on AECs \cite{shelahaecbook, shelahaecbook2}. Good frames give a sense in which an AEC is well-behaved locally (i.e.\ at a single cardinal $\lambda$). Key questions around them is when they exist and when they can be transferred upward. The main technical theorem of this paper, Theorem \ref{main-frame-constr}, gives sufficient conditions under which certain good frames exist. These sufficient conditions were known to follow from $\WGCH$, categoricity, and amalgamation, and we show here how to derive them from a strengthening of GCH, categoricity, and no maximal models. This technical theorem thus gives us good frame in many places below the categoricity cardinal.

Using the main result of \cite{tame-succ-v5-toappear}, these good frames are then shown to be ``connected'' in the sense that one can describe what happens in the frames above by looking at what happens in the frames below. Such connected good frames (called successful by Shelah), turn out to be powerful-enough to transfer categoricity and more generally understand the structure of the class completely. This is proven in recent joint work with Saharon Shelah \cite{multidim-v2}: the main result there is (assuming $\WGCH$) that from an $\omega$-successful $\goodp$ frame (roughly: a sequence of $\omega$-many connected good frames), we can derive a condition called \emph{excellence}. Excellence essentially says that for each $n < \omega$ there is an $n$-dimensional notion of nonforking amalgamation, with the usual properties. It implies that the class is very well-behaved and in particular that categoricity transfers can be proven. We emphasize that we do not deal at all with excellence in the present paper. We simply take as a black box that having an $\omega$-successful $\goodp$ frame is enough to prove what we want, see Fact \ref{limit-categ}.

The key tool in the proof of Theorem \ref{main-frame-constr} is the theory of \emph{towers}, connected to the problem of uniqueness of limit models (see \cite{gvv-mlq} for an overview). We give a simplification of this theory here (Section \ref{tower-sec}) incorporating recent developments \cite{uq-forking-mlq} as well as some new results that may have independent interest\footnote{For example, Theorem \ref{indep-sym} gives new conditions for forking symmetry of independent sequences, a key difficulty in \cite{tame-frames-revisited-jsl}.}. Ultimately, these results are used to build the good frames alluded to in the previous paragraph. The expert reader is encouraged to take a quick look at the proof of Theorem \ref{main-frame-constr} to get a better sense of how this is done.

We emphasize that the methods of this paper are usually very local. While AECs with amalgamation or no maximal models are used as convenient test cases, the theorems we obtain make no full use of these assumptions: it often suffices that a condition called solvability (introduced by Shelah in \cite[Chapter IV]{shelahaecbook} - in the first-order case it is equivalent to superstability \cite[5.3]{gv-superstability-jsl}), or really a weakening of it called \emph{semisolvability}, holds near the cardinals where we want to build good frames. Also key and closely related is the notion of a \emph{superlimit model} (due to Shelah \cite[3.1]{sh88}). We prove in fact (assuming the $\WGCH$) that classes with superlimit models at every cardinal satisfy the eventual categoricity conjecture (Corollary \ref{sl-categ-cor}). As discussed further at the end of the paper, this sheds light on several other conjectures of Shelah and paves the way for further work on the local superstability theory of AECs.

\subsection{Acknowledgments}

We thank John T.\ Baldwin, Will Boney, and Marcos Mazari-Armida for comments that helped improve the presentation of this paper. We also thank the referee for multiple thorough reports.

\section{Preliminaries}\label{prelim-sec}

To read the present paper, the reader should preferably have a solid knowledge of AECs and good frames, including knowing Chapter II of \cite{shelahaecbook} and \cite{jrsh875}. Still, we attempt here to give most of the relevant definitions and background facts. The reader may skip this section at first reading and come back to it as needed.

\subsection{Set-theoretic notation}\label{set-thy-sec}

We will often consider intervals of cardinals and may write  $[\lambda, \infty)$ for the class of cardinals greater than or equal to $\lambda$ (i.e.\ $\infty$ denotes an object that is greater than all cardinals).

We assume basic familiarity with ordinals and cardinals. We identify each cardinal $\lambda$ with the least ordinal of cardinality $\lambda$. For a cardinal $\lambda$, $\lambda^+$ denotes the successor of $\lambda$: the minimal cardinal strictly greater than $\lambda$. The \emph{cofinality} $\cf{\alpha}$ of an ordinal $\alpha$ is the least cardinality of an unbounded subset of $\alpha$. A cardinal $\lambda$ is \emph{regular} if $\cf{\lambda} = \lambda$ and singular otherwise.

For $\lambda$ an infinite cardinal, define the following three statements:

\begin{itemize}
\item $\WGCH (\lambda)$ means that $2^{\lambda} < 2^{\lambda^+}$ ($\WGCH$ stands for ``weak generalized continuum hypothesis'').
\item $\GCH (\lambda)$ means that $2^\lambda = \lambda^+$.
\item $\GCHWD (\lambda)$ means that $2^\lambda = \lambda^+$ and for all regular $\theta < \lambda^+$, $\Phi_{\lambda^+} (\{\delta < \lambda^+ \mid \cf{\delta} = \theta\})$ holds\footnote{Although this will not be used, it is known \cite{sh922} that taking $\theta = \cf{\lambda}$ suffices: the conjunction of $2^\lambda = \lambda^+$ with the principle $\Phi_{\lambda^+} (\{\delta < \lambda^+ \mid \cf{\delta} = \cf{\lambda}\})$ is equivalent to $\GCHWD (\lambda)$.}, where for a set $S \subseteq \lambda^+$, $\Phi_{\lambda^+} (S)$ holds if and only if for all $F: \fct{<\lambda^+}{2} \to 2$ there exists $g: \lambda^+ \to 2$ such that for every $f: \lambda^+ \to 2$ the set $\{\delta \in S \mid F (f \rest \delta) = g (\delta)\}$ is stationary (i.e.\ intersects every closed unbounded subset of $\lambda^+$). This was first studied by Devlin and Shelah \cite{dvsh65}, who proved that $2^\lambda < 2^{\lambda^+}$ implies $\Phi_{\lambda^+} (\lambda^+)$ ($\operatorname{WD}$ stands for ``weak diamond'').
\end{itemize}

Note that $\GCHWD (\lambda)$ implies $\GCH (\lambda)$ which implies $\WGCH (\lambda)$. 

For $\Theta$ a class of cardinals, we write $\WGCH (\Theta)$ if $\WGCH (\lambda)$ holds for all $\lambda \in \Theta$. We write $\WGCH$ for $\WGCH (\Card)$, where $\Card$ is the class of all infinite cardinals. Similarly define $\GCH (\Theta)$, $\GCH$, $\GCHWD (\Theta)$, and $\GCHWD$. It is well known that $\GCHWD$ is consistent (it holds for example in Gödel's constructible universe, see \cite[p.~550]{jechbook} and \cite[I.3.4]{vandierennomax}).

For a cardinal $\lambda$ and an ordinal $\alpha$, $\beth_{\alpha} (\lambda)$ is defined recursively as follows: $\beth_0 (\lambda) = \lambda$, $\beth_{\beta + 1} (\lambda) = 2^{\beth_\beta (\lambda)}$, and $\beth_\delta (\lambda) = \sup_{\beta < \delta} \beth_{\beta} (\lambda)$ for $\delta$ a limit ordinal. We write $\beth_\alpha$ for $\beth_\alpha (\aleph_0)$. Also recursively define $\lambda^{+\alpha}$, the $\alpha$th successor of $\lambda$, as follows: $\lambda^{+0} = \lambda$, $\lambda^{+(\alpha + 1)} = \left(\lambda^{+\alpha}\right)^+$, and $\lambda^{+\delta} = \sup_{\beta < \delta} \lambda^{+\beta}$ for $\delta$ limit. For an infinite cardinal $\lambda$, it will be convenient to write $\hanf{\lambda}$ instead of $\ehanf{\lambda}$ (this notation is used already in \cite[4.24]{baldwinbook09}; the relevance of this cardinal is given by Fact \ref{arb-large-fact}). 

\subsection{Model-theoretic notation}

Given a structure $M$, write $|M|$ for its universe and $\|M\|$ for the cardinality of its universe. We often do not distinguish between $M$ and $|M|$, writing e.g.\ $a \in M$ instead of $a \in |M|$. We write $M \subseteq N$ to mean that $M$ is a substructure of $N$.

\subsection{Abstract elementary classes}
An \emph{abstract class} is a pair $\K = (K, \lea)$, where $K$ is a class of structures in a fixed vocabulary $\tau = \tau (\K)$ and $\lea$ is a partial order, $M \lea N$ implies $M \subseteq N$, and both $K$ and $\lea$ respect isomorphisms (the definition is due to Grossberg). We often do not distinguish between $K$ (the class of structures) and $\K$ (the \emph{ordered} class of structures). Any abstract class admits a notion of \emph{$\K$-embedding}: these are functions $f: M \rightarrow N$ such that $f: M \cong f[M]$ and $f[M] \lea N$. Thus one can naturally see $\K$ as a category. Unless explicitly stated, any map $f: M \rightarrow N$ in this paper will be a $\K$-embedding. We write $f: M \xrightarrow[A]{} N$ to mean that $f$ is a $\K$-embedding from $M$ into $N$ which fixes the set $A$ pointwise (so $A \subseteq |M|$). We similarly write $f: M \cong_A N$ for isomorphisms from $M$ onto $N$ fixing $A$.

For $\lambda$ a cardinal, we will write $\K_\lambda$ for the restriction of $\K$ to models of cardinality $\lambda$. Similarly define $\K_{\ge \lambda}$, $\K_{<\lambda}$, or more generally $\K_{\Theta}$, where $\Theta$ is a class of cardinals. 

For an abstract class $\K$, we denote by $\Ii (\K)$ the number of models in $\K$ up to isomorphism (i.e.\ the cardinality of $\K /_{\cong}$). We write $\Ii (\K, \lambda)$ instead of $\Ii (\K_\lambda)$. When $\Ii (\K) = 1$, we say that $\K$ is \emph{categorical}. We say that $\K$ is \emph{categorical in $\lambda$} if $\K_\lambda$ is categorical, i.e.\ $\Ii (\K, \lambda) = 1$.

We say that $\K$ has \emph{amalgamation} if for any $M_0 \lea M_\ell$, $\ell = 1,2$, there is $M_3 \in \K$ and $\K$-embeddings $f_\ell : M_\ell \xrightarrow[M_0]{} M_3$, $\ell = 1,2$. $\K$ has \emph{joint embedding} if any two models can be $\K$-embedded in a common model. $\K$ has \emph{no maximal models} if for any $M \in \K$ there exists $N \in \K$ with $M \lea N$ and $M \neq N$ (we write $M \lta N$). Localized concepts such as \emph{amalgamation in $\lambda$} mean that $\K_\lambda$ has amalgamation.

The definition of an abstract elementary class is due to Shelah \cite{sh88}:

\begin{defin}\label{aec-def}
  An \emph{abstract elementary class (AEC)} is an abstract class $\K$ in a finitary vocabulary satisfying:

  \begin{enumerate}
  \item Coherence: if $M_0, M_1, M_2 \in \K$, $M_0 \subseteq M_1 \lea M_2$ and $M_0 \lea M_2$, then $M_0 \lea M_1$.
  \item Tarski-Vaught chain axioms: if $\seq{M_i : i \in I}$ is a $\lea$-directed system and $M := \bigcup_{i \in I} M_i$, then:
    \begin{enumerate}
    \item $M \in \K$.
    \item $M_i \lea M$ for all $i \in I$.
    \item If $N \in \K$ is such that $M_i \lea N$ for all $i \in I$, then $M \lea N$.
    \end{enumerate}
  \item Löwenheim-Skolem-Tarski axiom: there exists a cardinal $\lambda \ge |\tau (\K)| + \aleph_0$ such that for any $N \in \K$ and any $A \subseteq |N|$, there exists $M \in \K$ with $M \lea N$, $A \subseteq |M|$, and $\|M\| \le |A| + \lambda$. We write $\LS (\K)$ for the least such $\lambda$.
  \end{enumerate}
\end{defin}

\subsection{Types}

In any abstract class $\K$, we can define a semantic notion of type, called Galois or orbital types in the literature (such types were introduced by Shelah in \cite{sh300-orig}). For $M \in \K$, $A \subseteq |M|$, and $b \in M$, we write $\gtp_{\K} (b / A; M)$ for the orbital type of $b$ over $A$ as computed in $M$ (usually $\K$ will be clear from context and we will omit it from the notation). It is the finest notion of type respecting $\K$-embeddings, see \cite[2.16]{sv-infinitary-stability-afml} for a formal definition. For $M \in \K$, we write $\gS_{\K} (M) = \gS (M)$ for $\{\gtp (b / M; N) \mid M \lea N\}$, the class\footnote{If $\K$ is an AEC, $\gS (M)$ will of course be a set.} of all types over $M$. We define naturally what it means for a type to be realized inside a model, to extend another type, and to take the image of a type by a $\K$-embedding.

When $\K$ is an elementary class, $\gtp (b / A; M)$ contains the same information as the usual notion of $\Ll_{\omega, \omega}$-syntactic type. In particular, types in an elementary class are determined by their restrictions to finite sets. This idea was abstracted in \cite{tamenessone} and made into the following definition: for $\chi$ an infinite cardinal, an abstract class $\K$ is \emph{$(<\chi)$-tame} if for any $M \in \K$ and any distinct $p, q \in \gS (M)$, there exists $A \subseteq |M|$ such that $|A| < \chi$ and $p \rest A \neq q \rest A$. We say that $\K$ is \emph{$\chi$-tame} if it is $(<\chi^+)$-tame. Thus elementary classes are $(<\aleph_0)$-tame, but there are examples of non-tame AECs, see e.g.\ \cite[3.2.2]{bv-survey-bfo}.

\subsection{Stability and saturation}

We say that an abstract class $\K$, is \emph{stable in $\lambda$} (for $\lambda$ an infinite cardinal) if $|\gS (M)| \le \lambda$ for any $M \in \K_\lambda$. If $\K$ is an AEC, $\lambda \ge \LS (\K)$, $\K$ is stable in $\lambda$ and $\K$ has amalgamation in $\lambda$, then we will often use without comments the \emph{existence of universal extension} \cite[II.1.16]{shelahaecbook}: for any $M \in \K_\lambda$, there exists $N \in \K_\lambda$ universal over $M$. This means that $M \lea N$ and any extension of $M$ of cardinality $\lambda$ $\K$-embeds into $N$ over $M$.

For $\K$ an AEC and $\lambda > \LS (\K)$, a model $N \in \K$ is called \emph{$\lambda$-saturated} if for any $M \in \K_{<\lambda}$ with $M \lea N$, any $p \in \gS (M)$ is realized in $N$. $N$ is called \emph{saturated} if it is $\|N\|$-saturated.

We will also often use without mention the \emph{model-homogeneous = saturated lemma} \cite[II.1.14]{shelahaecbook}: it says that when $\K_{<\lambda}$ has amalgamation, a model $N \in \K$ is $\lambda$-saturated if and only if it is $\lambda$-model-homogeneous. The latter means that for any $M \in \K$ with $\|M\| < \lambda$, $M \lea N$, any $M' \in \K_{<\lambda}$ $\lea$-extending $M$ can be $\K$-embedded into $N$ over $M$. In particular, assuming amalgamation and joint embedding, there is at most one saturated model of a given cardinality. We write $\Ksatp{\lambda}$ for the abstract class of $\lambda$-saturated models in $\K$ (ordered by the appropriate restriction of $\lea$).

A local notion of saturation is given by the definition of a \emph{limit model}. For an AEC $\K$, $\lambda \ge \LS (\K)$, and $\delta < \lambda^+$ a limit ordinal, $N$ is \emph{$(\lambda, \delta)$-limit over $M$} if $M, N \in \K_\lambda$ and there exists an increasing continuous chain $\seq{M_i : i \le \delta}$ in $\K_\lambda$ with $M_{i + 1}$ universal over $M_i$ for all $i < \delta$ such that $M_0 = M$ and $M_\delta = N$. We say that $N$ is \emph{limit over $M$} if it is $(\lambda, \delta)$-limit over $M$ for some $\lambda$ and $\delta$. We say that $N$ is \emph{limit} if it is limit over $M$ for some $M$. By a back and forth argument \cite[1.3.6]{shvi635}, whenever $M_0, M_1, M_2 \in \K_\lambda$, $M_\ell$ is $(\lambda, \delta_\ell)$-limit over $M_0$ for $\ell = 1,2$, and $\cf{\delta_1} = \cf{\delta_2}$, we have that $M_1 \cong_{M_0} M_2$. Similarly, assuming joint embedding any two limit models of the same length (i.e.\ with the same limit ordinal $\delta$) are isomorphic. The question of \emph{uniqueness of limit models} asks whether the previous two results hold even when the lengths do not have the same cofinality (Fact \ref{sym-uq-lim} gives a positive answer to this question under some superstability-like assumptions).

\subsection{Superstability, forking, and symmetry}

Let $\K$ be an AEC. For $M \lea N$ and $p \in \gS (N)$, we say that $p$ \emph{does not $\lambda$-split over $M$} if whenever $N_1, N_2 \in \K_\lambda$ are such that $M \lea N_\ell \lea N$ for $\ell = 1,2$ and $f: N_1 \cong_M N_2$ is given, we have that $f (p \rest N_1) = p \rest N_2$. This gives a notion of independence with which we will be able, under the definition of superstability given below (already implicit in for example \cite{shvi635}), to define a forking-like notion. 

\begin{defin}[{\cite[10.1]{indep-aec-apal}}]
  An AEC $\K$ is \emph{$\lambda$-superstable} if:

  \begin{enumerate}
  \item $\lambda \ge \LS (\K)$ and $\K_\lambda \neq \emptyset$.
  \item $\K_\lambda$ has amalgamation, joint embeddings, and no maximal models.
  \item $\K$ is stable in $\lambda$.
  \item Splitting has universal local character: whenever $\delta < \lambda^+$ is a limit ordinal and $\seq{M_i : i \le \delta}$ is increasing continuous in $\K_\lambda$ with $M_{i + 1}$ universal over $M_i$ for all $i < \delta$,  then for any $p \in \gS (\bigcup_{i < \delta} M_i)$, there exists $i < \delta$ such that $p$ does not $\lambda$-split over $M_i$.
  \end{enumerate}
\end{defin}

Note that the definition is completely local: it only discusses models of cardinality $\lambda$. In a $\lambda$-superstable AEC, we define nonforking by ``shifting'' nonsplitting by a universal extension:

\begin{defin}[{\cite[3.8]{ss-tame-jsl}}]\label{forking-def}
  Let $\K$ be a $\lambda$-superstable AEC. Let $M \lea N$ both be limit models and let $p \in \gS (N)$. We say that $p$ \emph{does not $\lambda$-fork over $M$} if there exists $M_0 \in \K_\lambda$ such that $M$ is universal over $M_0$ and $p$ does not $\lambda$-split over $M_0$. Usually, $\lambda$ will be clear from context so we will just say that \emph{$p$ does not fork over $M$}. 
\end{defin}

The definition of forking may seem technical, so the reader can immediately forget it and remember instead the next two facts:

\begin{fact}\label{forking-fact}
  If $\K$ is a $\lambda$-superstable AEC, then forking has the following properties:

  \begin{enumerate}
  \item Invariance: if $M \lea N$ are both limit models in $\K_\lambda$ and $p \in \gS (N)$ does not fork over $M$, then if $f: N \cong N'$, $f (p)$ does not fork over $f[M]$.
  \item Monotonicity: if $M \lea M' \lea N' \lea  N$ are all limit models in $\K_\lambda$, $p \in \gS (N)$ does not fork over $M$, then $p \rest N'$ does not fork over $M'$.
  \item Universal local character: if $\delta < \lambda^+$, $\seq{M_i : i \le \delta}$ is an increasing continuous chain of limit models in $\K_\lambda$, with $M_{i + 1}$ universal over $M_i$ for all $i < \delta$, then for any $p \in \gS (M_\delta)$, there exists $i < \delta$ such that $p$ does not fork over $M_i$.
  \item Uniqueness: if $M \lea N$ are both limit models in $\K_\lambda$ and $p,q \in \gS (N)$ do not fork over $M$, then $p \rest M = q \rest M$ implies $p = q$.
  \item Extension: Let $M \lea N$ both be limit models in $\K_\lambda$. If $p \in \gS (M)$, then there exists $q \in \gS (N)$ such that $q$ does not fork over $M$ and $q$ extends $p$.
  \item Transitivity: Let $M_0 \lea M_1 \lea M_2$ all be limit models in $\K_\lambda$. Let $p \in \gS (M_2)$ and assume that $p$ does not fork over $M_1$ and $p \rest M_1$ does not fork over $M_0$. Then $p$ does not fork over $M_0$.
  \item Universal continuity: Let $\delta < \lambda^+$ be a limit ordinal and let $\seq{M_i : i \le \delta}$ be an increasing continuous chain of limit models with $M_{i + 1}$ universal over $M_i$. Let $\seq{p_i : i < \delta}$ be given such that for all $i < \delta$, $p_i \in \gS (M_i)$, and $p_i$ is a nonforking extension of $p_0$. Then there exists a unique $p_\delta \in \gS (M_\delta)$ such that $p_\delta$ does not fork over $M_0$. In particular, $p_\delta$ extends each $p_i$.
  \item Disjointness: If $M \lea N$ are both limit models in $\K_\lambda$ and $p \in \gS (N)$ does not fork over $M$, then $p$ is algebraic if and only if $p \rest M$ is algebraic.
  \end{enumerate}
\end{fact}
\begin{proof}
  Invariance, monotonicity, and universal local character are straightforward to check from the definition. Extension is by \cite[I.4.10]{vandierennomax}, uniqueness is \cite[2.16]{uq-forking-mlq}, and transitivity follows on general grounds (see e.g.\ the proof of \cite[II.2.18]{shelahaecbook}). To prove universal continuity, take $p_\delta$ to be the nonforking extension of $p_0$ and use uniqueness, universal local character, and transitivity. Finally, to see disjointness, first use universal local character and transitivity to find $M_0$ and a limit ordinal $\delta < \lambda^+$ so that $M$ is $(\lambda, \delta)$-limit over $M_0$ and $p$ does not fork over $M_0$. Let $N'$ be $(\lambda, \delta)$-limit over $N$, hence over $M_0$. Let $q$ be the nonforking extension of $p$ to $N'$. If $p \rest M$ is algebraic, then clearly $p$ is algebraic. Now if $p$ is algebraic, then for the same reason $q$ is algebraic. By \cite[2.7]{uq-forking-mlq}, $q$ and $p \rest M$ are conjugates, so $p \rest M$ is also algebraic, as desired.
\end{proof}

\begin{fact}[The canonicity theorem]\label{canon-thm}
  Let $\K$ be a $\lambda$-superstable AEC. Assume we have a relation ``$p$ is free over $M$'' for a type $p \in \gS (N)$ and $M, N \in \K_\lambda$ \emph{limit models}. If this relation satisfies invariance, monotonicity, universal local character, uniqueness, and extension (in the sense given in the statement of Fact \ref{forking-fact}), then $p$ is free over $M$ if and only if $p$ does not $\lambda$-fork over $M$.
\end{fact}
\begin{proof}
  As in the proof of \cite[9.6]{indep-aec-apal}.
\end{proof}

The following property of forking is crucial. It is open whether it follows from superstability. The definition we give is not the same as in \cite{vandieren-symmetry-apal} but is equivalent by \cite[2.18]{uq-forking-mlq}.

\begin{defin}\label{sym-def}
  Let $\K$ be a $\lambda$-superstable AEC. We say that $\K$ has \emph{$\lambda$-symmetry} if for any two limit model $M \lea N$ in $\K_\lambda$ and any $a, b \in N$, the following are equivalent:

  \begin{enumerate}
  \item There exists $M_b \lea N_b$ in $\K_\lambda$ both limits such that $N \lea N_b$, $M \lea M_b$, $b \in M_b$, and $\gtp (a / M_b; N_b)$ does not fork over $M$.
  \item There exists $M_a \lea N_a$ in $\K_\lambda$ both limits such that $N \lea N_a$, $M \lea M_a$, $a \in M_a$, and $\gtp (b / M_a; N_a)$ does not fork over $M$.
  \end{enumerate}
\end{defin}

Symmetry implies the uniqueness of limit models (whether symmetry is \emph{needed} for this is a major open question). This is due to VanDieren \cite{vandieren-symmetry-apal}. We will give another proof of this in Section \ref{tower-sec}, after Lemma \ref{full-constr-2}.

\begin{fact}\label{sym-uq-lim}
  Let $\K$ be a $\lambda$-superstable AEC with $\lambda$-symmetry. Let $M_0, M_1, M_2 \in \K_\lambda$ be given. If both $M_1$ and $M_2$ are limit over $M_0$, then $M_1 \cong_{M_0} M_2$. In particular, any two limit models in $\K_\lambda$ are isomorphic.
\end{fact}

This leads to yet another key property of forking (we use symmetry and Fact \ref{sym-uq-lim} to get that any two limit models are isomorphic):

\begin{fact}[The conjugation property]\label{conj-prop}
  Let $\K$ be a $\lambda$-superstable AEC with $\lambda$-symmetry. Let $M \lea N$ be limit models in $\K_\lambda$ and let $p \in \gS (N)$. If $p$ does not fork over $M$, then there is an isomorphism $f: N \cong M$ such that $f (p) = p \rest M$.
\end{fact}
\begin{proof}
  As in \cite[III.1.21]{shelahaecbook}.
\end{proof}

It is worth noting that both superstability and symmetry follow from categoricity, in an AEC with amalgamation and no maximal models. We will state stronger results in Section \ref{arb-large-sec}, but it will be easier to quote from:

\begin{fact}[Structure of categorical AECs with amalgamation]\label{categ-struct}
  Let $\K$ be an AEC with arbitrarily large models. Let $\mu > \LS (\K)$ be such that $\K_{<\mu}$ has amalgamation and no maximal models. If $\K$ is categorical in $\mu$, then:

  \begin{enumerate}
  \item For any $\lambda \in [\LS (\K), \mu)$, $\K$ is $\lambda$-superstable and has $\lambda$-symmetry.
  \item For any $\lambda \in (\LS (\K), \mu]$, $\Ksatp{\lambda}$ is an AEC with $\LS (\Ksatp{\lambda}) = \lambda$. In particular, the model of cardinality $\mu$ is saturated.
  \end{enumerate}
\end{fact}
\begin{proof}
  That superstability holds below the categoricity cardinal essentially appears in \cite[2.2.1]{shvi635}, but a full proof is in \cite{shvi-notes-apal}. That symmetry similarly follows from categoricity  is in \cite[5.7]{categ-saturated-afml}, and the statement on saturated models is also proven there (it is a direct consequence of \cite{vandieren-chainsat-apal}).
\end{proof}

\subsection{Good frames}

Good $\lambda$-frames were introduced by Shelah in \cite[II]{shelahaecbook} as a bare-bone axiomatization of superstability.  We give a simplified definition here.

\begin{defin}[{\cite[II.2.1]{shelahaecbook}}]\label{good-frame-def}
  A \emph{good $\lambda$-frame} is a triple $\s = (\K, \nf, \Sbs)$ where:

\begin{enumerate}
\item $\K$ is an AEC such that:
  \begin{enumerate}
  \item $\lambda \ge \LS (\K)$.
  \item $\K_\lambda \neq \emptyset$.
  \item $\K_\lambda$ has amalgamation, joint embedding, and no maximal models.
  \item $\K$ is stable\footnote{In Shelah's original definition, only the set of basic types is required to be stable. However full stability follows, see \cite[II.4.2]{shelahaecbook}.} in $\lambda$.
  \end{enumerate}
  \item For each $M \in \K_\lambda$, $\Sbs (M)$ (called the set of \emph{basic types} over $M$) is a set of nonalgebraic types over $M$ satisfying the \emph{density property}: if $M \lta N$ are both in $\K_\lambda$, there exists $a \in |N| \backslash |M|$ such that $\gtp (a / M; N) \in \Sbs (M)$.
  \item $\nf$ is an (abstract) independence relation on the basic types satisfying invariance, monotonicity, extension existence, uniqueness, continuity, local character, and symmetry (see \cite[II.2.1]{shelahaecbook} for the full definition of these properties).
\end{enumerate}

We say that $\s$ is \emph{type-full} \cite[III.9.2(1)]{shelahaecbook} if for any $M \in \K_\lambda$, $\Sbs (M) = \gSna (M)$, the set of \emph{all} nonalgebraic types over $M$. Rather than explicitly using the relation $\nf$, we will say that $\gtp (a / M; N)$ \emph{does not $\s$-fork} over $M_0$ if $\nfs{M_0}{a}{M}{N}$ (this is well-defined by the invariance and monotonicity properties). When $\s$ is clear from context, we omit it (this does not conflict with previous terminology by Facts \ref{canon-thm} and \ref{frame-ss}). We say that a good $\lambda$-frame $\s$ is \emph{on $\K$} if the underlying AEC of $\s$ is $\K$. We say that $\s$ is \emph{categorical} if $\K$ is categorical in $\lambda$.
\end{defin}

\begin{remark}
  We will \emph{not} use the axiom (B) \cite[II.2.1]{shelahaecbook} requiring the existence of a superlimit model of size $\lambda$. In fact many papers (e.g.\ \cite{jrsh875}) define good frames without this assumption. Further, we gave a shorter list of properties that in Shelah's original definition, but the other properties follow, see \cite[II.2]{shelahaecbook}. 
\end{remark}

The reader can forget about the class of basic types: in this paper, we will work exclusively with type-full frames. In this case, the existence of a good frame is stronger than superstability:

\begin{fact}\label{frame-ss}
  If the AEC $\K$ has a type-full good $\lambda$-frame, then $\K$ is $\lambda$-superstable and has $\lambda$-symmetry.
\end{fact}
\begin{proof}
  By \cite[4.2]{bgkv-apal}, $\s$-nonforking implies $\lambda$-nonsplitting. This immediately gives that $\K$ is $\lambda$-superstable. Using canonicity (Fact \ref{canon-thm}), $\lambda$-symmetry then follows from the symmetry axiom of good frames.
\end{proof}

The converse, getting a good frame from superstability and symmetry, is one of the main focus of this paper. One issue is that in a $\lambda$-superstable AEC, $\lambda$-nonforking is only well-behaved over limit models, and in general the class of limit models may not be closed under unions, hence may not form an AEC. If $\lambda > \LS (\K)$, the class of limit models is the same as the class of saturated models in $\K_\lambda$ and we often \emph{do} get that it is closed under unions (see e.g.\ Fact \ref{categ-struct}). Restricting to the class of limit models in $\lambda$, we then get that the AEC is categorical in $\lambda$, so we might as well assume this to begin with. A more serious issue is that the local character and continuity properties given by Fact \ref{forking-fact} are weaker than the corresponding ones for good frames (because in the definition of local character for good frames we do not require that the models in the chain are universal over the previous ones). This will be circumvented by using the theory of towers and previous result on what Jarden and Shelah call \emph{almost good frames}: good frames that still have continuity but have only the weak version of local character of superstable AECs. For now we state what we can get from Fact \ref{forking-fact}:

\begin{fact}\label{ss-frame}
  If $\K$ be an AEC which is $\lambda$-superstable, has $\lambda$-symmetry, and is categorical in $\lambda$, then $\lambda$-nonforking induces a triple $\s = (\K, \nf, \Sna)$ which satisfies all the axioms of a type-full good $\lambda$-frame, except perhaps for local character and continuity. Moreover, it has the conjugation property (in the sense of Fact \ref{conj-prop}) and satisfies the universal versions of local character and continuity stated in Fact \ref{forking-fact}.
\end{fact}
\begin{proof}
  Immediate from Fact \ref{forking-fact} and the definition of superstability.
\end{proof}
\begin{remark}
  Such weak good frames are called (modulo very minor variations in the definition) \emph{$H$-almost good frames} by Shelah \cite[VII.5.9]{shelahaecbook2}.
\end{remark}

In \cite[\S III.1]{shelahaecbook}, Shelah defines the following positive properties of good frames: weakly successful, successful, $n$-successful, and $\goodp$. He also defines what it means to take the successor $\s^+$ of a successful $\goodp$ frame and even to take the $n$th successor, $\s^{+n}$ of an $n$-successful $\goodp$ frame. We do not repeat the definitions here. The reader may look at \cite[\S2]{counterexample-frame-afml} for an overview. We will mostly just use these terms as ``black boxes'': they will appear in certain facts and theorems, but we will never need their exact definitions. The only fact that the reader must know is that if $\s$ is an $n$-successful $\goodp$ frame on the AEC $\K$, $n \ge 1$, then $\s$ can be extended to be type-full and one can define its $n$th successor $\s^{+n}$ and this is a type-full good $\lambda^{+n}$-frame on $\Ksatp{\lambda^{+n}}$. If $\s$ is $\omega$-successful (i.e.\ $n$-successful for all $n < \omega$), powerful results from \cite{multidim-v2} imply (assuming WGCH) that the good frame can be lifted to any cardinal above $\lambda$ and, in a sense, the structure of the AEC above $\lambda$ is completely understood. For an overview of what we will use, see Fact \ref{limit-categ}.

\section{Nice stability and superlimits}

In this section, we prove several technical results in the setup of nice stability, implicit already in \cite{shvi635}. This turns out to be the right context to study stability in AECs that may not have amalgamation. The main result is Theorem \ref{nicely-stable-sl}, showing roughly speaking that restricting to a superlimit (defined below) preserves nice stability and gives a class that is very close to the original one.

The notion of a \emph{superlimit model} is another attempt at defining a local notion of saturation. It was introduced by Shelah \cite[I.3.3]{shelahaecbook}. We give a definition that makes sense in any abstract class, but we will apply it to $\K_\lambda$, for $\K$ an AEC and $\lambda \ge \LS (\K)$.

\begin{defin}\label{sl-def}
  Let $\K$ be an abstract class and let $M \in \K$.

  \begin{enumerate}
  \item $M$ is \emph{universal} if for every $M_0 \in \K$ there exists a $\K$-embedding $f: M_0 \rightarrow M$.
  \item $M$ is \emph{superlimit} if:

    \begin{enumerate}
    \item $M$ has a proper extension.
    \item $M$ is universal.
    \item Whenever $\delta$ is a limit ordinal and $\seq{M_i : i \le \delta}$ is increasing continuous, if $M \cong M_i$ for all $i < \delta$, then $M \cong \bigcup_{i < \delta} M_i$.
    \end{enumerate}
  \end{enumerate}
\end{defin}
\begin{remark}\label{categ-sl}
  If an abstract class $\K$ is categorical, then any model in $\K$ is universal. If in addition $\K$ has no maximal models, any model in $\K$ is superlimit.
\end{remark}

Given a superlimit, the class of models isomorphic to it generates an AEC:

\begin{defin}[{\cite[II.1.25]{shelahaecbook}}]
  Let $\K$ be an AEC, let $\lambda \ge \LS (\K)$. For $M$ a superlimit in $\K_\lambda$, let $\K^{[M]}$, the \emph{AEC generated by $M$}, be defined as follows: $N \in \K^{[M]}$ if and only if for all $A \subseteq |N|$ with $|A| \le \lambda$, there exists $N_0 \in \K$ such that $N_0 \cong M$, $A \subseteq |N_0|$, and $N_0 \lea N$. Order $\K^{[M]}$ with the restriction of the ordering on $\K$.
\end{defin}

We will use the following straightforward results without comments:

\begin{fact}[{\cite[II.1.26]{shelahaecbook}}]\label{sl-basic}
  Let $\K$ be an AEC and let $\lambda \ge \LS (\K)$. If $M$ is a superlimit in $\K_\lambda$, then:

  \begin{enumerate}
  \item $\K^{[M]}$ is an AEC with $\LS (\K^{[M]}) = \lambda$.
  \item $\K_{\le \lambda}^{[M]} = \{N_0 \in \K_\lambda \mid N_0 \cong M\}$. In particular, $\K_{<\lambda}^{[M]} = \emptyset$ and $\K^{[M]}$ is categorical in $\lambda$.
  \item $\K_\lambda^{[M]}$ has no maximal models.
  \item For any $N_0 \in \K_\lambda$, there exists $N \in \K_\lambda^{[M]}$ such that $N_0 \lea N$.
  \end{enumerate}
\end{fact}

Note that a superlimit is unique if it exists:

\begin{fact}\label{sl-uq}
  Let $\K$ be an AEC and let $\lambda \ge \LS (\K)$. If $M$ and $N$ are superlimits in $\K_\lambda$, then $M \cong N$. In particular, $\K^{[M]} = \K^{[N]}$.
\end{fact}
\begin{proof}
  Straightforward: build a chain with interleaved copies of $M$ and $N$. See for example \cite[I.3.7(1)]{shelahaecbook}.
\end{proof}

This justifies the following definition:

\begin{defin}\label{ksl-def}
  Let $\K$ be an AEC and let $\lambda \ge \LS (\K)$ be such that $\K_\lambda$ has a superlimit. We write $\Kslp{\lambda}$ for the class $\K^{[M]}$, where $M \in \K_\lambda$ is superlimit (the choice of $M$ does not matter by Fact \ref{sl-uq}).
\end{defin}

The class $\Kslp{\lambda}_\lambda$ is dense in $\K_\lambda$, in the following sense\footnote{This is a special case of the definition of a skeleton, see \cite[5.3]{indep-aec-apal} but since we have no use for skeletons in this paper, we chose to only study the simpler case.}:

\begin{defin}\label{dense-def}
  A class $K^\ast$ of structures is \emph{dense} in an abstract class $\K = (K, \lea)$ if $K^\ast \subseteq K$ and for any $M \in \K$ there exists $N \in K^\ast$ such that $M \lea N$. We say that an abstract class $\K^\ast = (K^\ast, \leap{\K^\ast})$ is \emph{dense in $\K$} if $K^\ast$ is dense in $\K$ and for $M, N \in K^\ast$, $M \leap{\K^\ast} N$ if and only if $M \lea N$.
\end{defin}
\begin{remark}\label{sl-dense}
  By Fact \ref{sl-basic}, whenever $\K$ is an AEC, $\lambda \ge \LS (\K)$, and $\K_\lambda$ has a superlimit, then $\Kslp{\lambda}_\lambda$ is dense in $\K_\lambda$.
\end{remark}

We have the following obvious transitivity property:

\begin{remark}\label{skel-trans}
  If $\K^{\ast \ast}$ is dense in $\K^\ast$ and $\K^\ast$ is dense in $\K$, then $\K^{\ast \ast}$ is dense in $\K$.
\end{remark}

Many properties of an abstract class are preserved when passing to a dense subclass. To state the next remark, we need the following definition (which is already used in \cite{shvi635}):

\begin{defin}\label{ab-def}
  Let $\K$ be an abstract class. An \emph{amalgamation base} is a model $M_0 \in \K$ such that for any $M_1, M_2 \in \K$ with $M_0 \lea M_\ell$, $\ell = 1,2$, there exists $M_3 \in \K$ and $\K$-embeddings $f_\ell : M_\ell \xrightarrow[M_0]{} M_3$.
\end{defin}

The following combinatorial result of Shelah \cite[I.3.8]{shelahaecbook} gives us a way to find amalgamation bases:

\begin{fact}\label{ap-univ-comb}
  Let $\K$ be an AEC and let $\lambda \ge \LS (\K)$. Assume $\WGCH (\lambda)$. If $M$ is superlimit in $\K_\lambda$ and $\K_{\lambda^+}$ has a universal model, then $M$ is an amalgamation base in $\K_\lambda$.
\end{fact}

Note that an abstract class has amalgamation precisely when all its models are amalgamation bases. The following are all easy consequences of the definitions (many appear already in \cite[\S5]{indep-aec-apal}). They will be used without further comments:

\begin{remark}\label{skel-basics}
  Let $\K$ be an abstract class and let $\K^\ast$ be dense in $\K$.

  \begin{enumerate}
  \item $\K \neq \emptyset$ if and only if $\K^\ast \neq \emptyset$.
  \item $\K$ has no maximal models if and only if $\K^\ast$ has no maximal models.
  \item $\K$ has joint embedding if and only if $\K^\ast$ has joint embedding.
  \item For any $M \in \K^\ast$, $M$ is an amalgamation base in $\K^\ast$ if and only if $M$ is an amalgamation base in $\K$.
  \item For any $M \in \K^\ast$, $M$ is superlimit in $\K^\ast$ if and only if $M$ is superlimit in $\K$.
  \item If $\K^\ast$ is closed under unions of $\omega$-chains, then any superlimit (in $\K$) is in $\K^\ast$.
  \end{enumerate}
\end{remark}

Roughly, an AEC is nicely stable if its class of amalgamation bases is dense and behaves like a stable first-order theory. As usual, this is localized to a fixed cardinal $\lambda$. The definition appears for the first time in \cite[2.3]{aec-stable-aleph0-apal} but is studied already in \cite{shvi635}.

\begin{defin}
  An AEC $\K$ is \emph{nicely $\lambda$-stable} (or \emph{nicely stable in $\lambda$}) if:

  \begin{enumerate}
  \item $\lambda \ge \LS (\K)$ and $\K_\lambda \neq \emptyset$.
  \item $\K_\lambda$ has joint embedding and no maximal models.
  \item Density of amalgamation bases: for any $M \in \K_\lambda$ there exists $N \in \K_\lambda$ which is an amalgamation base in $\K_\lambda$ and so that $M \lea N$.
  \item For any amalgamation base (in $\K_\lambda$) $M$, there exists an amalgamation base $N \in \K_\lambda$ with $N$ universal over $M$.
  \item Any limit model in $\K_\lambda$ is an amalgamation base (in $\K_\lambda$).
  \end{enumerate}

  For $\Theta$ a class of cardinals, we say that $\K$ is \emph{nicely $\Theta$-stable} (or \emph{nicely stable in $\Theta$}) if $\K$ is nicely $\lambda$-stable for every $\lambda \in \Theta$.
\end{defin}

\begin{remark}\label{nice-stab-ap}
  Let $\K$ be an AEC and let $\lambda \ge \LS (\K)$. If $\K_\lambda \neq \emptyset$, $\K_\lambda$ has amalgamation, joint embedding, no maximal models, and $\K$ is stable in $\lambda$, then $\K$ is nicely $\lambda$-stable (universal extensions exist by \cite[II.1.16]{shelahaecbook}; all the other properties are easy to check). In particular, if $\K$ is $\lambda$-superstable then $\K$ is nicely $\lambda$-stable. In setups without amalgamation, it is known \cite{shvi635} that nice stability follows from categoricity, no maximal models, and GCH with enough instances of weak diamond ($\GCHWD$). See Fact \ref{semisolv-fact}(\ref{semisolv-1}) here.
\end{remark}

We will often assume in addition to nice stability that there is a superlimit. In this case, the superlimit will be limit, for all possible lengths. For the convenience of the reader, we sketch a proof.

\begin{lem}\label{sl-lim}
  Let $\K$ be a nicely $\lambda$-stable AEC. If $M$ is superlimit in $\K_{\lambda}$, then $M$ is $(\lambda, \delta)$-limit for all limit ordinals $\delta < \lambda^+$.
\end{lem}
\begin{proof}
  Fix a limit ordinal $\delta < \lambda^+$. We build $\seq{M_i : i \le \delta}$, $\seq{N_i : i \le \delta}$ increasing continuous in $\K_\lambda$ such that for all $i < \delta$:

  \begin{enumerate}
  \item $M_i \cong M$.
  \item $N_i$ is an amalgamation base.
  \item $N_{i + 1}$ is universal over $N_i$.
  \item $M_i \lea N_i \lea M_{i + 1}$.
  \end{enumerate}

  This is possible: take $M_0 := M$ and $N_0$ an amalgamation base containing $M_0$. For $i$ limits, take unions ($N_i$ is an amalgamation base since it is a limit model, and $M_i \cong M$ by definition of a superlimit). Given $M_i$ and $N_i$, use universality of the superlimit to pick $M_{i + 1}$ containing $N_i$ which is isomorphic to $M$. Then pick $N_{i + 1}'$ an amalgamation base containing $M_{i + 1}$ and let $N_{i + 1}$ be universal over $N_{i + 1}'$.

  This is enough: the chains are interleaved so $M_\delta = N_\delta$. By definition of a superlimit, $M_\delta \cong M$ and by construction $N_\delta$ is $(\lambda, \delta)$-limit.
\end{proof}

The following technical result will be used in Section \ref{categ-nmm-sec}:

\begin{lem}\label{nice-stab-elem}
  Let $\K$ be a nicely $\lambda$-stable AEC. Let $M \lea N$ be in $\K_\lambda$.

  \begin{enumerate}
  \item If $M$ is $(\lambda, \delta_1)$-limit and $N$ is $(\lambda, \delta_2)$-limit, then $M \lee_{\Ll_{\infty, \theta}} N$, where $\theta := \min (\cf{\delta_1}, \cf{\delta_2})$.
  \item If $M$ and $N$ are superlimits, then $M \lee_{\Ll_{\infty, \lambda}} N$.
  \end{enumerate}
\end{lem}
\begin{proof}
  The second part follows from the first and Lemma \ref{sl-lim}. Let us prove the first part. Write $\theta_\ell := \cf{\delta_\ell}$, for $\ell = 1,2$. Note that $\theta \le \theta_\ell$ for $\ell = 1,2$. Let $\phi (\bx)$ be an $\Ll_{\infty, \theta}$ formula and let $\ba$ be a sequence of length $\alpha$ in $M$, with $\alpha < \theta$. We want to see that $M \models \phi[\ba]$ if and only if $N \models \phi[\ba]$. We proceed by induction on the structure of $\phi$. If $\phi$ is quantifier-free, then this holds because by the definition of an abstract class, $M \subseteq N$. If $\phi$ is a conjunction, disjunction, or negation, the result follows from the induction hypothesis. Assume now that $\phi = \exists \by \psi (\bx; \by)$, where $\by$ has length $\beta < \theta$. If $M \models \phi[\ba]$, then by the induction hypothesis $N \models \phi[\ba]$. Assume now that $N \models \phi[\ba]$ and pick $\bb$ in $N$ such that $N \models \psi[\ba; \bb]$. Let $\theta_1 := \cf{\delta_1}$. Since $M$ and $N$ are limit models, they  are amalgamation bases. Thus we can pick $N' \in \K_\lambda$ which is $(\lambda, \theta_1)$-limit over $N$, hence also over $M$. By what has just been said (the induction hypothesis), $N' \models \psi[\ba; \bb]$. Now, $M$ is $(\lambda, \theta_1)$-limit over some $M_0$. By making $M_0$ bigger if necessary (noting that $\alpha < \theta \le \theta_1$), we can assume that $\ba$ is contained entirely in $M_0$. Note that $N'$ is also $(\lambda, \theta_1)$-limit over $M_0$, so by uniqueness of limit models of the same length, there is an isomorphism $f: N' \cong_{M_0} M$. Thus $M \models \psi[\ba; f (\bb)]$, hence $M \models \phi[\ba]$, as desired.
\end{proof}

The next result says that nice stability plays very well with taking dense subclasses:

\begin{thm}\label{dense-nicely-stable}
  Let $\K$ and $\K^\ast$ be AECs and let $\lambda \ge \LS (\K) + \LS (\K^\ast)$. If $\K_\lambda^\ast$ is dense in $\K_\lambda$, then $\K$ is nicely $\lambda$-stable if and only if $\K^\ast$ is nicely $\lambda$-stable.
\end{thm}
\begin{proof}
  First, check that for any $M_0 \lea M$ both in $\K_\lambda$, if $M$ is limit over $M_0$ (in $\K_\lambda$), then $M \in \K_\lambda^\ast$ (write $M$ as an increasing union of models in $\K^\ast$). It follows that limit models in $\K$ and $\K^\ast$ coincide. Notice also that if $M_0 \lea M_1 \lea M_2$, $M_2$ is universal over $M_1$, and $M_0$ is an amalgamation base, then $M_2$ is universal over $M_0$. Thus it is enough to check existence of universal extensions on a dense class of amalgamation bases. The rest of the proof is straightforward (see also Remark \ref{skel-basics}).
\end{proof}

We now aim to show that if $\K$ is nicely $[\lambda, \lambda^{+n}]$-stable and $\K_\lambda$ has a superlimit, then $\Kslp{\lambda}$ (the class generated by the superlimit) is also nicely $[\lambda, \lambda^{+n}]$-stable. The $n = 0$ case is given by what has just been proven, but for $n > 0$ we will need to work a little bit more.

For $\K$ a nicely $\lambda$-stable AEC, we call $M \in \K$ \emph{$(\lambda, \lambda^+)$-limit} if there exists an increasing continuous chain $\seq{M_i : i \le \lambda^+}$ such that $M = M_{\lambda^+}$ and for all $i < \lambda^+$, $M_i \in \K_\lambda$ and $M_{i + 1}$ is universal over $M_i$. By the proof of \cite[Proposition 14]{vandieren-chainsat-apal} (a standard back and forth argument), we have:

\begin{fact}\label{bf-lambdap}
  In a nicely $[\lambda, \lambda^+]$-stable AEC, a model is $(\lambda, \lambda^+)$-limit if and only if it is $(\lambda^+, \lambda^+)$-limit.
\end{fact}

\begin{thm}\label{nicely-stable-sl}
  Let $\K$ be an AEC, let $\lambda \ge \LS (\K)$, and let $n < \omega$. If $\K$ is nicely $[\lambda, \lambda^{+n}]$-stable and $\K_\lambda$ has a superlimit, then:

  \begin{enumerate}
  \item Any $(\lambda^{+n}, \lambda^{+(n + 1)})$-limit is in $\Kslp{\lambda}$. In particular for every $m \le n$, $\Kslp{\lambda}_{\lambda^{+m}}$ is dense in $\K_{\lambda^{+m}}$. 
  \item $\Kslp{\lambda}$ is nicely $[\lambda, \lambda^{+n}]$-stable.
  \end{enumerate}
\end{thm}
\begin{proof} \
  \begin{enumerate}
  \item By induction on $n$. If $n = 0$, note that by Lemma \ref{sl-lim}, the superlimit in $\K_\lambda$ is $(\lambda, \lambda)$-limit. Moreover, it is easy to check that the $(\lambda, \lambda^+)$-limit is a union of an increasing chains of $(\lambda, \lambda)$-limits, hence of members of $\Kslp{\lambda}$. Thus the $(\lambda, \lambda^+)$-limit is in $\Kslp{\lambda}$. If $n > 0$, then similarly the $(\lambda^{+n}, \lambda^{+(n + 1)})$-limit is a union of $(\lambda^{+n}, \lambda^{+n})$-limits. By Fact \ref{bf-lambdap}, these are $(\lambda^{+(n - 1)}, \lambda^{+n})$-limits. By the induction hypothesis, these are all in $\Kslp{\lambda}$, hence the $(\lambda^{+n}, \lambda^{+(n + 1)})$-limit is in $\Kslp{\lambda}$. The ``in particular'' part follows: if $m = 0$, this is Remark \ref{sl-dense}, and if $m > 0$ then by definition of nice stability, any model in $\K_{\lambda^{+m}}$ is contained in an $(\lambda^{+m}, \lambda^{+m})$-limit, hence (by Fact \ref{bf-lambdap}) in an $(\lambda^{+(m - 1)}, \lambda^{+m})$-limit, which is in $\Kslp{\lambda}$.
  \item By the previous part and Theorem \ref{dense-nicely-stable}.
  \end{enumerate}
\end{proof}

We will often be interested in situations where the AEC is not just nicely stable but nicely superstable, in the sense that it also has a superlimit and the class generated by the superlimit is superstable. 

\begin{defin}\label{nice-ss-def}
  An AEC $\K$ is \emph{nicely sl-$\lambda$-superstable} (or \emph{nicely sl-superstable in $\lambda$} -- the ``sl'' stands for ``superlimit'') if:

  \begin{enumerate}
  \item $\K$ is nicely $\lambda$-stable.
  \item $\K_\lambda$ has a superlimit.
  \item $\Kslp{\lambda}$ is $\lambda$-superstable and has $\lambda$-symmetry.
  \end{enumerate}

  For $\Theta$ a class of cardinals, we say that $\K$ is \emph{nicely sl-$\Theta$-superstable} (or \emph{nicely sl-superstable in $\Theta$}) if $\K$ is nicely sl-$\lambda$-superstable for every $\lambda \in \Theta$.
\end{defin}

We caution the reader: nice sl-superstability does not immediately imply superstability, since amalgamation is not assumed. However the following are immediate from the definitions:

\begin{remark}
  Let $\K$ be an AEC.
  
  \begin{enumerate}
  \item If $\K$ is $\lambda$-superstable, has $\lambda$-symmetry, and $\K_\lambda$ has a superlimit, then $\K$ is nicely sl-$\lambda$-superstable.

  \item If $\K$ is nicely sl-$\lambda$-superstable, then $\Kslp{\lambda}$ is $\lambda$-superstable, has $\lambda$-symmetry, and is categorical in $\lambda$.
  \end{enumerate}
\end{remark}

\section{AECs with arbitrarily large models}\label{arb-large-sec}

An AEC $\K$ has \emph{arbitrarily large models} if $\K_{\ge \chi} \neq \emptyset$ for all cardinals $\chi$. Equivalently, $\K$ is a large category (in the sense that it has a proper class of nonisomorphic objects). In this section, we recall some basic facts about AECs with arbitrarily large models. We study in particular a weakening of categoricity called \emph{semisolvability}, equivalent in the first-order case to superstability. 

The following is a sufficient condition for an AEC to have arbitrarily large models: 

\begin{fact}[{\cite[I.1.11]{shelahaecbook}}]\label{arb-large-fact}
  Let $\K$ be an AEC. If $\K_{\ge \chi} \neq \emptyset$ for all $\chi < \hanf{\LS (\K)}$, then $\K$ has arbitrarily large models.
\end{fact}

Clearly, if an AEC has no maximal models (and is not empty), then it has arbitrarily large models. We will make use of the following weakening of having no maximal models. Such notions are implicit already in both \cite[Chapter IV]{shelahaecbook} and \cite{sh893-lwb}.

\begin{defin}
  Let $\K$ be an AEC and let $\lambda < \mu$ be infinite cardinals. We call $\K$ \emph{$(\lambda, \mu)$-extendible} if for any $M \in \K_\lambda$ there exists $N \in \K_{\ge \mu}$ such that $M \lea N$. We say that $\K$ is \emph{$\lambda$-extendible} if $\K$ is $(\lambda, \mu)$-extendible for all $\mu$ and we say that $\K$ is \emph{extendible} if $\K$ is $\lambda$-extendible for all $\lambda$. More generally, for $\Theta$ a class of cardinals, we say that $\K$ is \emph{$(\Theta, \mu)$-extendible} if $\K$ is $(\lambda, \mu)$-extendible for all $\lambda \in \Theta$, and we say that $\K$ is \emph{$\Theta$}-extendible if $\K$ is $\lambda$-extendible for all $\lambda \in \Theta$.
\end{defin}

The following basic properties of extendibility may sometimes be used without mention.

\begin{remark}
  Let $\K$ be an AEC.

  \begin{enumerate}
  \item If $\mu \le \lambda$, then $\K$ is always $(\mu, \lambda)$-extendible.    
  \item Let $\lambda$ be a cardinal and let $\mu_0 \le \mu_1$. If $\K$ is $(\lambda, \mu_1)$-extendible, then $\K$ is $(\lambda, \mu_0)$-extendible.
  \item If $\LS (\K) \le \lambda_0 \le \lambda_1 \le \lambda_2$, $\K$ is $(\lambda_0, \lambda_1)$-extendible and $(\lambda_1, \lambda_2)$-extendible, then $\K$ is $(\lambda_0, \lambda_2)$-extendible.
  \item If $\K$ is $\lambda$-extendible for some $\lambda$ with $\K_\lambda \neq \emptyset$, then $\K$ has arbitrarily large models.
  \item For $\LS (\K) \le \lambda < \mu$, $\K$ is $([\lambda, \mu), \mu)$-extendible if and only if $\K_{[\lambda, \mu)}$ has no maximal models. In particular, $\K_{\ge \LS (\K)}$ is extendible if and only if $\K_{\ge \LS (\K)}$ has no maximal models.

  \end{enumerate}
\end{remark}

We recall Shelah's definition of solvability \cite[Definition IV.1.4]{shelahaecbook}, using a more convenient notation for it with only one cardinal parameter, introduced in \cite{categ-saturated-afml}. We also introduce a weakening, semisolvability in $\lambda$, which only asks for the EM model generated by $\lambda$ to be universal. This has the same name, but is weaker than, the notion introduced in \cite[3.1]{categ-saturated-afml}. Nevertheless, all the proofs there go through with the weaker notion. Both solvability and semisolvability are equivalent to superstability in the first-order case (see \cite{shvi-notes-apal} and \cite[5.3]{gv-superstability-jsl}). Shelah writes that solvability is perhaps the true analog of superstability in abstract elementary classes \cite[N\S4(B)]{shelahaecbook}.

\begin{defin}\label{solv-def}
  Let $\K$ be an AEC and let $\mu \ge \LS (\K)$.
  \begin{enumerate}
  \item Let $\Upsilon[\K]$ denote the set of Ehrenfeucht-Mostowski (EM) blueprints $\Phi$ with $|\tau (\Phi)| \le \LS (\K)$. See \cite[IV.0.8]{shelahaecbook} for the full definition. As is standard, for a linear order $I$ we write $\EM (I, \Phi)$ for the EM model generated by $I$ and $\Phi$, and $\EM_\tau (I, \Phi)$ for its reduct to $\tau$.
  \item \cite[IV.1.4.(1)]{shelahaecbook} We say that \emph{$\Phi$ witnesses $\mu$-solvability} if:
      \begin{enumerate}
        \item $\Phi \in \Upsilon[\K]$.
        \item If $I$ is a linear order of size $\mu$, then $\EM_{\tau (\K)} (I, \Phi)$ is superlimit in $\K_{\mu}$ (recall Definition \ref{sl-def}).
      \end{enumerate}
    \item \emph{$\Phi$ witnesses $\mu$-semisolvability} if:
      \begin{enumerate}
      \item $\Phi \in \Upsilon[\K]$
      \item $\EM_{\tau (\K)} (\mu, \Phi)$ is universal in $\K_{\mu}$.
      \end{enumerate}
    \item $\K$ is \emph{$\mu$-[semi]solvable} if there exists $\Phi$ witnessing $\mu$-[semi]solvability.
    \item For a class $\Theta$ of cardinals, $\K$ is \emph{$\Theta$-[semi]solvable]} if $\K$ is $\mu$-[semi]solvable for every $\mu \in \Theta$.
  \end{enumerate}
\end{defin}

Directly from the definitions, we have:

\begin{remark}\label{solv-rmk} Let $\K$ be an AEC and let $\mu \ge \LS (\K)$.
  \begin{enumerate}
  \item If $\K$ has arbitrarily large models and is categorical in $\mu$, then $\K$ is $\mu$-solvable.
  \item If $\K$ is $\mu$-solvable, then $\K$ is $\mu$-semisolvable.
  \item If $\K$ is $\mu$-semisolvable, then $\K$ is $\mu$-extendible. In particular, $\K$ has arbitrarily large models.
  \item If $\K$ is $\mu$-[semi]solvable, then for any $\lambda \in [\LS (\K), \mu]$, $\K_{\ge \lambda}$ is $\mu$-[semi]solvable (enlarge the blueprint of size $\LS (\K)$ witnessing [semi]solvability to a blueprint of size $\lambda$).
  \end{enumerate}
\end{remark}

In \cite[5.1]{categ-saturated-afml}, it was shown that solvability transfers down in AECs with amalgamation and no maximal models (this is of interest, since we do not know whether \emph{categoricity} itself transfers down). The reason for considering semisolvability here is that we can transfer it down using just no maximal models (and in fact extendibility in the relevant cardinals suffices). A very similar argument appears already in \cite{shvi635}.

\begin{fact}\label{solv-downward}
  Let $\K$ be an AEC and let $\mu > \lambda \ge \LS (\K)$. Assume that $\K$ is $\mu$-semisolvable. The following are equivalent:

  \begin{enumerate}
  \item $\K$ is $(\lambda, \mu)$-extendible.
  \item $\K$ is $\lambda$-semisolvable.
  \end{enumerate}
\end{fact}
\begin{proof}
  If $\K$ is $\lambda$-semisolvable, then $\K$ is $\lambda$-extendible, hence $(\lambda, \mu)$-extendible. Conversely, assume that $\K$ is $(\lambda, \mu)$-extendible and suppose that $\Phi$ is an EM blueprint witnessing $\mu$-semisolvability. Since $\K$ is $(\lambda, \mu)$-extendible, any $M \in \K_\lambda$ embeds inside $\EM_{\tau (\K)} (\mu, \Phi)$, hence inside $\EM_{\tau (\K)} (A, \Phi)$, for some $A \subseteq \mu$ with $|A| = \lambda$. In particular, $\lambda \le \otp (A) < \lambda^+$, so by renaming, $M$ also embeds inside $\EM_{\tau (\K)} (\otp (A), \Phi)$. We have shown that any model in $\K_\lambda$ embeds inside $\EM_{\tau (\K)} (\alpha, \Phi)$ for some $\alpha \in [\lambda, \lambda^+)$. Now by (for example) \cite[15.5]{baldwinbook09}, the linear order $I := \fct{<\omega}{\lambda}$ (ordered lexicographically) embeds any ordinal $\alpha < \lambda^+$, thus $\EM_{\tau (\K)} (I, \Phi)$ is universal in $\K_\lambda$. Expanding the vocabulary, build a new blueprint $\Psi$ such that for any linear order $J$, $\EM_{\tau (\K)} (J, \Psi) = \EM_{\tau (\K)} (\fct{<\omega}{J}, \Phi)$. Then by construction $\Psi$ witnesses $\lambda$-semisolvability.
\end{proof}

We end this section by stating some known consequences of semisolvability that we will use. Recall (Section \ref{set-thy-sec}) that $\GCHWD (\lambda)$ means that $2^\lambda = \lambda^+$ and enough instances of the weak diamond hold at $\lambda$.

\begin{fact}\label{semisolv-fact}
  Let $\K$ be an AEC and let $\mu > \lambda \ge \LS (\K)$. Assume that $\K$ is $\mu$-semisolvable and $(\lambda, \mu)$-extendible. 

  \begin{enumerate}
  \item\label{semisolv-1} If $\GCHWD (\lambda)$, then $\K$ is nicely $\lambda$-stable.
  \item\label{semisolv-2} If $\K_\lambda$ has amalgamation, then $\K$ is $\lambda$-superstable and has $\lambda$-symmetry.
  \item\label{semisolv-3} If $\K_\lambda$ has a superlimit which is an amalgamation base, then $\K$ is nicely sl-$\lambda$-superstable (recall Definition \ref{nice-ss-def}).
  \item\label{semisolv-4} If $\K_\lambda$ has a superlimit and $\WGCH (\lambda)$, then $\K$ is nicely sl-$\lambda$-superstable.
  \end{enumerate}
\end{fact}
\begin{proof} By Fact \ref{solv-downward}, $\K$ is $[\lambda, \lambda^+]$-semisolvable. In particular, $\K_{\lambda}$ has joint embedding and no maximal models. Of course, it is also not empty. Now:
  
  \begin{enumerate}
  \item By \cite[\S1]{shvi635}. See also \cite[\S3]{vandierennomax}.
  \item By \cite{shvi-notes-apal}, $\K$ is $\lambda$-superstable. By \cite[4.8]{categ-saturated-afml}, $\K$ has $\lambda$-symmetry.
  \item Note that the usual argument of Morley (see e.g.\ \cite[8.20]{baldwinbook09} or Claim 1 in the proof of \cite[3.4]{categ-saturated-afml}) shows that $\Kslp{\lambda}$ (see Definition \ref{ksl-def}) is stable in $\lambda$: if $M$ is superlimit in $\K_\lambda$, then $|\gS (M)| \le \lambda$. Since $\Kslp{\lambda}_\lambda$ has amalgamation, it is nicely $\lambda$-stable (see Remark \ref{nice-stab-ap}). By (for example) Theorem \ref{dense-nicely-stable}, This implies that $\K$ is nicely $\lambda$-stable. The result now follows from \cite{shvi-notes-apal}.
  \item By Fact \ref{ap-univ-comb} and the previous part.
  \end{enumerate}
\end{proof}

\section{Towers and disjoint amalgamation}\label{tower-sec}

In this section, we prove some technical lemmas relative to how much models can be amalgamated to be ``as nonforking as possible''. This relies on the theory of towers. Towers were introduced in \cite{shvi635}, and further studied in several papers since then (e.g.\ \cite{vandierennomax, nomaxerrata, gvv-mlq, vandieren-symmetry-apal, vv-symmetry-transfer-afml}). Essentially, everything before Lemma \ref{tower-ext-cont} is known and appears in some form in either \cite{gvv-mlq} or \cite{vandieren-symmetry-apal}. We give some proofs and definitions here both because the statements are slightly different and the arguments have been simplified (\cite{gvv-mlq} and \cite{vandieren-symmetry-apal} had to work without the uniqueness property of nonforking, proven in \cite{uq-forking-mlq}; thus the definition of tower there is more complicated).

Throughout this section, we assume:

\begin{hypothesis}\label{hyp-5}
  $\K$ is a $\lambda$-superstable AEC with $\lambda$-symmetry.
\end{hypothesis}

We work inside $\K_\lambda$: except if said otherwise, all models come from there. The following consequence of symmetry will be crucial. The idea is that we can make sure that two elements are independent ``in a uniform way''.

\begin{fact}[Nonforking amalgamation]\label{nf-amalgam}
  Let $M_0 \lea M_\ell$, $\ell = 1,2$, be limit models. Let $a_\ell \in M_\ell$. There exists $f_1, f_2, M_3$ such that $M_3$ is limit and $f_\ell : M_\ell \xrightarrow[M_0]{} M_3$ is such that $\gtp (f_\ell (a_\ell) / f_{3 - \ell} (M_{3 - \ell}); M_3)$ does not fork over $M_0$ for $\ell = 1,2$. 
\end{fact}
\begin{proof}
  As in \cite[II.2.16]{shelahaecbook}.
\end{proof}

To define towers, we first introduce some notation:

\begin{notation} \
  \begin{enumerate}
  \item The letter $I$ will denote a well-ordering $(|I|, <_I)$. We usually write $<$ instead of $<_I$.
  \item For $I$ a well-ordering, let $I^-$ be the initial segment of $I$ which is isomorphic to $I$ if $I$ is isomorphic to a limit ordinal or zero, or isomorphic to $\alpha$ if $I$ is isomorphic to $\alpha + 1$.
  \item For $I$ a well-ordering, $i \in I$, and $\alpha$ an ordinal, let $i +_I \alpha$ denote the unique element $j \in I$ (if it exists) such that $\otp (j) = \otp (i) + \alpha$. We write $i + \alpha$ instead of $i +_I \alpha$ when $I$ is clear from context.
  \end{enumerate}
\end{notation}
\begin{defin}
  A \emph{tower} $\Tow$ consists of $\seq{M_i : i \in I} \smallfrown \seq{a_i : i \in I^-}$, where:

  \begin{enumerate}
  \item $I$ is a well-ordering of cardinality at most $\lambda$.
  \item $\seq{M_i : i \in I}$ is an increasing chain of limit models, not necessarily continuous.
  \item $a_i \in M_{i + 1} \backslash M_i$ for each $i \in I^-$.
  \end{enumerate}

  We call $I$ the \emph{length} (or \emph{index set}) of the tower. We call $\Tow$ \emph{continuous} if $\seq{M_i : i \in I}$ is continuous. We say that $\Tow$ is \emph{limit} if $M_{i + 1}$ is limit over $M_i$ for each $i \in I^-$. We may often identify a tower $\Tow$ indexed by $I$ with the tower indexed by the ordinal $\otp (I)$.
\end{defin}
\begin{defin}
  For $\Tow = \seq{M_i : i \in I} \smallfrown \seq{a_i : i \in I^-}$ and $I_0 \subseteq I$, we let $\Tow \rest I_0$ be the sequences $\seq{M_i : i \in I_0} \smallfrown \seq{a_i : i \in I_0^-}$.
\end{defin}
\begin{remark}
  If $\Tow$ is a tower indexed by $I$ and $I_0 \subseteq I$, then $\Tow \rest I_0$ is a tower indexed by $I_0$.
\end{remark}

The reason for indexing towers by a well-ordering instead of just an ordinal is that we will allow towers to be grown by inserting elements not only at the end, but also in the middle. What it means to ``grow'' a tower is given by the following definition:

\begin{defin}[Orderings on towers]
  For $\Tow^\ell = \seq{M_i^\ell : i \in I^\ell} \smallfrown \seq{a_i^\ell : i \in (I^\ell)^-}$, $\ell = 1,2$, two towers, we write $\Tow^1 \tlt \Tow^2$ if:

  \begin{enumerate}
  \item $I^1 \subseteq I^2$.
  \item $M_i^2$ is limit over $M_i^1$ for all $i \in I^1$.
  \item $a_i^1 = a_i^2$ for all $i \in (I^1)^-$.
  \item $\gtp (a_i^1 / M_i^2; M_{i + 1}^2)$ does not fork over $M_i^1$ for all $i \in (I^1)^-$.
  \end{enumerate}

  We write $\Tow^1 \tleq \Tow^2$ if $\Tow^1 = \Tow^2$ or $\Tow^1 \tlt \Tow^2$.
\end{defin}

\begin{remark}
  $\tleq$ is a partial order on the class of all towers. Moreover, if $\Tow^1$ and $\Tow^2$ have index sets $I^1 \subseteq I^2$ respectively, then $\Tow^1 \tleq \Tow^2$ if and only if $\Tow^1 \tleq \Tow^2 \rest I^1$.
\end{remark}

\begin{defin}
  Let $\delta < \lambda^+$ be a limit ordinal and let $\seq{\Tow^j : j < \delta}$ be a $\tleq$-increasing chain of towers. Assume that $\Tow^j = \seq{M_i^j : i \in I^j} \smallfrown \seq{a_i^j : i \in (I^j)^-}$. We define $\Tow^\delta := \seq{M_i^\delta : i \in I^\delta} \smallfrown \seq{a_i^\delta : i \in (I^\delta)^-}$ as follows:

  \begin{enumerate}
  \item $I^\delta = \bigcup_{j < \delta} I^j$.
  \item $a_i^\delta = a_i^j$ for some (any) $j < \delta$ such that $i \in I^j$.
  \item $M_i^\delta = \bigcup_{j < \delta} M_i^j$.
  \end{enumerate}

  We write $\bigcup_{j < \delta} \Tow^j$ for $\Tow^\delta$.
\end{defin}

\begin{remark}
  If $\seq{\Tow^j : j < \delta}$ is a $\tleq$-increasing chain of towers, where $\Tow^j$ is indexed by $I^j$, and $\bigcup_{j < \delta} I^j$ is a well-ordering, then $\bigcup_{j < \delta} \Tow^j$ is a tower and $\Tow^k \tleq \bigcup_{j < \delta} \Tow^j$ for every $k < \delta$.
\end{remark}

\begin{defin}
  If a $\tleq$-increasing chain of towers $\seq{\Tow^j : j < \gamma}$ is such that for every limit $j < \gamma$, $\Tow^j = \bigcup_{k < j} \Tow^k$, we call the chain \emph{continuous}.
\end{defin}

We will often want to extend a given tower to another tower that is both limit and continuous. It is easy to extend a given tower to a limit tower, not requiring continuity (Fact \ref{tower-ext-lem}). But intuitively the continuity of a tower conflicts with the limit requirement imposed on the order relation $\tlt$ between towers, so obtaining both of them simultaneously is a challenge. It turns out (Fact \ref{reduced-density}) that it is easier to extend a given tower to a tower with a stronger property than continuity, being \emph{reduced}, which we now proceed to define (it says that any extension must be ``as disjoint as possible''). Intuitively, being reduced and being limit seem to be conflicting properties for a tower to satisfy. But being reduced and satisfying a property that is similar to being limit (fullness, see Definition \ref{full-def}) are properties that are both preserved under unions of towers (see Lemma \ref{full-union}). Combining these two properties, we will be able to build the desired continuous limit extension (Lemma \ref{tower-ext-cont}).

\begin{defin}[{\cite[3.1.11(1)]{shvi635}}]
  A tower $\Tow = \seq{M_i : i < \alpha} \smallfrown \seq{a_i : i + 1 < \alpha}$ is called \emph{reduced} if whenever $\Tow' = \seq{M_i' : i < \alpha} \smallfrown \seq{a_i : i + 1 < \alpha}$ is such that $\Tow \tleq \Tow'$, we have that $M_i' \cap M_j = M_i$ for any $i \le j < \alpha$.
\end{defin}
\begin{remark}\label{reduced-cont-ex}
  If $\seq{\Tow^j : j < \delta}$ is a chain of towers and $\Tow^j$ is reduced for all $j < \delta$, then $\bigcup_{j < \delta} \Tow^j$ is reduced (provided that its index is a well-ordering).
\end{remark}

The following related definition appears (stated differently) in (for example) \cite[3.3.2]{jrsh875}:

\begin{defin}
  We say a triple $(a, M, N)$ is \emph{reduced} if $(M, N) \smallfrown (a)$ is a reduced tower. 
\end{defin}

To check that a tower is reduced, it is enough to check its restrictions of length two:

\begin{lem}\label{red-tower-triple}
  Let $\Tow = \seq{M_i : i < \alpha} \smallfrown \seq{a_i : i + 1 < \alpha}$ be a continuous tower. If $(a_i, M_i, M_{i + 1})$ is a reduced triple for any $i + 1 < \alpha$, then $\Tow$ is reduced.
\end{lem}
\begin{proof}
  Let $\Tow' = \seq{M_i' : i < \alpha} \smallfrown \seq{a_i : i + 1 < \alpha}$ be such that $\Tow \tleq \Tow'$. Let $i \le j < \alpha$. We have to see that $M_i' \cap M_j = M_i$. We proceed by induction on $j$. If $j = i$, this is obvious. If $j$ is limit, this is immediate from the induction hypothesis and (since $\Tow$ is assumed to be continuous) $M_j = \bigcup_{k < j} M_k$. Assume now that $j = k + 1$. By assumption, $(a_k, M_k, M_{k + 1})$ is a reduced triple. Thus $M_i' \cap M_j \subseteq M_k' \cap M_j = M_k$, and so $M_i' \cap M_j = M_i' \cap M_k$ which by the induction hypothesis is just $M_i$, as desired.
\end{proof}

Reduced towers exist: any tower has a reduced extension.

\begin{fact}[Density of reduced towers]\label{reduced-density}
  For any tower $\Tow$ of length $\alpha$, there exists a reduced tower $\Tow'$ of length $\alpha$ such that $\Tow \tleq \Tow'$.
\end{fact}
\begin{proof}
  As in (for example) \cite[5.5]{gvv-mlq}.
\end{proof}

So far, we have not shown that towers have any nontrivial $\tleq$-extensions. In fact, nonforking amalgamation gives a crucially stronger statement. 

\begin{fact}[Existence of extensions of towers]\label{tower-ext-lem}
  Let $\Tow = \seq{M_i : i < \alpha} \smallfrown \seq{a_i : i + 1 < \alpha}$ be a tower. 

  \begin{enumerate}
  \item There exists a limit tower $\Tow'$ of length $\alpha$ such that $\Tow \tlt \Tow'$.
  \item\label{setup-2} Assume in addition that $\Tow$ is limit and continuous. If $q \in \gS (M_0)$, then there exists a limit tower $\Tow' = \seq{M_i' : i < \alpha} \smallfrown \seq{a_i : i + 1 < \alpha}$ and $b \in M_0'$ such that:
    \begin{enumerate}
    \item $\Tow \tlt \Tow'$.
    \item $\gtp (b / M_0; M_0') = q$.
    \item $\gtp (b / \bigcup_{i < \alpha} M_i; \bigcup_{i < \alpha} M_i')$ does not fork over $M_0$.
    \end{enumerate}
  \end{enumerate}
\end{fact}
\begin{proof}
  As in \cite[II.4.9]{shelahaecbook} or \cite[3.1.8]{jrsh875}. The first part does not use symmetry: we simply use the extension property repeatedly (we first carry out the construction without insisting that the morphisms all be inclusions, then rename). The second part uses symmetry, more specifically nonforking amalgamation (Fact \ref{nf-amalgam}), to deal with $b$ at the same time. We use universal continuity of forking (see Fact \ref{forking-fact}) at limit steps, and for this we need the original tower to be continuous and limit (or only that $M_{\alpha + 1}$ is universal over $M_\alpha$ for all $\alpha$, but for our purpose this weakening does not change much). 
\end{proof}

We emphasize again that the extension $\Tow'$ obtained from Fact \ref{tower-ext-lem} is not necessarily continuous: the definition of tower extension requires that $M_{\delta}'$ should be limit over $M_\delta$, even when $\delta$ is a limit ordinal. We naively fulfill this requirement by breaking continuity. One point of the rest of this section is to develop more elaborate tools to recover it.

The following technical consequence will be used in the proof of the next theorem.

\begin{lem}\label{tower-ext-technical}
  Let $\delta < \lambda^+$ be a limit ordinal. Let $\Tow = \seq{M_i : i \le \delta} \smallfrown \seq{a_i : i < \delta}$ be a tower such that $\Tow \rest \delta$ is limit and continuous. Let $b \in M_\delta$. If $\gtp (b / \bigcup_{i < \delta} M_i; M_\delta)$ does not fork over $M_0$, then there exists a tower $\Tow' = \seq{M_i' : i \le \delta} \smallfrown \seq{a_i : i < \delta}$ such that $\Tow \tlt \Tow'$ and $b \in M_0'$.
\end{lem}
\begin{proof}
  Write $M_\delta^0 := \bigcup_{i < \delta} M_i$. Let $q := \gtp (b / M_0; M_\delta)$. By Fact \ref{tower-ext-lem} applied to the tower $\Tow \rest \delta$, there exists $\Tow^\ast = \seq{M_i^\ast : i < \delta} \smallfrown \seq{a_i : i < \delta}$ and $b^\ast \in M_0^\ast$ such that $\Tow \rest \delta \tlt \Tow^\ast$, $\gtp (b^\ast / M_0; M_0^\ast) = q$, and $\gtp (b^\ast / M_\delta^0; M_\delta^\ast)$ does not fork over $M_0$ (we have set $M_\delta^\ast := \bigcup_{i < \delta} M_i^\ast$). Note that we have used that $\delta$ is a limit ordinal to make sure that all the $a_i$'s are still in $\Tow \rest \delta$. By uniqueness, $\gtp (b^\ast / M_\delta^0; M_\delta^\ast) = \gtp (b / M_\delta^0; M_\delta)$. Pick $M_\delta'$ limit over $M_\delta$ and $f: M_\delta^\ast \xrightarrow[M_\delta^0]{} M_\delta'$ such that $f (b^\ast) = b$. Let $M_i' := f[M_i^\ast]$ for $i < \delta$.
\end{proof}

We obtain the following powerful tool to build continuous towers. This was first proven by VanDieren (using a slightly different notion of tower) \cite{vandieren-symmetry-apal}. We give a simplification of VanDieren's original proof here.

\begin{fact}\label{reduced-continuous}
  Any reduced tower is continuous.
\end{fact}
\begin{proof}
  Suppose not. Let $\alpha$ be the least length of a reduced non-continuous tower. Then it is easy to see that $\alpha = \delta + 1$, where $\delta$ is a limit ordinal. Let $\Tow = \seq{M_i : i \le \delta} \smallfrown \seq{a_i : i < \delta}$ be such a reduced non-continuous tower. Thus $\bigcup_{i < \delta} M_i \neq M_\delta$. Pick $b \in M_\delta \backslash \bigcup_{i < \delta} M_i$.

  \underline{Claim}: There is \emph{no} $k < \delta$ and no tower $\Tow' = \seq{M_i' : i \in [k, \delta]} \smallfrown \seq{a_i : i \in [k, \delta]}$ such that $\Tow \rest [k, \delta] \tlt \Tow'$ and $b \in \bigcup_{i \in [k, \delta)} M_i'$.

  \underline{Proof of Claim}: Suppose $\Tow'$ is such a tower and fix $i < \delta$ such that $b \in M_i'$. Then $b \in M_i' \cap M_\delta$ but $b \notin M_i$, so $M_i' \cap M_\delta \neq M_i$. Moreover, one can extend $\Tow'$ to a tower $\Tow''$ of length $\delta + 1$ so that $\Tow'' \rest [k, \delta) = \Tow'$ and $\Tow \tleq \Tow''$ (use universality of $M_k'$ over $M_k$). This implies that $\Tow$ is not reduced, contradiction. $\dagger_{\text{Claim}}$

    We aim to build a tower as in the claim to get a contradiction. We will use Lemma \ref{tower-ext-technical}, but for this we need to start with a tower $\Tow^{\ast \ast}$ extending $\Tow$ that is both limit \emph{and} continuous before $\delta$. This motivates the following construction of a ``diagonal tower'': build $\seq{\Tow^j : j \le \delta}$ an $\tlt$-increasing continuous chain of reduced towers such that $\Tow^0 = \Tow$. Write $\Tow^j = \seq{M_i^j : i \le \delta} \smallfrown \seq{a_i : i < \delta}$. Now consider the diagonal tower $\Tow^\ast := \seq{M_i^i : i \le \delta} \smallfrown \seq{a_i : i < \delta}$.

    It is easy to check that $\Tow^\ast$ is indeed a tower, and further it is limit. Since $\delta$ was minimal, $\Tow^j \rest \delta$ is continuous for all $j \le \delta$, and hence $\Tow^\ast \rest \delta$ is also continuous. By universal local character (which we can use \emph{precisely because} $\Tow^\ast$, hence $\Tow^\ast \rest \delta$, is limit), there exists $i < \delta$ such that $\gtp (b / \bigcup_{k < \delta} M_k^k; M_{\delta}^\delta)$ does not fork over $M_i^i$. Let $\Tow^{\ast \ast} := \seq{M_i^i : i \in [k + 1, \delta]} \smallfrown \seq{a_i : i \in [k + 1, \delta]}$. By Lemma \ref{tower-ext-technical}, where $\Tow$ there stands for $\Tow^{\ast \ast}$ here, there exists a tower $\Tow'$ such that $\Tow^{\ast \ast} \tlt \Tow'$ and $\Tow' \rest \delta$ contains $b$. Since $\Tow'$ also extends $\Tow \rest [k, \delta]$, this contradicts the claim. 
\end{proof}

We now want to give conditions under which a tower $\seq{M_i : i < \alpha} \smallfrown \seq{a_i : i + 1 < \alpha}$ is such that $\bigcup_{i < \alpha} M_i$ is limit over $M_0$. Of course, being a limit tower suffices but it is not clear whether this property is closed under unions. Instead, we will rely on the following weakening (a variation appears in \cite[4.3]{gvv-mlq}):

\begin{defin}\label{full-def}
  A tower $\Tow = \seq{M_i : i \in I} \smallfrown \seq{a_i : i \in I^-}$ is \emph{$I_0$-full} if:

  \begin{enumerate}
  \item $I_0 \subseteq I$.
  \item For any $i \in I_0^-$ and any $p \in \gSna (M_i)$, there exists $k \in [i, i +_{I_0} 1)_{I}$ such that $\gtp (a_k / M_k; M_{k +_{I} 1})$ is the nonforking extension of $p$.
  \end{enumerate}
\end{defin}

We have the following monotonicity properties:

\begin{remark}\label{full-monot}
  Let $I_0 \subseteq I_1 \subseteq I_2$ be well-orderings and let $\Tow$ be a tower indexed by $I_2$.

  \begin{enumerate}
  \item If $\Tow \rest I_1$ is $I_0$-full, then $\Tow$ is $I_0$-full.
  \item If $\Tow$ is $I_1$-full, then $\Tow$ is $I_0$-full.
  \end{enumerate}
\end{remark}

Intuitively, full towers are those for which the $a_i$'s realize all the types many times. To see that full towers generate limit models, we will use:

\begin{fact}[{\cite[II.1.16(4)]{shelahaecbook}}]\label{univ-constr-lem}
  If $\seq{M_i : i \le \lambda}$ is increasing continuous in $\K_\lambda$ such that $M_{i + 1}$ realizes all types over $M_i$ for every $i < \lambda$, then $M_\lambda$ is universal over $M_0$.
\end{fact}

\begin{lem}\label{full-limit}
  If $\Tow = \seq{M_i : i \in I} \smallfrown \seq{a_i : i \in I^-}$ is an $I_0$-full tower and $\lambda \cdot (1 + \otp (I_0)) = \otp (I_0)$, then $\bigcup_{i \in I_0} M_i$ is $(\lambda, \cf{I_0})$-limit over $M_0$.
\end{lem}
\begin{proof}
  Fix $i \in I_0$. By assumption, $M_{i +_{I_0} 1}$ realizes all types over $M_i$. Thus by Fact \ref{univ-constr-lem} $M_{i +_{I_0} \lambda}$ is universal over $M_i$. Since $i$ was arbitrary, we obtain that $M_0 \smallfrown \seq{M_{i} : i \in I_0}$ is the desired witness that $\bigcup_{i \in I_0} M_i$ is $(\lambda, \cf{I_0})$-limit over $M_0$.
\end{proof}

Full towers are also preserved by unions:

\begin{lem}\label{full-union}
  Let $\delta < \lambda^+$ be a limit ordinal. Let $\seq{\Tow^j : j < \delta}$ be an increasing chain of towers. Assume that $\Tow^j$ is indexed by $I^j$, and $I^\delta := \bigcup_{j < \delta} I^j$ is a well-ordering. Suppose that $I_0 \subseteq I^0$ is such that $\Tow^j$ is $I_0$-full for each $j < \delta$. Then $\bigcup_{j < \delta} \Tow^j$ is $I_0$-full.
\end{lem}
\begin{proof}
  Let $\Tow^\delta := \bigcup_{j < \delta} \Tow^j$. For $j \le \delta$, write $\Tow^j = \seq{M_i^j : i \in I^j} \smallfrown \seq{a_i : i \in (I^j)^-}$. Let $i \in I_0^-$ and let $p \in \gSna (M_i^\delta)$. Pick $j < \delta$ such that $p$ does not fork over $M_i^j$. Since $\Tow^{j}$ is $I_0$-full, there exists $k \in [i, i +_{I_0} 1)_{I^j}$ such that $\gtp (a_k / M_k^j; M_{k +_{I^j} 1}^j)$ is the nonforking extension of $p \rest M_i^j$. Clearly, $k \in [i, i +_{I_0} 1)_{I^\delta}$. Moreover, by definition of extension of tower, $\gtp (a_k / M_k^\delta; M_{k + 1}^\delta)$ does not fork over $M_k^j$. By transitivity of forking and uniqueness, $\gtp (a_k / M_k^\delta; M_{k + 1}^\delta)$ is the nonforking extension of $p \rest M_i^j$, hence of $p$. 
\end{proof}

\begin{defin}
  For $I$ and $J$ linear orders, let $I \times J$ be the usual cartesian product, ordered lexicographically: $(i_1, j_1) < (i_2, j_2)$ if and only if either $i_1 < i_2$, or $i_1 = i_2$ and $j_1 < j_2$. 
\end{defin}

We now prove two construction lemmas about full towers. The idea is also described on p.~373 of \cite{gvv-mlq}. First, in order to be able to ensure that towers have full extensions, they need to have enough space. This is the notion of a strong limit tower:

\begin{defin}
  A tower $\seq{M_i: i < \alpha} \smallfrown \seq{a_i : i + 1 < \alpha}$ is \emph{strongly limit} if for any $i \in (0, \alpha)$, $M_i$ is limit over $\bigcup_{j < i} M_j$.
\end{defin}
\begin{remark}\label{str-lim-ext}
  The proof of Fact \ref{tower-ext-lem} shows that any tower $\tlt$-extends to a strongly limit tower of the same length.
\end{remark}

The difference between this definition and that of a limit tower is that if $i$ is a limit ordinal, then $M_i$ is also required to be limit over the union of its predecessors (so in particular the tower is \emph{not} continuous). 

\begin{lem}[First construction lemma]\label{full-constr-1}
  Let $I$ be a well-ordering and let $\alpha < \gamma$ be ordinals in $[1, \lambda^+)$ with $\gamma \cdot \lambda = \gamma$. Let $\Tow$ be a strongly limit tower indexed by $I \times \alpha$. Then there is a strongly limit tower $\Tow'$ indexed by $I \times \gamma$ such that $\Tow'$ is $(I \times \{0\})$-full and $\Tow' \rest (I \times \alpha) = \Tow$.
\end{lem}
\begin{proof}
  Straightforward: realize all the relevant types.
\end{proof}

\begin{lem}[Second construction lemma]\label{full-constr-2}
  Let $I$ be a well-ordering and let $\alpha \in [1, \lambda^+)$ be an ordinal. Let $\Tow$ be a tower indexed by $I \times \alpha$. Then there is an ordinal $\beta \in [\alpha, \lambda^+)$ and a tower $\Tow'$ indexed by $I \times \beta$ such that $\Tow \tlt \Tow'$, $\Tow'$ is reduced, and $\Tow'$ is $(I \times \{0\})$-full.
\end{lem}
\begin{proof}
  We build a $\tleq$-increasing continuous chain of towers $\seq{\Tow^i : i \le \omega}$ and an increasing continuous chain of ordinals $\seq{\alpha_i : i \le \omega}$ such that:

  \begin{enumerate}
  \item $\Tow^i$ is indexed by $I \times \alpha_i$.
  \item $\Tow \tlt \Tow^0$.
  \item For $i < \omega$ nonzero and even, $\Tow^i$ is reduced.
  \item For $i < \omega$ odd, $\Tow^i$ is strongly limit and $I$-full.
  \end{enumerate}

  This is enough: $\Tow^\omega$ is reduced by Remark \ref{reduced-cont-ex} and $I$-full by Lemma \ref{full-union}. Thus one can set $\Tow' := \Tow^\omega$, $\beta := \alpha_\omega$.

  This is possible: take $\alpha_0 := \alpha$ and any $\Tow^0$ such that $\Tow \tlt \Tow^0$ (exists by Fact \ref{tower-ext-lem}). Now let $i > 0$ and assume that $\Tow^{i - 1}$ and $\alpha_{i - 1}$ are given. If $i$ is even, take $\alpha_i := \alpha_{i - 1}$ and let $\Tow^i$ be a reduced tower $\tleq$-extending $\Tow^{i - 1}$ which is indexed by $I \times \alpha_i$ (exists by Fact \ref{reduced-density}). If $i$ is odd, use Remark \ref{str-lim-ext} and Lemma \ref{full-constr-1}.
\end{proof}

In passing, we can now give a proof of Fact \ref{sym-uq-lim}:

\begin{proof}[Proof of Fact \ref{sym-uq-lim}]
  It is enough to show that for any $M \in \K_\lambda$ and any limit ordinals $\delta_1, \delta_2 < \lambda^+$ there exists $N$ which is both $(\lambda, \delta_1)$-limit and $(\lambda, \delta_2)$-limit over $M$. Extending $M$ if necessary, we can assume without loss of generality that $M$ is a limit model. We build a $\tlt$-increasing continuous chain of towers $\seq{\Tow^i : i \le \delta_2}$ such that:

  \begin{enumerate}
  \item The first model in $\Tow^0$ is $M$.
  \item For each $i \le \delta_2$, $\Tow^i$ is indexed by $(\delta_1 + 1) \times \alpha_i$, for some nonzero ordinal $\alpha_i$.
  \item For each nonzero $i \le \delta_2$, $\Tow^i$ is $((\delta_1 + 1) \times \{0\})$-full and reduced.
  \end{enumerate}

  This is possible: use Lemma \ref{full-constr-2} at successor steps, recalling that being full and reduced is preserved at limits (Remark \ref{reduced-cont-ex} and Lemma \ref{full-union}).
  
  This is enough: write $\Tow^i = \seq{M_j^i : j \in (\delta_1 + 1) \times \alpha_i}$. By Remark \ref{full-monot}, $\Tow^{\delta_2}$ is $\delta_1$-full. By Fact \ref{reduced-continuous}, $\Tow^{\delta_2}$ is continuous. Thus by Lemma \ref{full-limit}, $M_{\delta_1}^{\delta_2}$ is $(\lambda, \delta_1)$-limit over $M_0^{\delta_2}$, hence over $M$. Moreover by definition of $\tlt$, $M_{\delta_1}^{\delta_2}$ is also $(\lambda, \delta_1)$-limit over $M_{\delta_1}^0$, hence over $M$, as desired.
\end{proof}

We are now ready to prove that any tower is extended by a continuous limit tower (this could also be shown using the diagonal tower from the proof of Fact \ref{reduced-continuous}).

\begin{lem}\label{tower-ext-cont}
  For any tower $\Tow$ there exists a continuous limit tower $\Tow'$ such that $\Tow \tleq \Tow'$.
\end{lem}
\begin{proof}
  By Fact \ref{tower-ext-lem}, we can assume without loss of generality (taking an extension of $\Tow$ if necessary) that $\Tow$ is a limit tower. Say $\Tow$ is indexed by $I$. Let $\gamma < \lambda^+$ be nonzero such that $\gamma \cdot \lambda = \gamma$. Let $\Tow^\ast$ be any limit tower indexed by $I \times \gamma$ such that $\Tow^\ast \rest I = \Tow$ (it is easy to see that such towers exist). By Lemma \ref{full-constr-2} (where $I$, $\Tow$ there stands for $I \times \gamma$, $\Tow^\ast$ here), there is a reduced and $(I \times \gamma)$-full tower $\Tow^{\ast \ast}$ indexed by $I' := I \times \gamma \times \gamma'$ (for some $\gamma'$) such that $\Tow^\ast \tleq \Tow^{\ast \ast}$. By Fact \ref{reduced-continuous}, $\Tow^{\ast \ast}$ is continuous. Let $\Tow' := \Tow^{\ast \ast} \rest I$. Clearly, $\Tow'$ is also continuous. To see that $\Tow'$ is limit, let $i \in I^-$. Observe that $\Tow^{\ast \ast} \rest [i, i +_I 1]_{I'}$ is $(\{i\} \times \gamma)$-full, hence by Lemma \ref{full-limit}, $M_{i +_I 1}$ is limit over $M_i$, and so $\Tow'$ is indeed a limit tower.
\end{proof}

We can now prove a strengthening of Fact \ref{tower-ext-lem}, where the starting tower is no longer assumed to be continuous and limit.

\begin{lem}\label{tower-ext-step}
  Let $\Tow = \seq{M_i : i < \alpha} \smallfrown \seq{a_i : i + 1 < \alpha}$ be a tower and let $N_0$ be limit over $M_0$. Let $b \in N_0$. There exists a limit continuous tower $\Tow' = \seq{M_i' : i < \alpha} \smallfrown \seq{a_i : i + 1 < \alpha}$ and an isomorphism $f: N_0 \cong_{M_0} M_0'$ such that $\Tow \tleq \Tow'$ and $\gtp (f (b) / \bigcup_{i < \alpha} M_i; \bigcup_{i < \alpha} M_i') $ does not fork over $M_0$.
\end{lem}
\begin{proof}
  First, let $\Tow^1 = \seq{M_i^1 : i < \alpha} \smallfrown \seq{a_i : i + 1 < \alpha}$ be a limit continuous tower extending $\Tow$, as given by Lemma \ref{tower-ext-cont}. Find $N_0'$ limit over $M_0^1$ and $f_0 : N_0 \cong_{M_0} N_0'$ such that $q := \gtp (f_0(b) / M_0^1; N_0')$ does not fork over $M_0$. By Fact \ref{tower-ext-lem}(\ref{setup-2}), there exists a limit continuous tower $\Tow^2 = \seq{M_i^2 : i < \alpha} \smallfrown \seq{a_i : i + 1 < \alpha}$ and a $b' \in M_0^2$ such that $\Tow^1 \tlt \Tow^2$, $b'$ realizes $q$, and $\gtp (b' / \bigcup_{i < \alpha} M_i^1; \bigcup_{i < \alpha} M_i^2)$ does not fork over $M_0^1$. Since $b'$ realizes $q$, transitivity and monotonicity imply that $\gtp (b' / \bigcup_{i < \alpha} M_i; \bigcup_{i < \alpha} M_i^2)$ does not fork over $M_0$. Since $M_0^2$ is limit over $M_0^1$, there exists $g: N_0' \cong_{M_0^1} M_0^2$ such that $g (f_0 (b)) = b'$. Let $\Tow' := \Tow^2$, $f := g f_0$.
\end{proof}

We have arrived to one of the key results of this section. Roughly, it says that any two towers can be amalgamated into a continuous rectangle of models. The conclusion is similar to \cite[3.1.10]{jrsh875} but, crucially, the hypotheses are weaker: the proof given here is very similar to that of Jarden and Shelah, but we do \emph{not} assume the continuity property of good frames. We only have the continuity for universal chains given to us by local character. Thus we have to use the (hard) fact that towers can be extended to continuous limit towers, as well as Lemma \ref{tower-ext-step}.

\begin{lem}[Amalgamation of towers]\label{tower-ap}
  Let $\Tow = \seq{N_j : j < \beta} \smallfrown \seq{b_j : j + 1 < \beta}$ and $\Tow' := \seq{M_i : i < \alpha} \smallfrown \seq{a_i : i + 1 < \alpha}$ be towers. Assume that $N_0 = M_0$. Then there exists $\seq{f_i : i < \alpha}$ and $\seq{M_{i, j} : i < \alpha, j < \beta}$ such that:

  \begin{enumerate}
  \item For all $i < \alpha$ and $j < \beta$, $M_{i, j}$ is a limit model.

  \item For all $j < \beta$, $M_{0, j} = N_j$.
  \item For each $j < \beta$, $\seq{M_{i, j} : i < \alpha}$ is increasing continuous. Moreover, $M_{i + 1, j}$ is limit over $M_{i, j}$ whenever $i + 1 < \alpha$.
  \item For each $i < \alpha$, $\seq{M_{i, j} : j < \beta}$ is increasing continuous.
  \item $\seq{f_i : i < \alpha}$ is increasing.
  \item For all $i < \alpha$, $f_i : M_i \cong_{M_{0, 0}} M_{i, 0}$.

  \item For all $i + 1 < \alpha$ and $j < \beta$, $\gtp (f_{i + 1} (a_i) / M_{i, j}; M_{i + 1, j})$ does not fork over $M_{i, 0}$.
  \item For all $j + 1 < \beta$ and $i < \alpha$, $\gtp (b_j / M_{i, j}; M_{i, j + 1})$ does not fork over $M_{0, j}$.
  \item\label{tower-ap-9} If $\Tow$ is reduced, then for any $i < \alpha$ and $j < \beta$, $M_{0, j} \cap M_{i, 0} = M_{0, 0}$
  \end{enumerate}
\end{lem}
\begin{proof}
  First, we claim that we can assume without loss of generality that $\Tow'$ is limit and continuous. Indeed, suppose that we could prove the lemma for limit and continuous $\Tow'$. By Lemma \ref{tower-ext-cont} there exists $\Tow^\ast = \seq{M_i^\ast : i < \alpha} \smallfrown \seq{a_i : i + 1 < \alpha}$ $\tleq$-extending $\Tow'$ which is limit and continuous. Let $\Tow'' := \seq{M_i' : i < \alpha} \smallfrown \seq{a_i : i + 1 < \alpha}$ be defined by $M_0' = M_0$, $M_i' := M_i^\ast$ for $i > 0$. Then $\Tow''$ is also a limit continuous tower. By what we are assuming, there exists $\seq{M_{i, j}' : i < \alpha, j < \beta}$, $\seq{g_i : i < \alpha}$ satisfying the conclusion of the lemma, where $\Tow'$, $\seq{M_{i, j} : i < \alpha, j < \beta}$, $\seq{f_i : i < \beta}$ there is $\Tow''$, $\seq{M_{i, j}' : i < \alpha, j < \beta}$,$\seq{g_i : i < \alpha}$ here. Now define $f_i := g_i \rest M_i$ and let $M_{i, 0} := f_i[M_i]$, $M_{i, j} := M_{i, j}'$ for $j > 0$. This works. The main point is that, for $i + 1 < \alpha$,  $j \in (0, \beta)$, $\gtp (f_{i + 1} (a_i) / M_{i, j}; M_{i + 1, j})$ does not fork over $M_{i, 0}'$ (by construction), and since $\Tow' \tleq \Tow^\ast$, using invariance and the definition of $\tleq$, $\gtp (f_{i + 1} (a_i) / M_{i, 0}'; M_{i + 1, 0}')$ does not fork over $M_{i, 0}$. Thus by transitivity $\gtp (f_{i + 1} (a_i) / M_{i, j}; M_{i + 1, j})$ does not fork over $M_{i, 0}$.

  We have shown that we can assume without loss of generality that $\Tow'$ is limit and continuous. So assume that it is. We build a $\tleq$-increasing continuous chain of towers $\seq{\Tow^i : i < \alpha}$ and an increasing continuous chain $\seq{f_i : i < \alpha}$ such that, writing $\Tow^i = \seq{M_{i, j} : j < \beta} \smallfrown \seq{b_j : j + 1 < \beta}$, we have:

  \begin{enumerate}
  \item $\Tow^0 = \Tow$.
  \item For all $i < \alpha$, $f_i: N_i \cong_{N_0} M_{i, 0}$.
  \item For all $i + 1 < \alpha$, $\gtp (f_{i + 1} (a_i) / \bigcup_{j < \beta} M_{i, j}; M_{i + 1, j})$ does not fork over $M_{i, 0}$.
  \end{enumerate}

  This is possible by Lemma \ref{tower-ext-step} and some renaming. This is enough: check the requirements and the definition of $\tleq$.
\end{proof}

As an application of Lemma \ref{tower-ap}, we consider a key problem from \cite{tame-frames-revisited-jsl}: whether independent sequences have the symmetry property (or equivalently \cite[5.7]{tame-frames-revisited-jsl}, whether independence of a sequence depends on the order in which it is enumerated). This was shown by Shelah \cite[III.\S5]{shelahaecbook} assuming additional properties involving uniqueness triples. Here we answer positively assuming categoricity in $\lambda$. Categoricity is a much weaker hypothesis in practice, since usually we can restrict to a superlimit. Moreover categoricity is only needed here because Fact \ref{forking-fact} gives us a nonforking notion only over limit models. If we have a nonforking notions over all members of $\K_\lambda$, we can remove categoricity from the hypotheses.

Since the result will not be needed for the rest of the present paper so we only sketch the proof. The idea is similar to that in \cite[Theorem 3]{vandieren-symmetry-apal}, but we give a self-contained abstract proof.

\begin{thm}\label{indep-sym}
  Let $\K$ be a $\lambda$-superstable AEC with $\lambda$-symmetry. If $\K$ is categorical in $\lambda$, then the frame of independent sequences of length less than $\lambda^+$ defined in \cite[4.3]{tame-frames-revisited-jsl} has the symmetry property.
\end{thm}
\begin{proof}[Proof sketch]
  We assume some knowledge of \cite{tame-frames-revisited-jsl}, in particular we will use without further mention \cite[4.12]{tame-frames-revisited-jsl}, which says that most of the basic properties of forking generalize to independent sequences. For simplicity, work inside a model-homogeneous model $\sea \in \K_{\lambda^+}$. For $\bb_1, \bb_2 \in \fct{<\lambda^{+}}{\sea}$, we write $\bb_1 \equiv_M \bb_2$ to mean that the two sequences have the same type over $M$ (again inside $\sea$). Equivalently, there is an automorphism of $\sea$ fixing $M$ sending $\bb_1$ to $\bb_2$. We write $\bb_1 \nf_M \bb_2$ to mean that the sequence $\bb_2 \bb_1$ is independent over $M$ (inside $\sea$).

  By the proof of \cite[5.7]{tame-frames-revisited-jsl}, it is enough to see that whenever $b \ba$ is independent in $(M_0, M, N)$, we also have that $\ba b$ is independent in $(M_0, M, N)$. By transitivity, it is in fact enough to see that whenever $\ba \nf_M b$, we also have $b \nf_M \ba$. So assume $\ba \nf_M b$.

  In the present setup, an independent sequence is nothing more than a tower, so rephrasing Lemma \ref{tower-ap} in terms of independent sequences, we get that we can find $b'$ such that $b' \equiv_M b$, $\ba \nf_M b'$, and $b' \nf_M \ba$. It is now a simple matter of forking calculus to get the result: pick $\ba'$ such that $\ba' b' \equiv_M \ba b$. By invariance, $\ba' \nf_M b'$. By uniqueness, $\ba b' \equiv_M \ba' b'$. Thus $\ba b' \equiv_M \ba b$, and so by invariance, $b \nf_M \ba$, as desired.
\end{proof}

We end this section by deducing some results around disjoint amalgamation. The following appears in \cite[3.2.3(3)]{jrsh875} and the proof carries through to our setup.

\begin{fact}\label{ext-decomp}
  Assume that $\K$ is categorical in $\lambda$. Let $\R$ be a class of triples $(a, M, N)$ with $M \lea N$ limit  and $a \in N \backslash M$ such that:

  \begin{enumerate}
  \item $\R$ is closed under isomorphisms: $(a, M, N) \in \R$ and $f: N \cong N'$ implies $(f (a), f[M], N') \in \R$.
  \item $\R$ has the existence property: for any $M$ and any $p \in \gSna (M)$, there exists $(a, M, N) \in \R$ such that $p = \gtp (a / M; N)$.
  \end{enumerate}

  Let $(a, M, N) \in \R$ and let $N'$ be limit such that $N \lea N'$. Then there exists a limit ordinal $\delta < \lambda^+$ and a continuous tower $\Tow = \seq{M_i : i \le \delta} \smallfrown \seq{a_i : i < \delta}$ that $a_0 = a$, $M_0 = M$, $N' \lea M_\delta$, and $(a_i, M_i, M_{i + 1}) \in \R$ for all $i < \delta$.
\end{fact}

Reduced triples have the existence property in the sense just given (we are still assuming $\lambda$-superstability and $\lambda$-symmetry, see Hypothesis \ref{hyp-5}):

\begin{fact}[Existence property for reduced triples]\label{ext-reduced}
  Assume that $\K$ is categorical in $\lambda$. For any $M \in \K_\lambda$ and any $p \in \gSna (M)$, there exists a reduced triple $(a, M, N)$ such that $p = \gtp (a / M; N)$.
\end{fact}
\begin{proof}
  Let $p \in \gS (M)$ and write $p = \gtp (b / M; N_0)$. By Fact \ref{reduced-density}, there exists limit $M', N'$ such that $M \lea M'$, $N_0 \lea N'$, $q := \gtp (b / M'; N')$ does not fork over $M$, and $(b, M', N')$ is reduced. By the conjugation property, there exists $f: M' \cong M$ such that $f (q) = p$. Extend $f$ to some $g: N' \cong N$ and let $a := g (b)$. Then $(a, M, N)$ is reduced and $\gtp (a / M; N) = f(q) = p$, as desired.
\end{proof}

We can now improve Fact \ref{nf-amalgam} to make also the amalgams of $M_1$ and $M_2$ disjoint over $M_0$. Note that this gives in particular disjoint amalgamation.

\begin{lem}[Disjoint nonforking amalgamation]\label{d-nf-amalgam}
  Assume that $\K$ is categorical in $\lambda$. Let $M_0 \lea M_\ell$, $\ell = 1,2$, be in $\K_\lambda$. Let $a_\ell \in M_\ell$. There exists $f_1, f_2, M_3$ such that $M_3 \in \K_\lambda$, and $f_\ell : M_\ell \xrightarrow[M_0]{} M_3$ is such that $\gtp (f_\ell (a_\ell) / f_{3 - \ell} (M_{3 - \ell}); M_3)$ does not fork over $M_0$ for $\ell = 1,2$. Moreover, $f_1[M_1] \cap f_2[M_2] = M_0$.
\end{lem}
\begin{proof}
  Extending $M_1$ if necessary, we may assume without loss of generality that there exists $M_1' \lea M_1$ containing $a_1$ such that $(a_1, M_0, M_1')$ is reduced. By Facts \ref{ext-decomp} and \ref{ext-reduced}, there exists a continuous tower $\Tow = \seq{N_j : j \le \delta} \smallfrown \seq{b_j : j < \delta}$ such that $b_0 = a_1$, $M_1 \lea N_\delta$, and $(b_j, N_j, N_{j + 1})$ is reduced for all $j < \delta$. This implies in particular that $\Tow$ is reduced (Lemma \ref{red-tower-triple}). Similarly (and after perhaps growing $M_2$), there exists a continuous tower $\Tow' = \seq{M_i^\ast : i \le \delta^\ast} \smallfrown \seq{a_i^\ast : i < \delta^\ast}$ such that $a_2 = a_0^\ast$, $M_2 \lea M_{\delta^\ast}^\ast$, and $(a_i^\ast, M_i^\ast, M_{i  +1}^\ast)$ is reduced for all $i < \delta^\ast$. Let $\alpha := \delta^\ast + 1$, $\beta := \delta + 1$. Let $\seq{M_{i, j} : i < \alpha, j < \beta}$, $\seq{g_i : i < \alpha}$ be as given by Lemma \ref{tower-ap}, where $\seq{M_i : i < \alpha} \smallfrown \seq{a_i : i + 1 < \alpha}$ there stand for $\seq{M_i^\ast : i < \alpha} \smallfrown \seq{a_i^\ast : i < \alpha}$ here. Then $M_3 := M_{\delta^\ast, \delta}$, $f_1 := \id_{M_1}$, and $f_2 := g_{\delta^\ast} \rest M_2$ are as desired.
\end{proof}

\section{Building weakly successful good frames}

The goal of this section is to build a good $\lambda$-frame from $\lambda$-superstability, $\lambda$-symmetry, and categoricity in $\lambda$. We will use uniqueness triples, a key object in (for example) \cite[\S II.6]{shelahaecbook}. We quote the definition from \cite[\S4]{jrsh875}.

\begin{defin}\label{weakly-succ-def} Let $\K$ be a $\lambda$-superstable AEC.
  \begin{enumerate}
  \item\index{amalgam} For $M_0 \lea M_\ell$ all in $\K_\lambda$, $\ell = 1,2$, an \emph{amalgam of $M_1$ and $M_2$ over $M_0$} is a triple $(f_1, f_2, N)$ such that $N \in \K_\lambda$ and $f_\ell : M_\ell \xrightarrow[M_0]{} N$.
  \item\index{equivalence of amalgam} Let $(f_1^x, f_2^x, N^x)$, $x = a,b$ be amalgams of $M_1$ and $M_2$ over $M_0$. We say $(f_1^a, f_2^a, N^a)$ and $(f_1^b, f_2^b, N^b)$ are \emph{equivalent over $M_0$} if there exists $N_\ast \in \K_\lambda$ and $f^x : N^x \rightarrow N_\ast$ such that $f^b f_1^b = f^a  f_1^a$ and $f^b f_2^b = f^a f_2^a$, namely, the following commutes:

  \[
  \xymatrix{ & N^a \ar@{.>}[r]^{f^a} & N_\ast \\
    M_1 \ar[ru]^{f_1^a} \ar[rr]|>>>>>{f_1^b} & & N^b \ar@{.>}[u]_{f^b} \\
    M_0 \ar[u] \ar[r] & M_2 \ar[uu]|>>>>>{f_2^a}  \ar[ur]_{f_2^b} & \\
  }
  \]

  Note that being ``equivalent over $M_0$'' is an equivalence relation (\cite[4.3]{jrsh875}).
\item  A triple $(a, M, N)$ is a \emph{uniqueness triple} if $M \lea N$ are both in $\K_\lambda$, $a \in |N| \backslash |M|$, and for any $M_1 \gea M$ in $\K_\lambda$, there exists a \emph{unique} (up to equivalence over $M$) amalgam $(f_1, f_2, N_1)$ of $N$ and $M_1$ over $M$ such that $\gtp (f_1 (a) / f_2[M_1] ; N_1)$ does not fork over $M$.
\item $\K$ has the \emph{$\lambda$-existence property for uniqueness triples} if for any $M \in \K_{\lambda}$ and any $p \in \gSna (M)$, one can write $p = \gtp (a / M; N)$ with $(a, M, N)$ a uniqueness triple.
  \end{enumerate}
\end{defin}

Under mild conditions, uniqueness triples are reduced:

\begin{fact}\label{uq-triple-red}
  If $\K$ is a $\lambda$-superstable AEC with $\lambda$-symmetry which is categorical in $\lambda$, then any uniqueness triple is reduced.
\end{fact}
\begin{proof}
  By the proof of \cite[4.1.11(2)]{jrsh875}. Note that conjugation (which is assumed there) holds by Fact \ref{conj-prop}. Further, the disjoint amalgamation statement used in the proof follows from Lemma \ref{d-nf-amalgam}.
\end{proof}

Once we have the existence property for uniqueness triples, we can build the desired good frame. This is the main technical result of the paper. Note that by canonicity (Fact \ref{canon-thm}), the nonforking relation of the frame will automatically coincide with that given by Definition \ref{forking-def}.

\begin{thm}\label{main-frame-constr}
  Let $\K$ be a $\lambda$-superstable AEC with $\lambda$-symmetry which is categorical in $\lambda$. If $\K$ has the $\lambda$-existence property for uniqueness triples, then there is a weakly successful good $\lambda$-frame on $\K$.
\end{thm}
\begin{proof}
  We first follow the proof of \cite[5.5.4]{jrsh875} to show that there exists a $4$-ary relation $\NF$ on $\K_\lambda$ satisfying several of the properties of nonforking amalgamation in a stable first-order theory: monotonicity, existence, uniqueness, symmetry and long transitivity (see \cite[5.2.1]{jrsh875}). 

  The hypotheses of \cite[5.5.4]{jrsh875} are that we work with a semi-good $\lambda$-frame with the conjugation property which has the existence property for uniqueness triples (\cite[5.1.1]{jrsh875}). In the context here, we have the existence property for uniqueness triples by assumption and also have the conjugation property (Fact \ref{conj-prop}). Moreover, nonforking induces a good $\lambda$-frame, except that it only satisfies weak versions of continuity and local character (Fact \ref{ss-frame}). Now it is pointed out explicitly by Jarden and Shelah \cite[5.1.2]{jrsh875} that local character is never used in their proof of \cite[5.5.4]{jrsh875}.

  As for continuity, it is only used in one place: the proof of \cite[5.4.6]{jrsh875} (the ``opposite uniqueness proposition''). There an appeal to \cite[3.1.10]{jrsh875} is made and its proof relies on continuity. We show how to bypass it so that the proposition still goes through. The statement is: if $\NF^\ast (N_0, N_1, N_2, N_3^a)$ and $\NF^\ast (N_0, N_2, N_1, N_3^b)$, then there is an amalgam of $N_3^a$ and $N_3^b$ fixing $N_1 \cup N_2$. Here, $\NF^\ast (N_0, N_1, N_2, N_3)$ means \cite[5.3.1]{jrsh875} that $N_0, N_1, N_2, N_3$ are in $\K_\lambda$, $N_0 \lea N_\ell \lea N_3$ for $\ell = 1,2$, and there exists $\alpha^\ast < \lambda^+$, increasing continuous $\seq{N_{\ell, i} : i \le \alpha^\ast}$, and $\seq{d_i : i < \alpha^\ast}$ such that $N_{1, 0} = N_0$, $N_{1, \alpha^\ast} = N_1$, $N_{2, 0} = N_2$, $N_{2, \alpha^\ast} = N_3$, $N_{1, i} \lea N_{2, i}$ for $i \le \alpha^\ast$, $\gtp (d_i / N_{2, i}; N_{2, i + 1})$ does not fork over $N_{1, i}$, and $(d_i, N_{1, i}, N_{1, i + 1})$ is a uniqueness triple.

  Now in the setup of the opposite uniqueness proposition, we are starting with $\seq{N_{1, i}^a, d_i^a : i \le \alpha^\ast}$ witnessing $\NF^\ast (N_0, N_1, N_2, N_3^a)$ and $\seq{N_{1, i}^b, d_i^b : i \le \beta^\ast}$ witnessing $\NF^\ast (N_0, N_2, N_1, N_3^a)$. They form two towers, as in the hypotheses of Lemma \ref{tower-ap}. Further, for $x = a,b$, $(d_i^x, N_{1, i}^x, N_{1, i + 1}^x)$ is reduced by Fact \ref{uq-triple-red}. Thus by Lemma \ref{red-tower-triple} the entire towers are reduced so the hypothesis of clause (\ref{tower-ap-9}) of Lemma \ref{tower-ap} holds. The opposite uniqueness proposition therefore follows from applying Lemma \ref{tower-ap} instead of \cite[3.1.10]{jrsh875} to amalgamate the two towers into a rectangle of models, and then following the proof of opposite uniqueness in \cite{jrsh875} (no further appeals to \cite[3.1.10]{jrsh875} are made).

  In the end, the proof of \cite[5.5.4]{jrsh875} goes through, and we have the desired relation $\NF$. We now apply \cite[4.3]{jash940-v1}, which says that given a weak kind of good frame (called an almost good $\lambda$-frame) and a nonforking relation $\NF$ as here, we get a good $\lambda$-frame. In our setup, the frame induced by nonforking satisfies all the axioms of almost good $\lambda$-frame except perhaps for continuity. However the proof of \cite[4.3]{jash940-v1} does not rely on continuity except to prove stability (which we already have), so we are done: at last by local character, type-fullness, and transitivity, we get continuity.
\end{proof}

To get the existence property for uniqueness triples, we will use that their existence follows from the weak diamond and some amount of stability in $\lambda^+$ (the argument is essentially due to Shelah):

\begin{fact}\label{uq-triple-wd}
  Let $\K$ be a $\lambda$-superstable AEC with $\lambda$-symmetry which is categorical in $\lambda$. If $\WGCH (\lambda)$ and for any saturated model $M$ in $\K_{\lambda^+}$ there is $N$ universal over $M$, then $\K$ has the $\lambda$-existence property for uniqueness triples.
\end{fact}
\begin{proof}
  Recalling Fact \ref{ss-frame}, the proof of \cite[E.8]{downward-categ-tame-apal} goes through, noting that the tree $\seq{M_\eta : \eta \in \fct{\le \lambda^+}{2}}$ in the proof can be built with the additional property that $M_{\eta \smallfrown \ell}$ is universal over $M_\eta$ for $\ell = 0,1$ so that universal continuity rather than continuity or local character is sufficient.
\end{proof}

\begin{cor}[The local good frame construction theorem]\label{local-frame-constr}
  Let $\K$ be an AEC and let $\lambda \ge \LS (\K)$. If:

  \begin{enumerate}
  \item $\WGCH (\lambda)$.
  \item $\K$ is $\lambda$-superstable, has $\lambda$-symmetry, and is categorical in $\lambda$.
  \item\label{local-frame-2} For every saturated model $M$ in $\K_{\lambda^+}$, there is $N$ universal over $M$.
  \end{enumerate}

  Then there is a weakly successful good $\lambda$-frame with underlying class $\K_\lambda$. 
\end{cor}
\begin{proof}
  Combine Theorem \ref{main-frame-constr} and Fact \ref{uq-triple-wd}.
\end{proof}
\begin{remark}\label{optimality-rmk}
  There are examples of categorical good frames that are not weakly successful \cite{counterexample-frame-afml}. Thus we cannot expect to remove condition (\ref{local-frame-2}) from Corollary \ref{local-frame-constr}.
\end{remark}

We can state a slightly more precise version of Corollary \ref{local-frame-constr} using superlimits and the terminology from Definition \ref{nice-ss-def}. First, one more definition:

\begin{defin}
  A nicely sl-$\lambda$-superstable AEC $\K$ \emph{has sl-uniqueness triples (in $\lambda$)} if $\Kslp{\lambda}$ (see Definition \ref{ksl-def}) has the $\lambda$-existence property for uniqueness triples. For $\Theta$ a class of cardinals, a nicely sl-$\Theta$-superstable AEC $\K$ \emph{has sl-uniqueness triples in $\Theta$} if it has sl-uniqueness triples in $\lambda$ for every $\lambda \in \Theta$.
\end{defin}

\begin{cor}\label{very-nice-frame}
  If $\K$ is nicely sl-$\lambda$-superstable and has sl-uniqueness triples, then there is a categorical weakly successful good $\lambda$-frame on $\Kslp{\lambda}$.
\end{cor}
\begin{proof}
  Apply Theorem \ref{main-frame-constr} to $\Kslp{\lambda}$.
\end{proof}

\begin{cor}\label{weak-succ-constr}
  Let $\K$ be an AEC and let $\lambda \ge \LS (\K)$. If:

  \begin{enumerate}
  \item $\WGCH (\lambda)$.
  \item $\K$ is nicely sl-$\lambda$-superstable.
  \item $\K$ is nicely $\lambda^+$-stable.
  \end{enumerate}

  Then $\K$ has sl-uniqueness triples in $\lambda$.
\end{cor}
\begin{proof}
  By definition, it suffices to check that $\Kslp{\lambda}$ has the $\lambda$-existence property for uniqueness triples. By Fact \ref{uq-triple-wd}, it suffices to see that for any saturated $M \in \Kslp{\lambda}_{\lambda^+}$, there is $N \in \Kslp{\lambda}_{\lambda^+}$ universal over $M$. We know that $\K$ is nicely $[\lambda, \lambda^+]$-stable, so by Theorem \ref{nicely-stable-sl}, $\Kslp{\lambda}$ is nicely $\lambda^+$-stable. Now note that a saturated model in $\Kslp{\lambda}_{\lambda^+}$ is nothing but a $(\lambda, \lambda^+)$-limit, which by Fact \ref{bf-lambdap} is a $(\lambda^+, \lambda^+)$-limit. By definition of nice stability, this means that $M$ is an amalgamation base, and hence has a universal extension, as desired.
\end{proof}

\section{Building successful good frames}

Suppose that $\K$ is an AEC with some good structural properties, $\lambda \ge \LS (\K)$. In this section, we want to apply Corollary \ref{local-frame-constr} to all successors of $\lambda$. What do we get if we succeed? In a nutshell, we get a sequence $\seq{\s^n : n < \omega}$, where each $\s^n$ is a good $\lambda^{+n}$-frame. Recall that by the canonicity theorem (Fact \ref{canon-thm}), the exact definition of nonforking in these frames does not matter: there can be at most one at each cardinal.

The following crucial result tells us that in fact the frames will be ``connected'', in the sense that $\s^{n + 1}$ is the successor of $\s^n$ (so forking in $\s^{n + 1}$ is described in terms of forking in $\s^n$).

\begin{fact}[{\cite[5.15]{tame-succ-v5-toappear}}]\label{succ-frame-thm}
  Let $\s$ be a weakly successful categorical good $\lambda$-frame on an AEC $\K$. If there is a good $\lambda^+$-frame on $\Ksatp{\lambda^+}$, then $\s$ is successful $\goodp$.
\end{fact}

We deduce a close relationship between nice sl-superstability with sl-uniqueness triples in $[\lambda, \lambda^{+n}]$ and the existence of an $n$-successful $\goodp$ $\lambda$-frame:

\begin{thm}\label{very-nice-succ-thm}
  Let $\K$ be an AEC, let $\lambda \ge \LS (\K)$, and let $n \in [1, \omega]$.
  
  \begin{enumerate}
  \item If $\K$ is nicely sl-superstable and has sl-uniqueness triples in $[\lambda, \lambda^{+n}]$, then there is an $n$-successful $\goodp$ $\lambda$-frame on $\Kslp{\lambda}$ (see Definition \ref{ksl-def}).
  \item If $\K_\lambda$ has a superlimit and there is an $n$-successful $\goodp$ $\lambda$-frame on $\Kslp{\lambda}$, then $\K$ is nicely sl-$[\lambda, \lambda^{+n}]$-superstable and has sl-uniqueness triples in $[\lambda, \lambda^{+(n - 1)}]$.
  \end{enumerate}
\end{thm}
\begin{proof} \
  \begin{enumerate}
  \item By Corollary \ref{very-nice-frame}, there is a weakly successful good $\lambda$-frame $\s$ on $\Kslp{\lambda}$. We prove by induction on $n \in [1, \omega)$ that $\s^{+(n - 1)}$ is defined and is a successful $\goodp$ frame on $\Kslp{\lambda^{+(n - 1)}}$. When $n = 1$, $\s^{+0} = \s$. Moreover by Corollary \ref{very-nice-frame} applied to $\K_{\lambda^+}$, there is a weakly successful good $\lambda^+$-frame on $\Kslp{\lambda^+}$. Now, $\Kslp{\lambda}$ is dense in $\K_{[\lambda, \lambda^+]}$ (Theorem \ref{nicely-stable-sl}), so the superlimit in $\K_{\lambda^+}$ must be in $\Kslp{\lambda}$ and must be superlimit there too. Also, the superlimit in $\K_{\lambda^+}$ is $(\lambda^+,\lambda^+)$-limit by Lemma \ref{sl-lim}, so it is saturated, i.e.\ $\Kslp{\lambda^+} = \Ksatp{\lambda^+}$. By Fact \ref{succ-frame-thm}, $\s$ is successful $\goodp$. For the inductive case, apply the base case to $\Kslp{\lambda^{+(n - 1)}}$ (recalling the canonicity theorem, Fact \ref{canon-thm}).
  \item Similarly to before (using Fact \ref{frame-ss} to derive superstability from a good frame and \cite[5.12]{tame-succ-v5-toappear} to get the existence property for uniqueness triples), we get that $\K$ is nicely sl-superstable and has sl-uniqueness triples in $[\lambda, \lambda^{+(n - 1)}]$. At the last step, we do not know whether we get a weakly successful good frame, but we still have enough to derive that $\K$ is nicely sl-$\lambda^{+n}$-superstable.
  \end{enumerate}
\end{proof}
\begin{remark}\label{opt-rmk-2}
  As in Remark \ref{optimality-rmk}, the Hart-Shelah example shows \cite{counterexample-frame-afml} that the second conclusion of Theorem \ref{very-nice-succ-thm} \emph{cannot} be strengthened to get sl-uniqueness triples in $\lambda^{+n}$.
\end{remark}

We now iterate Corollary \ref{local-frame-constr}. To get the best possible results, it will be handy to know we get superlimits automatically after one iteration:

\begin{fact}\label{succ-goodp}
  If $\s$ is a successful categorical $\goodp$ $\lambda$-frame on an AEC $\K$, then $\s^+$ is a $\goodp$ $\lambda$-frame. In particular, $\Ksatp{\lambda^{+}}_{\lambda^{++}}$ has a superlimit.
\end{fact}
\begin{proof}
  The frame $\s^+$ is $\goodp$ by \cite[III.1.9]{shelahaecbook}. The ``in particular'' part then follows from \cite[3.14]{aec-stable-aleph0-apal}.
\end{proof}

\begin{thm}\label{n-ss-1}
  Let $\K$ be an AEC, let $\lambda \ge \LS (\K)$, and let $n \in [1, \omega)$. If:
  
  \begin{enumerate}
  \item $\WGCH ([\lambda, \lambda^{+n}])$.
  \item $\K$ is nicely $[\lambda^{+2}, \lambda^{+(n + 1)}]$-stable.
  \item $\K_{\lambda}$ has a superlimit.
  \item $\K_{\lambda^+}$ has a superlimit.
  \item $\K$ is $\lambda^{+(n + 1)}$-semisolvable (recall Definition \ref{solv-def}).
  \end{enumerate}

  Then $\K$ is nicely sl-superstable and has sl-uniqueness triples in $[\lambda, \lambda^{+n}]$.
\end{thm}
\begin{proof}
  Since $\K_{\lambda}$ and $\K_{\lambda^+}$ have superlimits, $\K_{[\lambda, \lambda^+]}$ has no maximal models. Since $\K$ is nicely $[\lambda^{+2}, \lambda^{+n}]$-stable, $\K_{[\lambda, \lambda^{+n}]}$ has no maximal models. Thus $\K$ is $(\lambda^{+m}, \lambda^{+(n + 1)})$-extendible for any $m \le n + 1$. By Fact \ref{semisolv-fact}(\ref{semisolv-4}) (used twice), $\K$ is nicely sl-$[\lambda, \lambda^+]$-superstable, hence $\K$ is nicely $[\lambda, \lambda^{+(n + 1)}]$-stable. By Corollary \ref{weak-succ-constr} (used twice), $\K$ has sl-uniqueness triples in $[\lambda, \lambda^+]$. In particular (Theorem \ref{very-nice-succ-thm}), $\Kslp{\lambda}$ has a successful $\goodp$ $\lambda$-frame. By Fact \ref{succ-goodp}, $\Kslp{\lambda}$ has a superlimit in $\lambda^{++}$. By Remark \ref{skel-basics}, this implies that $\K_{\lambda^{++}}$ has a superlimit. Now iterate what has just been proven, replacing $\lambda$ by $\lambda^+$, $\lambda^{++}$, ..., $\lambda^{+(n - 1)}$.
\end{proof}

The assumption of nice $\lambda^{+(n + 1)}$-stability in Theorem \ref{n-ss-1} cannot be removed because of the Hart-Shelah example (see Remark \ref{opt-rmk-2}). If we are only interested in building an $n$-successful $\goodp$-frame, we can replace it by categoricity in $\lambda^{+(n + 1)}$, or even by few models in $\lambda^{+(n + 1)}$. To see this, we will use two more facts:

\begin{fact}[{\cite[7.1.3]{jrsh875}}]\label{succ-many-models}
  Let $\s$ be a weakly successful categorical good $\lambda$-frame on an AEC $\K$. If $\Ii (\K, \lambda^{++}) < 2^{\lambda^{++}}$, then $\s$ is successful.
\end{fact}

\begin{fact}\label{succ-decent}
  Let $\s$ be a successful categorical good $\lambda$-frame on an AEC $\K$. Then $\K_{\lambda^+}$ has a superlimit if and only if $\s$ is $\goodp$.
\end{fact}
\begin{proof}
  By \cite[3.2]{tame-succ-v5-toappear} $\s$ is decent if and only if there is a superlimit in $\K_{\lambda^+}$. By \cite[5.12]{tame-succ-v5-toappear}, $\s$ is decent if and only if it is $\goodp$.
\end{proof}

\begin{thm}\label{n-ss-2}
  Let $\K$ be an AEC, let $\lambda \ge \LS (\K)$, and let $n \in [1, \omega)$. If:

    \begin{enumerate}
    \item $\WGCH ([\lambda, \lambda^{+n}])$.
    \item $\K$ is nicely $[\lambda^{+2}, \lambda^{+n}]$-stable.
    \item $\K_\lambda$ has a superlimit.
    \item $\K_{\lambda^+}$ has a superlimit.
    \item $\K$ is $\lambda^{+(n + 1)}$-semisolvable.
    \item $\Ii (\K, \lambda^{+(n + 1)}) < 2^{\lambda^{+(n + 1)}}$.
    \end{enumerate}

    Then there is an $n$-successful $\goodp$ $\lambda$-frame on $\Kslp{\lambda}$.
\end{thm}
\begin{proof}
  As in the proof of Theorem \ref{n-ss-1}, we can replace $\lambda$ by $\lambda^{+(n - 1)}$ if necessary to assume without loss of generality that $n = 1$. Also as in the proof of Theorem \ref{n-ss-1}, $\K$ is nicely sl-$\lambda$-superstable and has sl-uniqueness triples. By Corollary \ref{very-nice-frame}, there is a weakly successful good $\lambda$-frame $\s$ on $\Kslp{\lambda}$. By Fact \ref{succ-many-models}, $\s$ is successful. By Fact \ref{succ-decent}, $\s$ is $\goodp$.
\end{proof}

\begin{remark}
  In Theorem \ref{n-ss-2}, if the superlimit in $\K_{\lambda^+}$ is an amalgamation base, the proof shows that we can replace $\WGCH ([\lambda, \lambda^{+n}])$ by $\WGCH ([\lambda, \lambda^{+n}))$.
\end{remark}

\section{The main categoricity transfer}

Having successful good frames implies that the categoricity spectrum is quite well understood, as the next two facts show:

\begin{fact}\label{succ-categ}
  Let $\s$ be a categorical $n$-successful $\goodp$ $\lambda$-frame on an AEC $\K$. If $\K$ is categorical in \emph{some} $\mu \in (\lambda, \lambda^{+(n + 1)}]$, then $\K$ is categorical in \emph{all} $\mu' \in (\lambda, \lambda^{+(n + 1)}]$. Moreover, $\K_{[\lambda, \lambda^{+n}]}$ has amalgamation, no maximal models, and is $\lambda$-tame.
\end{fact}
\begin{proof}
  Combining \cite[III.2.3(4)]{shelahaecbook}, \cite[III.2.9]{shelahaecbook}, and \cite[III.2.12]{shelahaecbook}, we get that $\K$ is categorical in every $\mu' \in (\lambda, \lambda^{+(n + 1)}]$. Tameness then follows from \cite[III.1.10]{shelahaecbook}. As for amalgamation and no maximal models, they are immediate from categoricity and the definition of a good frame.
\end{proof}

\begin{fact}\label{limit-categ}
  Let $\s$ be a categorical $\omega$-successful $\goodp$ $\lambda$-frame on an AEC $\K$. If $\WGCH ([\lambda^{+n}, \lambda^{+\omega}))$ holds for some $n < \omega$, then $\K$ has arbitrarily large models, and if $\K$ is categorical in \emph{some} $\mu > \lambda$, then $\K$ is categorical in \emph{all} $\mu' > \lambda$ and moreover $\K_{\ge \lambda}$ has amalgamation, no maximal models, and is $\lambda$-tame.
\end{fact}
\begin{proof}
  By \cite[7.20(1)]{multidim-v2} and \cite[14.4]{multidim-v2}
\end{proof}

We can now prove the main technical categoricity transfer theorem of this paper. As in most results of the present paper, the hypotheses include a set-theoretic assumption, the $\WGCH$. The model-theoretic assumptions are essentially the following: no maximal models, some superlimits, and some instances of nice stability. They all hold in categorical AECs with amalgamation (see e.g.\ Fact \ref{categ-struct}), and we will see (Theorem \ref{sl-nmm-thm}) that they can also be derived from only (categoricity and) no maximal models. 

\begin{thm}[Main theorem]\label{categ-main}
  Let $\K$ be an AEC, let $ \lambda \ge \LS (\K)$, $\mu > \lambda^+$, and set $\mu_0 := \min (\mu, \lambda^{+\omega})$. If:

  \begin{enumerate}
  \item $\WGCH ([\lambda, \mu_0))$.
  \item $\K$ has arbitrarily large models.
  \item $\K_\lambda$ has a superlimit.
  \item $\K_{\lambda^+}$ has a superlimit.
  \item $\K$ is nicely $[\lambda^{+2}, \mu_0)$-stable. 
  \item $\K$ is $(\mu_0, \mu)$-extendible.
  \item $\K$ is categorical in $\mu$.
  \end{enumerate}

  Then $\Kslp{\lambda}_{[\lambda, \mu]}$ ($\Kslp{\lambda}$ is the class generated by the superlimit in $\lambda$, see Definition \ref{ksl-def}) is categorical in every $\mu' \in [\lambda, \mu]$, has amalgamation, no maximal models, and is $\lambda$-tame. Moreover if $\mu \ge \lambda^{+\omega}$, then $\Kslp{\lambda}$ is categorical in every $\mu' \ge \lambda$, has amalgamation, no maximal models, and is $\lambda$-tame.
\end{thm}
\begin{proof}
  By Remark \ref{solv-rmk}, $\K$ is $\mu$-semisolvable. Either $\mu < \lambda^{+\omega}$ or $\mu \ge \lambda^{+\omega}$. We consider these cases separately.

  \begin{itemize}
  \item If $\mu < \lambda^{+\omega}$, fix $n \in [1, \omega)$ such that $\mu = \lambda^{+(n + 1)}$ (recall that we are assuming that $\mu > \lambda^+$). By Theorem \ref{n-ss-2}, there is an $n$-successful $\goodp$ $\lambda$-frame on $\Kslp{\lambda}$. Now apply Fact \ref{succ-categ}.
  \item If $\mu \ge \lambda^{+\omega}$, first fix $n \in [1, \omega)$. Note that the hypotheses imply that $\K_{[\lambda, \lambda^{+\omega})}$ has no maximal models, so $\K$ must be $(\lambda^{+(n + 1)}, \mu)$-extendible. By Fact \ref{solv-downward}, $\K$ is $\lambda^{+(n + 1)}$-semisolvable. By Theorem \ref{n-ss-1}, $\K$ is nicely sl-superstable and has sl-uniqueness triples in $[\lambda, \lambda^{+n}]$. Since this is valid for all $n < \omega$, $\K$ is nicely sl-superstable and has sl-uniqueness triples in $[\lambda, \lambda^{+\omega})$. By Theorem \ref{very-nice-succ-thm}, for each $n < \omega$ there is an $n$-successful $\goodp$ $\lambda$-frame on $\Kslp{\lambda}$. By canonicity (Fact \ref{canon-thm}), all these good frames must be the same, and so in fact there is an $\omega$-successful $\goodp$ $\lambda$-frame on $\Kslp{\lambda}$. Now apply Fact \ref{limit-categ}.
  \end{itemize}
\end{proof}

\begin{remark}\label{sl-categ-rmk}
  For any AEC $\K$ and any $\mu > \lambda \ge \LS (\K)$ such that $\K_\lambda$ has a superlimit and $\K$ is categorical in $\mu$, $\Kslp{\lambda}_{\ge \mu} \subseteq \K_{\ge \mu}$. In the conclusion of the above theorem, equality will hold, because $\Kslp{\lambda}$ has a model of cardinality $\mu$.
\end{remark}

For the convenience of the reader, we state the special case of Theorem \ref{categ-main} where the AEC is also categorical in $\lambda$ (this will suffice to prove categoricity transfers in AECs with amalgamation, but in Section \ref{categ-nmm-sec} we will use the general case).

\begin{cor}\label{categ-main-cor}
  Let $\K$ be an AEC, let $ \lambda \ge \LS (\K)$, $\mu > \lambda^+$, and set $\mu_0 := \min (\mu, \lambda^{+\omega})$. If:

  \begin{enumerate}
  \item $\WGCH ([\lambda, \mu_0))$.
  \item $\K$ has arbitrarily large models.
  \item $\K$ is categorical in $\lambda$.
  \item $\K_{\lambda^+}$ has a superlimit.
  \item $\K$ is nicely $[\lambda^{+2}, \mu_0)$-stable. 
  \item $\K$ is $(\mu_0, \mu)$-extendible.
  \item $\K$ is categorical in $\mu$.
  \end{enumerate}

  Then $\K$ is categorical in every $\mu' \in [\lambda, \mu]$, has amalgamation, no maximal models, and is $\lambda$-tame. Moreover if $\mu \ge \lambda^{+\omega}$, then $\K$ is categorical in every $\mu' \ge \lambda$, has amalgamation, no maximal models, and is $\lambda$-tame.
\end{cor}
\begin{proof}
  Since $\K$ is categorical in $\lambda$, $\Kslp{\lambda} = \K_{\ge \lambda}$, so this is immediate from Theorem \ref{categ-main}.
\end{proof}

\section{Categoricity assuming amalgamation}

In this section, we use Corollary \ref{categ-main-cor} to fully describe the categoricity spectrum of AECs with amalgamation. Recall the definition of the categoricity spectrum:

\begin{defin}\label{cat-spec-def}
  The \emph{categoricity spectrum} of an AEC $\K$ is the class of cardinals $\mu \ge \LS (\K)$ such that $\K$ is categorical in $\mu$. We write $\Cat (\K)$ for the categoricity spectrum of $\K$.
\end{defin}
\begin{remark}
  For any AEC $\K$ and any cardinal $\lambda$, $\Cat (\K_{\ge \lambda}) = \Cat (\K) \cap [\lambda, \infty)$.
\end{remark}

We will use the following known results about the categoricity spectrum of AECs with amalgamation and arbitrarily large models. First, it is not too difficult to see that it is a closed class (see for example \cite[2.22]{categ-universal-2-selecta}):

\begin{fact}\label{closed-cat-spec}
  If $\K$ is an AEC with amalgamation and arbitrarily large models, then $\Cat (\K)$ is a closed class. More precisely, let $\K$ be an AEC with arbitrarily large models and let $\seq{\mu_i : i < \delta}$ be an increasing sequence in $\Cat (\K)$. If $\K_{\le \mu_i}$ has amalgamation for each $i < \delta$, then $\sup_{i < \delta} \mu_i \in \Cat (\K)$.
\end{fact}

We will also use that a version of Morley's omitting type holds in AECs. This is essentially due to Shelah. See \cite[9.2]{downward-categ-tame-apal} for a full proof.

\begin{fact}[Morley's omitting type theorem for AECs]\label{morley-omitting}
  Let $\K$ be an AEC with amalgamation and let $\mu > \LS (\K)$. If every model in $\K_\mu$ is $\LS (\K)^+$-saturated, then there exists $\chi < \hanf{\LS (\K)}$ such that every model in $\K_{\ge \chi}$ is $\LS (\K)^+$-saturated.
\end{fact}

Specializing Corollary \ref{categ-main-cor} to AECs with amalgamation, we get that the categoricity spectrum has no gaps:

\begin{lem}\label{cat-spec-1}
  Assume $\WGCH$. Let $\K$ be an AEC with amalgamation and arbitrarily large models. If $\mu_1 < \mu_2$ are both in\footnote{Recall that by definition of $\Cat (\K)$, this implies that $\mu_1 \ge \LS (\K)$.}  $\Cat (\K)$, then $[\mu_1, \mu_2] \subseteq \Cat (\K)$.
\end{lem}
\begin{proof}
  If $\mu_2 = \mu_1^+$, there is nothing to prove, so assume that $\mu_2 > \mu_1^+$. We apply Corollary \ref{categ-main-cor}, where $\lambda, \mu$ there stand for $\mu_1, \mu_2$ here. We have to check that its hypotheses are satisfied. By assumption $\WGCH$ holds and $\K$ has arbitrarily large models. By assumption, $\K$ is categorical in $\mu_1$. Since $\K$ is categorical in $\mu_1$, $\K$ has joint embedding in $\mu_1$, and so a diagram chase (using amalgamation and the fact that $\mu_1 \ge \LS (\K)$) shows that $\K_{\ge \mu_1}$ also has joint embedding. In particular, $\K_{\ge \mu_1}$ has no maximal models and so in particular $\K$ is $(\mu_0, \mu_2)$-extendible, where we have set $\mu_0 := \min (\mu_2, \mu_1^{+\omega})$. By Fact \ref{semisolv-fact}(\ref{semisolv-2}) (using categoricity in $\mu_2$ and Remark \ref{solv-rmk}), $\K$ is $\mu$-superstable and has $\mu$-symmetry, for every $\mu \in [\mu_1, \mu_2)$. In particular, $\K$ is nicely $[\mu_1^{+2}, \mu_0)$-stable. Also, by Fact \ref{categ-struct}, the saturated model of cardinality $\mu_1^+$ is a superlimit. This completes the verification that the hypotheses of Corollary \ref{categ-main-cor} hold.
\end{proof}

Further, if we are only categorical in one cardinal we can apply Corollary \ref{categ-main-cor} to a subclass of saturated models:

\begin{lem}\label{cat-spec-2}
  Assume $\WGCH$. Let $\K$ be an AEC with amalgamation and arbitrarily large models. If $\Cat (\K) \cap [\LS (\K)^{+\omega}, \infty) \neq \emptyset$, then there exists $\chi < \hanf{\LS (\K)}$ such that $[\chi, \infty) \subseteq \Cat (\K)$.
\end{lem}
\begin{proof}
  Let $\mu \in \Cat (\K) \cap [\LS (\K)^{+\omega}, \infty)$. First, we use the folklore trick of partitioning the AEC into disjoint classes, each of which has joint embedding: define an equivalence relation $\sim$ on $\K$ as follows: $M \sim N$ if $M$ and $N$ embed into a common model. Amalgamation implies that this is an equivalence relation. Let $\seq{\K^i}{i \in I}$ list the $\sim$-equivalence classes without repetition. We make several observations:

    \begin{enumerate}
    \item For each $i \in I$, $\K^i$ has joint embedding (by definition of $\sim$), is an AEC with amalgamation, and $\LS (\K^i) \le \LS (\K)$ (easy to check just use that if $M \lea N$ then $M \sim N$). 
    \item There exists a unique $i = i^\ast \in I$ such that $\K^i$ has arbitrarily large models: take $i^\ast$ such that $\K^{i^\ast}$ contains the model of cardinality $\mu$. No other class can contain a model of cardinality $\mu$ by categoricity (equivalence classes are disjoint).
    \item For each $i \in I - \{i^\ast\}$, there exists $\chi_i < \hanf{\LS (\K)}$ so that $\K_{\ge \chi_i}^i = \emptyset$ (by what has just been observed and Fact \ref{arb-large-fact}).      
    \item $|I| \le 2^{\LS (\K)}$: there are at most $2^{\LS (\K)}$-many nonisomorphic models in $\K_{\le \LS (\K)}$, and each of the $\K^i$'s must contain one such model.
    \item $\chi^0 := \sup_{i \in I - \{i^\ast\}} \chi_i$ is strictly less than $\hanf{\LS (\K)}$ by cofinality considerations.
    \item $\K_{\ge \chi^0} = \K_{\ge \chi^0}^{i^\ast}$ by definition.
    \end{enumerate}

    Let $\K^\ast := \K^{i^\ast}$. We have that $\K^\ast$ has amalgamation, joint embedding, and arbitrarily large models, so in particular $\K_{<\mu}^\ast$ has no maximal models. By Fact \ref{categ-struct}, the class $\K^{\ast \ast}$ of $\LS (\K)^+$-saturated models of $\K^\ast$ is an AEC which is categorical in both $\LS (\K)^+$ and $\mu$. Proceeding as in the proof of Lemma \ref{cat-spec-1} (where $\K, \mu_1, \mu_2$ there stand for $\K^{\ast \ast}$, $\LS (\K)^+, \mu$ here), we see that Corollary \ref{categ-main-cor} applies, and hence $\K^{\ast \ast}$ is categorical in every $\mu' \ge \LS (\K)^+$. Now apply Fact \ref{morley-omitting} to get that $\K^\ast$ is categorical in all $\mu' \ge \chi$ for some $\chi < \hanf{\LS (\K)}$. Making $\chi$ bigger, we can assume without loss of generality that $\chi \ge \chi^0$ so $\K_{\ge \chi} = \K_{\ge \chi}^\ast$, i.e.\ $\K$ is also categorical in all $\mu' \ge \chi$.
\end{proof}

The two lemmas just proven are enough to give the full description of the categoricity spectrum:

\begin{cor}\label{main-cor}
  Assume $\WGCH$. If $\K$ is an AEC with amalgamation and arbitrarily large models, then exactly one of the following holds:

  \begin{enumerate}
  \item $\Cat (\K) = \emptyset$.
  \item $\Cat (\K) = [\LS (\K)^{+m}, \LS (\K)^{+n}]$ for some $m \le n < \omega$.
  \item $\Cat (\K) = [\chi, \infty)$ for some $\chi \in [\LS (\K), \hanf{\LS (\K)})$.

  \end{enumerate}
\end{cor}
\begin{proof}
  Assume first that $\Cat (\K) \cap [\LS (\K)^{+\omega}, \infty) \neq \emptyset$. By Lemma \ref{cat-spec-2}, there is $\chi < \hanf{\LS (\K)}$ such that $[\chi, \infty) \subseteq \Cat (\K)$. Take the minimal such $\chi$ with $\chi \ge \LS (\K)$. We have that $\Cat (\K)  = [\chi, \infty)$. Indeed, if $\chi_0 \in \Cat (\K)$ and $\chi_0 < \chi$, then by Lemma \ref{cat-spec-1}, $[\chi_0, \chi] \subseteq \Cat (\K)$, and hence $[\chi_0, \infty) \subseteq \Cat (\K)$, contradicting the minimality of $\chi$.

  Assume now that $\Cat (\K) \cap [\LS (\K)^{+\omega}, \infty) = \emptyset$ but $\Cat (\K) \neq \emptyset$. Let $m < \omega$ be least such that $\LS (\K)^{+m} \in \Cat (\K)$ and let $n < \omega$ be maximal such that $\LS (\K)^{+n} \in \Cat (\K)$. Note that such an $n$ exists: otherwise by Fact \ref{closed-cat-spec}, $\LS (\K)^{+\omega} \in \Cat (\K)$, which we have assumed does not happen. By Lemma \ref{cat-spec-1}, $[\LS (\K)^{+m}, \LS (\K)^{+n}] \subseteq \Cat (\K)$, and by minimality of $m$ and maximality of $n$, this must be an equality: $\Cat (\K) = [\LS (\K)^{+m}, \LS (\K)^{+n}]$, as desired.
\end{proof}

\subsection{Examples}

For completeness, we end with the well known examples showing that each of the cases of Corollary \ref{main-cor} may happen. We will need:

\begin{fact}\label{coding}
  For each infinite cardinal $\lambda$ and each ordinal $\alpha < \left(2^{\lambda}\right)^+$, there exist AECs $\K^\alpha, \K^{\ast, \alpha}$ such that:

  \begin{enumerate}
  \item $\LS (\K^\alpha) = \LS (\K^{\ast, \alpha}) = \lambda$.
  \item $\K^\alpha$, $\K^{\ast, \alpha}$ have amalgamation.
  \item $\K^\alpha$ has a model of every cardinality $\mu \in [\lambda, \beth_{\alpha} (\lambda)]$, but no model of cardinality $\left(\beth_{\alpha} (\lambda)\right)^+$.
  \item $\K^{\ast, \alpha}$ has a model of every cardinality $\mu \in [\lambda, \lambda^{+\alpha})$ but no model of cardinality $\lambda^{+\alpha}$.
  \end{enumerate}
\end{fact}
\begin{proof}[Proof sketch]
  By \cite[VII.5.5(6)]{shelahfobook}, $\delta (\lambda, 2^{\lambda}) = \left(2^{\lambda}\right)^+$, where $\delta (\lambda, 2^{\lambda})$ is the least order type of a well-ordering not definable using the class of models of a first-order theory in a vocabulary of cardinality $\lambda$ omitting at most $2^{\lambda}$-many types. Thus by (the proof of \cite[VII.5.4]{shelahfobook}), for any $\alpha < \left(2^{\lambda}\right)^+$, we can code the cumulative hierarchy: there is an AEC $\K^{\alpha}$ such that $\LS (\K^{\alpha}) = \lambda$, $\K^{\alpha}$ has a model of cardinality $\beth_\alpha (\lambda)$ but not models of cardinality $\left(\beth_{\alpha} (\lambda)\right)^+$. 

  Coding successors is done similarly to obtain $\K^{\ast, \alpha}$: the basic template to code the successor operation is to work in a vocabulary with a linear order $<$, a predicate $P$, and a binary function $F$. We then consider structures where for any $x$ not in $P$, $F (x, \cdot)$ codes an injection from the predecessors of $x$ into $P$. The size of the universe of such a structure is at most the successor of the size of $P$ (the idea is probably folklore; it is described in \cite{grossbergbook}). Using this idea, one can define $\K^{\ast, \alpha}$ when $\alpha$ is a successor. When $\alpha$ is zero, we can take $\K^{\ast, 0} = \K^{0}$, and when $\alpha$ is limit we can take a disjoint union of $\K^{\ast, \beta + 1}$ for $\beta < \alpha$.
\end{proof}

\begin{fact}[\cite{hs-example, bk-hs}]\label{hs-fact}
  For each $n < \omega$, there is an AEC $\K^{\text{hs}, n}$ (``hs'' stands for Hart-Shelah) such that:

  \begin{enumerate}
  \item $\LS (\K^{\text{hs}, n}) = \aleph_0$.
  \item $\K^{\text{hs}, n}$ has amalgamation and arbitrarily large models.
  \item $\Cat (\K^{\text{hs}, n}) = [\aleph_0, \aleph_n]$.
  \end{enumerate}
\end{fact}

\begin{example}\label{categ-examples} \
  \begin{enumerate}
  \item The AEC of all fields (ordered by subfield) has empty categoricity spectrum, amalgamation, and arbitrarily large models. There are of course many other such examples.
  \item Fix an infinite cardinal $\lambda$ and an ordinal $\alpha < \left(2^{\lambda}\right)^+$. Take the disjoint union of the AEC $\K^\alpha$ given by Fact \ref{coding} with a totally categorical AEC to get an AEC $\K$ with amalgamation and arbitrarily large models such that $\LS (\K) = \lambda$ and $\Cat (\K) = [\beth_\alpha (\lambda)^+, \infty)$. Similarly, using $\K^{\ast, \alpha}$ we can get an AEC $\K'$ with amalgamation and arbitrarily large models such that $\LS (\K') = \lambda$ and $\Cat (\K) = [\lambda^{+\alpha}, \infty)$. If GCH holds, any $\chi \in [\lambda, \hanf{\lambda})$ is of the form $\lambda^{+\beta}$ for some $\beta < \left(2^{\lambda}\right)^+$, so this gives a complete list of examples for the third case of Corollary \ref{main-cor}. 
  \item Fix $m \le n < \omega$. Take a disjoint union of the Hart-Shelah example ($\K^{\text{hs}, n}$ given by Fact \ref{hs-fact}) and the AEC $\K^{\ast, m}$ given by Fact \ref{coding} (i.e.\ we take the coproduct of the Hart-Shelah example for categoricity up to $\aleph_n$ with an AEC with model only up to size $\aleph_{m - 1}$). We obtain an AEC $\K$ with amalgamation and arbitrarily large models such that $\LS (\K) = \aleph_0$ and $\Cat (\K) = [\LS (\K)^{+m}, \LS (\K)^{+n}]$. This gives a complete list of examples for the second case of Corollary \ref{main-cor} in case $\LS (\K) = \aleph_0$. In case $\LS (\K) > \aleph_0$, work in preparation of Shelah and Villaveces \cite{shvi648v1} constructs for each $n < \omega$ (assuming GCH) an AEC with Löwenheim-Skolem-number $\lambda$ which has categoricity spectrum $[\lambda, \lambda^{+n}]$. Using this, one could then also get a complete list of examples for the second case of Corollary \ref{main-cor} in case $\LS (\K) = \lambda$.
  \end{enumerate}
\end{example}
\begin{remark}
  The last two examples given above both fail joint embedding. It would be interesting to know whether there are examples with joint embedding, or whether (assuming $\WGCH$) for AECs with amalgamation, joint embedding, and arbitrarily large models, categoricity above $\LS (\K)^{+\omega}$ implies categoricity everywhere strictly above $\LS (\K)$.
\end{remark}

\section{Categoricity assuming no maximal models}\label{categ-nmm-sec}

In this section, we study the categoricity spectrum of AECs with no maximal models. Take an AEC with no maximal models categorical in a high-enough cardinal. By Fact \ref{semisolv-fact}(\ref{semisolv-1}), nice stability holds in this setup (assuming $\GCHWD$, see Section \ref{set-thy-sec}), so to apply Theorem \ref{categ-main}, it remains to check that there are superlimits in $\lambda$ and $\lambda^+$, for some suitable $\lambda$. The existence of superlimits may well follow from the uniqueness of limit models, but the latter is not known in this context (see the discussion in \cite{vandierennomax, nomaxerrata}). Instead, we will take $\lambda$ a suitable fixed point of the $\beth$ function and rely heavily on \cite[Chapter IV]{shelahaecbook} (we do \emph{not} rely on \cite[\S IV.2]{shelahaecbook}, which contains a gap -- see the discussion at the beginning of \cite[\S4]{ap-universal-apal}).

We give a name to the kind of fixed points we will consider:

\begin{defin}
  A cardinal $\lambda$ is a \emph{nice fixed point} if there exists a strictly increasing sequence $\seq{\lambda_n : n < \omega}$ such that $\lambda = \sup_{n < \omega} \lambda_n$ and $\lambda_n = \beth_{\lambda_n}$ for all $n < \omega$.
\end{defin}

Note that such fixed points can easily be found by iterating the $\beth$ functions enough times. The following result is where we will use \cite[Chapter IV]{shelahaecbook}. Essentially, everything in the conclusion below except the existence of the superlimit in $\K_{\lambda^+}$ is proven by Shelah there. For the convenience of the reader, we include here some of Shelah's definitions and results. Nevertheless, a sharp understanding of the proof of Theorem \ref{sl-nmm-thm} requires digging a little bit into some of Shelah's arguments. 

\begin{defin}[{\cite[IV.1.4.3]{shelahaecbook}}]
  An EM blueprint $\Phi \in \Upsilon[\K]$ (Definition \ref{solv-def}) \emph{witnesses $\mu$-pseudosolvability} (for $\mu \ge \LS (\K)$) if for every linear order $I$ of cardinality $\mu$, $M := \EM_{\tau (\K)} (I)$ has a proper extension and whenever $\delta < \mu^+$ is a limit ordinal and $\seq{M_i : i \le \delta}$ is increasing continuous, $M_i \cong M$ for all $i < \delta$ implies $M_\delta \cong M$. In other words, $M$ satisfies all the properties in the definition of a superlimit except universality.
\end{defin}
\begin{remark}\label{iso-sl-rmk}
  If $\Phi$ witnesses $\mu$-pseudosolvability, then for any linear orders $I, J$ of cardinality $\mu$, $\EM_{\tau (\K)} (I, \Phi) \cong \EM_{\tau (\K)} (J, \Phi)$. The proof is similar to that of Fact \ref{sl-uq} (embed $I$ and $J$ into a common extension). In particular, just as in Fact \ref{sl-basic}, the EM model of cardinality $\mu$ generates an AEC. We will use this without comments.
\end{remark}

\begin{fact}\label{gf-fact}
  Let $\K$ be an AEC. Let $\mu > \lambda \ge \LS (\K)$. If:

  \begin{enumerate}
  \item $\lambda$ is a nice fixed point.
  \item $\Phi$ witnesses $\mu$-pseudosolvability.
  \item For $M \lea N$ both isomorphic to $\EM_{\tau (\K)} (\mu, \Phi)$, $M \lee_{\Ll_{\infty, \lambda}} N$.
  \end{enumerate}

  Then $\Phi$ witnesses $\lambda$-pseudosolvability. Moreover there is a type-full good $\lambda$-frame on the AEC generated by $\EM_{\tau (\K)} (\lambda, \Phi)$.
\end{fact}
\begin{proof}
  The $\lambda$-pseudosolvability is \cite[IV.1.41]{shelahaecbook}. What Shelah refers to as hypothesis 1.18 is exactly the statement that for $M \lea N$ isomorphic to the EM model in $\mu$ witnessing pseudosolvability, $M \lee_{\Ll_{\infty, \lambda}} N$. Existence of the good frame is \cite[IV.4.10]{shelahaecbook}.
\end{proof}

\begin{thm}\label{sl-nmm-thm}
  Let $\K$ be an AEC and let $\mu > \lambda > \LS (\K)$. If:

  \begin{enumerate}
  \item $\lambda$ is a nice fixed point.
  \item $\K$ is $\mu$-solvable.
  \item For $M, N$ in $\Kslp{\mu}_\mu$, $M \lea N$ implies $M \lee_{\Ll_{\infty, \lambda}} N$.
  \item $\K$ is $(\lambda, \mu)$-extendible.
  \end{enumerate}

  Then $\K$ is $\lambda$-solvable, $\Kslp{\lambda}_{\lambda^+}$ has a superlimit, and there is a type-full good $\lambda$-frame on $\Kslp{\lambda}$.
\end{thm}
\begin{proof}
  Fix $\Phi$ witnessing $\mu$-solvability. In particular, it also witnesses $\mu$-pseudosolvability. By Fact \ref{gf-fact}, $\Phi$ witnesses $\lambda$-pseudosolvability. In the present setup, we are assuming that $\K$ is $(\lambda, \mu)$-extendible and $\mu$-solvable, and so in particular any $M \in \K_\lambda$ embeds inside $\EM_{\tau (\K)} (\mu, \Phi)$, hence inside $\EM_{\tau (\K)} (S, \Phi)$ for some $S \subseteq \mu$ of cardinality $\lambda$. This in turn must be isomorphic to $\EM_{\tau (\K)} (\lambda, \Phi)$ (Remark \ref{iso-sl-rmk}), so in fact $\EM_{\tau (\K)} (\lambda, \Phi)$ is universal, so superlimit: $\K$ is $\lambda$-solvable. In particular, $\K_\lambda$ has a superlimit.

  By Fact \ref{gf-fact}, there is a type-full good $\lambda$-frame on the class generated by the EM models witnessing pseudo-solvability. As argued in the previous paragraph, this class must be $\Kslp{\lambda}$.

  It remains to see that $\Kslp{\lambda}_{\lambda^+}$ has a superlimit. Since $\K$ has a superlimit in $\lambda$ and $\Kslp{\lambda}$ has a good $\lambda$-frame, we have by Fact \ref{frame-ss} and Theorem \ref{dense-nicely-stable} that $\K$ is nicely sl-$\lambda$-superstable. It is easy to check that the $(\lambda, \lambda^+)$-limit model is the saturated model in $\Kslp{\lambda}_{\lambda^+}$. Thus (using Theorem \ref{nicely-stable-sl} in the background), proving that $\Kslp{\lambda}_{\lambda^+}$ has a superlimit is equivalent to showing that for any limit $\delta < \lambda^{++}$, if $\seq{M_i : i < \delta}$ is an increasing chain of saturated models (in $\Kslp{\lambda}_{\lambda^+})$, then $\bigcup_{i < \delta} M_i$ is saturated. Work inside some saturated $M$ with $M_i \lea M$ for all $i < \delta$.

  The proof of \cite[IV.4.4]{shelahaecbook} gives that $\K_\lambda$ is $(<\kappa)$-tame, for some $\kappa < \lambda$. We claim that $\K$ (and hence $M$) does not have the $(\kappa,0)$-order property of length $\hanf{\kappa}$ (see \cite[4.3]{sh394} for the definition). Otherwise, by \cite[4.7.2]{sh394} (see \cite[5.13]{bgkv-apal} for a proof), the EM model of cardinality $\lambda$ would have such an order property, and this is impossible by \cite[3.4]{categ-saturated-afml}. Note that $\hanf{\kappa} < \lambda$, so we can apply \cite[4.27]{bv-sat-afml} to prove using stability theory inside a model and averages that $\bigcup_{i < \delta} M_i$ is saturated. Note that hypothesis (3) of \cite[4.27]{bv-sat-afml} holds because of the definition of the good frame in \cite[IV.4.10]{shelahaecbook} (see the proof of the local character property there).
\end{proof}

We will really use the following consequence of Theorem \ref{sl-nmm-thm} (the additional hypotheses are $\GCHWD (\lambda^+)$ and  $(\lambda^+,\mu)$-extendibility).

\begin{cor}\label{sl-nmm-cor}
  Let $\K$ be an AEC and let $\mu > \lambda > \LS (\K)$. If:

  \begin{enumerate}
  \item $\GCHWD (\lambda^+)$ holds.
  \item $\lambda$ is a nice fixed point.
  \item $\K$ is $\mu$-solvable.
  \item\label{mc-hyp} For $M, N$ in $\Kslp{\mu}_\mu$, $M \lea N$ implies $M \lee_{\Ll_{\infty, \lambda}} N$.
  \item $\K$ is $([\lambda, \lambda^+], \mu)$-extendible.
  \end{enumerate}

  Then both $\K_\lambda$ and $\K_{\lambda^+}$ have a superlimit.
\end{cor}
\begin{proof}
  By Theorem \ref{sl-nmm-thm}, $\K$ is $\lambda$-solvable hence in particular $\K_\lambda$ has a superlimit. In fact, Theorem \ref{sl-nmm-thm} tells us that there is a good $\lambda$-frame on $\Kslp{\lambda}$. In particular, $\K$ is nicely sl-$\lambda$-superstable, hence nicely $\lambda$-stable. By Fact \ref{semisolv-fact}(\ref{semisolv-1}) (where $\lambda$ here stands for $\lambda^+$ there), $\K$ is nicely $\lambda^+$-stable. By Theorem \ref{nicely-stable-sl}, $\Kslp{\lambda}_{\lambda^+}$ is dense in $\K_{\lambda^+}$ (recall Definition \ref{dense-def}). Theorem \ref{sl-nmm-thm} tells us that $\Kslp{\lambda}_{\lambda^+}$ has a superlimit, hence (Remark \ref{skel-basics}), $\K_{\lambda^+}$ has a superlimit as well.
\end{proof}

To obtain hypothesis (\ref{mc-hyp}) in Corollary \ref{sl-nmm-cor}, we will use:

\begin{fact}[{\cite[IV.1.12]{shelahaecbook}}]\label{mc-fact}
  Let $\K$ be an AEC and let $\mu \ge \lambda > \LS (\K)$. If $\K$ is categorical in $\mu$ and $\mu = \mu^{<\lambda}$, then whenever $M \lea N$ are both in $\K_\mu$, we have that $M \lee_{\Ll_{\infty, \lambda}} N$.
\end{fact}

We can now prove an upward categoricity transfer from categoricity in one cardinal in AECs with no maximal models. The categoricity cardinal is assumed to satisfy some cardinal arithmetic -- this will be remedied in the next result, essentially at the cost of assuming categoricity in two cardinals. Recall (Section \ref{set-thy-sec}) that $\GCHWD (\lambda)$ means that $2^\lambda = \lambda^+$ and enough instances of the weak diamond hold at $\lambda$.

\begin{thm}\label{categ-nmm-1}
  Let $\K$ be an AEC and let $\mu > \lambda > \LS (\K)$. If:

  \begin{enumerate}
  \item $\GCHWD ([\lambda, \lambda^{+\omega}))$.
  \item $\mu = \mu^{<\lambda}$ (or just: $M \lea N$ both in $\K_\mu$ implies $M \lee_{\Ll_{\infty, \lambda}} N$).
  \item $\lambda$ is a nice fixed point.
  \item $\lambda^{+\omega} \le \mu$.
  \item $\K$ is $([\lambda, \lambda^{+\omega}], \mu)$-extendible.
  \item $\K$ is categorical in $\mu$.
  \end{enumerate}

  Then $\K$ is categorical in every $\mu' \ge \mu$.
\end{thm}
\begin{proof}
  By Fact \ref{mc-fact}, $M \lea N$ both in $\K_\mu$ implies $M \lee_{\Ll_{\infty, \lambda}} N$. We want to apply Theorem \ref{categ-main} (the result will then follow, see Remark \ref{sl-categ-rmk}). We check its hypotheses. Since $\GCHWD ([\lambda, \lambda^{+\omega}))$, we have that $\WGCH ([\lambda, \lambda^{+\omega}))$. Since $\lambda$ is a nice fixed point, we have that $\hanf{\LS (\K)} < \lambda$, so $\K$ has arbitrarily large models (Fact \ref{arb-large-fact}). By Corollary \ref{sl-nmm-cor}, $\K_{\lambda}$ has a superlimit and $\K_{\lambda^+}$ has a superlimit. By Fact \ref{semisolv-fact}(\ref{semisolv-1}), categoricity in $\mu$ (Remark \ref{solv-rmk}) and the extendibility hypothesis, $\K$ is nicely $[\lambda^{+2}, \lambda^{+\omega})$-stable. This completes the verification that the hypotheses of Theorem \ref{categ-main} hold.
\end{proof}
\begin{remark}
  The proof shows that $\GCHWD ([\lambda, \lambda^{+\omega}))$ can be replaced with the conjunction of $\WGCH ([\lambda, \lambda^+])$ and $\GCHWD ([\lambda^{+2}, \lambda^{+\omega}))$.
\end{remark}

The next result removes the assumption that $\mu = \mu^{<\lambda}$ in Theorem \ref{categ-nmm-1} at the cost of assuming categoricity in two cardinals.

\begin{thm}\label{categ-nmm-2}
  Let $\K$ be an AEC and let $\mu_2 > \mu_1 \ge \lambda > \LS (\K)$. If:

  \begin{enumerate}
  \item $\GCHWD ([\lambda, \lambda^{+\omega}))$.
  \item $\WGCH (\mu_1)$.
  \item $\lambda$ is a nice fixed point.
  \item\label{ext-nmm-h} $\K$ is $([\lambda, \lambda^{+\omega}], \mu_2)$-extendible.
  \item $\K$ is categorical in $\mu_1$.
  \item $\K$ is $(\mu_1^+, \mu_2)$-extendible.
  \item $\K$ is categorical in $\mu_2$.
  \item If $\mu_1 = \lambda$, then $\mu_2 = \mu_2^{<\lambda}$.
  \end{enumerate}

  Then $\K$ is categorical in every $\mu \in [\mu_1, \mu_2]$. Moreover if $\mu_2 \ge \lambda^{+\omega}$, $\K$ is categorical in every $\mu' \ge \mu_1$.
\end{thm}
\begin{proof}
  If $\mu_2 = \lambda^+$, the conclusion is vacuous, so assume that $\mu_2 > \lambda^+$. As in the proof of Theorem \ref{categ-nmm-1}, $\K$ has arbitrarily large models. In particular (Remark \ref{solv-rmk}), $\K$ is $\mu_1$-solvable and $\mu_2$-solvable. By Fact \ref{semisolv-fact}(\ref{semisolv-4}), $\K$ is nicely $\mu_1$-stable. By Lemma \ref{nice-stab-elem}, categoricity in $\mu_1$, and clause (\ref{ext-nmm-h}), we get that for $M \lea N$ both in $\K_{\mu_1}$, $M \lee_{\Ll_{\infty, \mu_1}} N$. In particular, $M \lee_{\Ll_{\infty, \lambda}} N$. By Corollary \ref{sl-nmm-cor} (setting $\mu = \mu_1$ if $\mu_1 > \lambda$, or setting $\mu = \mu_2$ and using Fact \ref{mc-fact} if $\mu_1 = \lambda$), $\K_\lambda$ has a superlimit and $\K_{\lambda^+}$ has a superlimit. Let $\mu_2' := \min (\mu_2, \lambda^{+\omega})$. We apply Theorem \ref{categ-main}, where $\mu_0, \mu$ there stand for $\mu_2', \mu_2$ here. As before, Remark \ref{sl-categ-rmk} will then imply the result, but we first have to check the hypotheses of Theorem \ref{categ-main}. $\WGCH ([\lambda, \mu_2'))$ holds because we are assuming more: $\GCHWD ([\lambda, \lambda^{+\omega}))$. We have already established that $\K$ has arbitrarily large models, $\K_\lambda$ has a superlimit, and $\K_{\lambda^+}$ has a superlimit. That $\K$ is $(\mu_2', \mu_2)$-extendible follows from hypothesis (\ref{ext-nmm-h}) here, and that $\K$ is categorical in $\mu_2$ is also assumed. It remains to see that $\K$ is nicely $[\lambda^{+2}, \mu_2')$-stable: use Fact \ref{semisolv-fact}(\ref{semisolv-1}).
\end{proof}

The following is a variation on Fact \ref{closed-cat-spec} that does not use (but rather derives) amalgamation. It relies on the following result of Shelah \cite[I.3.8]{shelahaecbook}:

\begin{fact}\label{ap-categ-fact}
  Let $\K$ be an AEC and let $\lambda \ge \LS (\K)$. Assume $\WGCH (\lambda)$. If $\K$ is categorical in $\lambda$ and has a universal model in $\lambda^+$, then $\K$ has amalgamation in $\lambda$.
\end{fact}

\begin{thm}\label{categ-spec-lim}
  Let $\K$ be an AEC and let $\mu > \lambda \ge \LS (\K)$. Assume $\GCH ([\lambda, \mu))$. If $\mu$ is a limit cardinal and $[\lambda, \mu) \subseteq \Cat (\K)$, then $\mu \in \Cat (\K)$.
\end{thm}
\begin{proof}
  Fix $\theta \in [\lambda, \mu)$. Since $\mu$ is limit, $\theta^+ < \mu$, hence by assumption $\K$ is categorical in both $\theta$ and $\theta^+$. By Fact \ref{ap-categ-fact}, $\K_\theta$ has amalgamation. Since $\theta$ was arbitrary, it follows \cite[I.2.12]{shelahaecbook} that $\K_{[\lambda, \mu)}$ has amalgamation. Thus it suffices to see that any object in $\K_\mu$ is saturated. For this, it is in turn enough to show that for any $\theta \in (\lambda, \mu)$, the model of cardinality $\theta^+$ is saturated. This is an immediate consequence of $\GCH (\theta)$: $(\theta^+)^{<\theta^+} = \theta^+$ so on general grounds we can build a saturated model of cardinality $\theta^+$.
\end{proof}
\begin{remark}
  If $\K$ has arbitrarily large models, then we can replace $\GCH ([\lambda, \mu))$ by $\WGCH ([\lambda, \mu))$ (this is enough to derive amalgamation, and then one uses Fact \ref{closed-cat-spec} with Remark \ref{solv-rmk}).
\end{remark}

We are now ready to give a short list of possibilities for the categoricity spectrum above a nice fixed point. This shows in particular that the categoricity spectrum of an AEC with no maximal models is either bounded or contains an end segment (assuming $\GCHWD$). Here, we do not know whether the three intermediate possibilities listed below can happen (Conjecture \ref{categ-aec} would imply they cannot happen).

\begin{cor}\label{spec-nmm}
  Assume $\GCHWD$. Let $\K$ be an AEC and let $\lambda > \LS (\K)$ be a nice fixed point. If $\K_{\ge \lambda}$ has no maximal models, then exactly one of the following possibilities holds:

  \begin{enumerate}
  \item\label{spec-1} $\Cat (\K_{\ge \lambda}) = \emptyset$.
  \item\label{spec-2} $\Cat (\K_{\ge \lambda}) = [\lambda^{+m}, \lambda^{+n}]$ for some $m \le n < \omega$.
  \item\label{spec-3} $\Cat (\K_{\ge \lambda}) = \{\mu\}$, for some $\mu \ge \lambda^{+\omega}$ with $\mu < \mu^{<\lambda}$.
  \item\label{spec-4} $\Cat (\K_{\ge \lambda}) = \{\lambda, \mu\}$, for some $\mu \ge \lambda^{+\omega}$ with $\mu < \mu^{<\lambda}$.
  \item\label{spec-5} $\Cat (\K_{\ge \lambda}) = [\chi, \infty)$ for some $\chi \ge \lambda$.
  \end{enumerate}
\end{cor}
\begin{proof}
  First, we give a couple of consequences of the two theorems proven above:

  \underline{Claim 1}: If $\lambda < \mu_1 < \mu_2$ are such that $\mu_1, \mu_2 \in \Cat (\K)$, then $[\mu_1, \mu_2] \subseteq \Cat (\K)$. Moreover if in addition $\mu_2 \ge \lambda^{+\omega}$ then $[\mu_1, \infty) \subseteq \Cat (\K)$.

  \underline{Proof of Claim 1}: By Theorem \ref{categ-nmm-2}. $\dagger_{\text{Claim 1}}$

  \underline{Claim 2}: If $\mu > \lambda$ is such that $\mu = \mu^{<\lambda}$ and $\lambda, \mu \in \Cat (\K)$, then $[\lambda, \mu] \subseteq \Cat (\K)$.

  \underline{Proof of Claim 2}: By Theorem \ref{categ-nmm-2}, where $\mu_1, \mu_2$ there stand for $\lambda, \mu$ here. $\dagger_{\text{Claim 2}}$

  \underline{Claim 3}: If $\mu \ge \lambda^{+\omega}$ is such that $\mu = \mu^{<\lambda}$ and $\mu \in \Cat (\K)$, then $[\mu, \infty) \subseteq \Cat (\K)$.
    \underline{Proof of Claim 3}: By Theorem \ref{categ-nmm-1}. $\dagger_{\text{Claim 3}}$
  
    Now assume first that $\Cat (\K)$ is unbounded. By Claim 1, there must exist a cardinal $\chi \ge \lambda$ such that $[\chi, \infty) \subseteq \Cat (\K_{\ge \lambda})$. Take the least such $\chi$. By Claim 1 and Claim 2, we must have that $[\chi, \infty) = \Cat (\K_{\ge \lambda})$. Thus (\ref{spec-5}) holds.
        
        Now assume that $\Cat (\K)$ is bounded but (\ref{spec-1}) fails, i.e.\ $\Cat (\K_{\ge \lambda}) \neq \emptyset$. If $\Cat (\K_{\ge \lambda^{+\omega}})  \neq \emptyset$, then by the moreover part of Claim 1, $\Cat (\K_{> \lambda}) = \{\mu\}$ for some $\mu \ge \lambda^{+\omega}$, and by Claim 3, $\mu < \mu^{<\lambda}$. In this case either (\ref{spec-3}) or (\ref{spec-4}) must happen.

        Finally, if $\Cat (\K_{\ge \lambda^{+\omega}}) = \emptyset$ but $\Cat (\K_{\ge \lambda}) \neq \emptyset$, pick $m < \omega$  minimal such that $\lambda^{+m} \in \Cat (\K)$ and $n < \omega$ maximal such that $\lambda^{+n} \in \Cat (\K)$. As in the proof of Corollary \ref{main-cor}, such an $n$ must exist: if $\K$ were categorical in $\lambda^{+n}$ for all $n \ge m$, then by Theorem \ref{categ-spec-lim} it would be categorical in $\lambda^{+\omega}$, which we assumed does not happen. If $m = n$, (\ref{spec-2}) holds. Otherwise $n > 0$, so by $\GCH$, $\left(\lambda^{+n}\right)^{<\lambda} = \lambda^{+n}$. By Claim 1 and Claim 2, $[\lambda^{+m}, \lambda^{+n}] \subseteq \Cat (\K)$. By minimality of $m$ and maximality of $n$, this must be an equality, so (\ref{spec-2}) holds.
\end{proof}

We conclude with the eventual categoricity conjecture for AECs with no maximal models. Note that the proof of existence for the map $\lambda \mapsto \mu_\lambda$ below is not constructive: it does not give an explicit bound on $\mu_\lambda$.

\begin{cor}[Eventual categoricity for AECs with no maximal models]\label{event-nmm}
  Assuming $\GCHWD$, there exists a map $\lambda \mapsto \mu_\lambda$ such that for any AEC $\K$ with no maximal models, if $\K$ is categorical in \emph{some} $\mu \ge \mu_{\LS (\K)}$, then $\K$ is categorical in \emph{all} $\mu' \ge \mu_{\LS (\K)}$.
\end{cor}
\begin{proof}
  A consequence of Corollary \ref{spec-nmm} is that in any AEC with no maximal models, the categoricity spectrum is either bounded or contains an end segment. As shown in (for example) \cite[15.13]{baldwinbook09}, this suffices: for a fixed $\lambda$, there is only a set of AECs with Löwenheim-Skolem-Tarski number $\lambda$. By the axiom of replacement, there will exist $\mu_\lambda^0$ bounding the categoricity spectrum of all those where it is bounded. Similarly, there will exist $\mu_\lambda^1$ bounding the start of the end-segment for all those where it is unbounded. Take $\mu_\lambda := \mu_\lambda^0 + \mu_\lambda^1$.
\end{proof}

\section{Some auxiliary results}

We give here some easy corollaries of our methods that were buried in our notation. Specifically, we study tameness and existence of good frames in AECs with amalgamation, and also start the investigation of AECs with a lot of superlimits.

First, we show that good frames, and in fact successful ones, exist as low as possible in categorical AECs with amalgamation. This implies tameness and (if the categoricity cardinal is high-enough) a transfer of the structure of the class upward.

\begin{cor}\label{good-frame-ap-1}
  Assume $\WGCH$. Let $\K$ be an AEC with arbitrarily large models and let $\mu > \lambda \ge \LS (\K)$ be such that $\K_{<\mu}$ has amalgamation. If $\K$ is categorical in both $\lambda$ and $\mu$, then:

  \begin{enumerate}
  \item There is a $\goodp$ $\lambda$-frame on $\K$, which is $n$-successful whenever $\lambda^{+(n + 1)} \le \mu$.
  \item $\K_{< \mu}$ is $\lambda$-tame and categorical in all $\mu' \in [\lambda, \mu]$.
  \item If $\mu \ge \lambda^{+\omega}$, then $\K$ has amalgamation, no maximal models, is $\lambda$-tame, and is categorical in all $\mu' \ge \lambda$.
  \end{enumerate}
\end{cor}
\begin{proof}
  The proof of Lemma \ref{cat-spec-1} (where $\mu_1, \mu_2$ there stand for $\lambda, \mu$ here) shows that the hypotheses of Theorem \ref{n-ss-2} are satisfied. Now apply Facts \ref{succ-categ} and \ref{limit-categ}.
\end{proof}

If we assume just categoricity in one cardinal, we obtain the previous result on a subclass of saturated models:

\begin{cor}\label{good-frame-ap-2}
  Assume $\WGCH$. Let $\K$ be an AEC with arbitrarily large models and let $\mu > \lambda \ge \LS (\K)$ be such that $\K_{<\mu}$ has amalgamation and no maximal models. If $\mu > \lambda^{+2}$ and $\K$ is categorical in $\mu$, then:

  \begin{enumerate}
  \item There is a $\goodp$ $\lambda^+$-frame on $\Ksatp{\lambda^+}$, which is $n$-successful whenever $\lambda^{+(n + 2)} \le \mu$.
  \item $\Ksatp{\lambda^+}_{<\mu}$ is $\lambda^+$-tame. In particular (using the notation of \cite[11.6]{baldwinbook09}), $\K$ is $(\lambda^+, <\mu)$-weakly tame.
  \item If $\mu \ge \lambda^{+\omega}$, then $\K$ has amalgamation, no maximal models, and is $\lambda^+$-weakly tame.
  \end{enumerate}
\end{cor}
\begin{proof}
  Apply Corollary \ref{good-frame-ap-1} to $\Ksatp{\lambda^+}$ (it is an AEC by Fact \ref{categ-struct}).
\end{proof}

Grossberg and VanDieren have conjectured \cite[1.5]{tamenessthree} that tameness should follow from high-enough categoricity. We are now in a position to prove that this holds, assuming $\WGCH$ and amalgamation. This also proves Conjecture \ref{categ-aec} for AECs with amalgamation.

\begin{cor}\label{gv-cor}
  Assume $\WGCH$. Let $\K$ be an AEC with amalgamation. If $\K$ is categorical in \emph{some} $\mu \ge \hanf{\LS (\K)}$, then there is $\chi < \hanf{\LS (\K)}$ such that $\K$ is categorical in \emph{all} $\mu' \ge \chi$ and $\K$ is $\chi$-tame.
\end{cor}
\begin{proof}
  By Fact \ref{arb-large-fact}, $\K$ has arbitrarily large models. By Corollary \ref{main-cor}, there is $\chi < \hanf{\LS (\K)}$ such that $\K$ is categorical in all $\mu' \ge \chi$. By Corollary \ref{good-frame-ap-1} (with $\lambda = \chi$), $\K$ is $\chi$-tame.
\end{proof}

In \cite[II.3.7]{shelahaecbook}, Shelah gives a construction of a good frame using $\WGCH$, categoricity in $\lambda$, $\lambda^+$, and few models in $\lambda^{++}$. This can be iterated to get $n$-successful frames (see \cite[II.9.1]{shelahaecbook}). One can see Theorems \ref{n-ss-1} and \ref{n-ss-2} as variations where less structural assumptions are made in the $\lambda^{+m}$'s, but some global assumptions (semisolvability, which implies arbitrarily large models) are also made. We give here a consequence of Theorem \ref{n-ss-1} which shows that it is enough to assume that the class has superlimits in the $\lambda^{+m}$'s. Thus in a sense we replace categoricity in Shelah's theorem by superlimits, and we replace few models by semisolvability: 

\begin{cor}\label{sl-frame-cor}
  Let $\K$ be an AEC, let $\lambda \ge \LS (\K)$, and let $n < \omega$. If:

    \begin{enumerate}
    \item $\WGCH ([\lambda, \lambda^{+(n + 1)}])$.
    \item $\K_{\lambda^{+m}}$ has a superlimit for every $m \le n + 1$.
    \item $\K$ is $\lambda^{+(n + 2)}$-semisolvable.
    \end{enumerate}

    Then $\K$ is nicely sl-superstable and has sl-uniqueness triples in $[\lambda, \lambda^{+n}]$. In particular, there is a weakly successful good $\lambda$-frame on $\Kslp{\lambda}$, which will be $n$-successful $\goodp$ when $n \ge 1$.
\end{cor}
\begin{proof}
  The ``in particular'' part follows from Corollary \ref{very-nice-frame}  and Theorem \ref{very-nice-succ-thm}. To see that $\K$ is nicely sl-superstable and has sl-uniqueness triples in $[\lambda, \lambda^{+n}]$,  first observe that by the existence of superlimits, $\K_{[\lambda, \lambda^{+(n + 1)}]}$ has no maximal models, so $\K$ is $([\lambda, \lambda^{+(n + 1)}], \lambda^{+(n + 2)})$-extendible. By Fact \ref{semisolv-fact}(\ref{semisolv-4}), we get that $\K$ is nicely sl-$[\lambda, \lambda^{+(n + 1)}]$-superstable. Now apply Corollary \ref{weak-succ-constr} to get sl-uniqueness triples.
\end{proof}

Similarly, we get that in an AEC with a superlimit at every cardinal, the categoricity spectrum has no gaps. More precisely:

\begin{cor}\label{sl-categ-cor}
  Let $\K$ be an AEC and let $\mu > \lambda \ge \LS (\K)$. If:

  \begin{enumerate}
  \item $\WGCH ([\lambda, \min (\mu, \lambda^{+\omega})))$.
  \item $\K$ has arbitrarily large models.
  \item $\K$ is categorical in $\lambda$.
  \item For every $\lambda' \in (\lambda, \mu)$, $\K_{\lambda'}$ has a superlimit.
  \item $\K$ is categorical in $\mu$.
  \end{enumerate}

  Then $\K$ is categorical in every $\mu' \in [\lambda, \mu]$. Moreover if $\mu \ge \lambda^{+\omega}$, $\K$ is categorical in every $\mu' \ge \lambda$.
\end{cor}
\begin{proof}
  By definition of superlimits, $\K_{(\lambda, \mu)}$ has no maximal models. Categoricity in $\lambda$ and $\mu$, together with existence of arbitrarily large models, also imply that $\K_\lambda$ and $\K_\mu$ have no maximal models. Thus $\K_{[\lambda, \mu]}$ has no maximal models. By Fact \ref{semisolv-fact}(\ref{semisolv-4}), $\K$ is nicely sl-$[\lambda, \mu)$-superstable. Thus the hypotheses of Corollary \ref{categ-main-cor} hold, hence its conclusion.
\end{proof}

As in the proof of Corollary \ref{event-nmm}, we deduce that AECs with a superlimit everywhere satisfy the eventual categoricity conjecture. Note that having superlimits imply no maximal models, but as opposed to Corollary \ref{event-nmm}, Corollary \ref{sl-categ-cor} assumes only $\WGCH$ instead of $\GCHWD$.

In the opinion of the author, Corollaries \ref{sl-frame-cor} and \ref{sl-categ-cor} suggest that AECs with superlimits are an interesting framework to study. This was already foreshadowed by Shelah in the introduction to \cite{shelahaecbook}. Specifically, Shelah conjectures \cite[N.4.6]{shelahaecbook} that a variation of the eventual categoricity conjecture holds, but with solvability replacing categoricity. Corollary \ref{sl-categ-cor} shows that (a weakening of) this conjecture implies the eventual categoricity conjecture, assuming $\WGCH$. Thus the superlimit and (semi)solvability spectrums seem to be the key to an understanding of superstability and categoricity in abstract elementary classes.

\bibliographystyle{amsalpha}
\bibliography{categ-amalg}

\providecommand{\bysame}{\leavevmode\hbox to3em{\hrulefill}\thinspace}
\providecommand{\MR}{\relax\ifhmode\unskip\space\fi MR }
\providecommand{\MRhref}[2]{%
  \href{http://www.ams.org/mathscinet-getitem?mr=#1}{#2}
}
\providecommand{\href}[2]{#2}
\begin{thebibliography}{BGKV16}

\bibitem[Bal06]{baldwin-aec-survey-2006}
John~T. Baldwin, \emph{Abstract elementary classes: some answers, more
  questions}, URL:
  \url{http://homepages.math.uic.edu/~jbaldwin/pub/turino2.pdf}, 2006.

\bibitem[Bal09]{baldwinbook09}
\bysame, \emph{Categoricity}, University Lecture Series, vol.~50, American
  Mathematical Society, 2009.

\bibitem[BET07]{bet}
John~T. Baldwin, Paul~C. Eklof, and Jan Trlifaj, \emph{${}^\perp{N}$ as an
  abstract elementary class}, Annals of Pure and Applied Logic \textbf{149}
  (2007), 25--39.

\bibitem[BGKV16]{bgkv-apal}
Will Boney, Rami Grossberg, Alexei Kolesnikov, and Sebastien Vasey,
  \emph{Canonical forking in {A}{E}{C}s}, Annals of Pure and Applied Logic
  \textbf{167} (2016), no.~7, 590--613.

\bibitem[BGVV17]{shvi-notes-apal}
Will Boney, Rami Grossberg, Monica VanDieren, and Sebastien Vasey,
  \emph{Superstability from categoricity in abstract elementary classes},
  Annals of Pure and Applied Logic \textbf{168} (2017), no.~7, 1383--1395.

\bibitem[BK09]{bk-hs}
John~T. Baldwin and Alexei Kolesnikov, \emph{Categoricity, amalgamation, and
  tameness}, Israel Journal of Mathematics \textbf{170} (2009), 411--443.

\bibitem[Bon14]{tamelc-jsl}
Will Boney, \emph{Tameness from large cardinal axioms}, The Journal of Symbolic
  Logic \textbf{79} (2014), no.~4, 1092--1119.

\bibitem[BR12]{beke-rosicky}
Tibor Beke and Jiř{\'{\i}} Rosický, \emph{Abstract elementary classes and
  accessible categories}, Annals of Pure and Applied Logic \textbf{163} (2012),
  2008--2017.

\bibitem[BV17a]{bv-sat-afml}
Will Boney and Sebastien Vasey, \emph{Chains of saturated models in
  {A}{E}{C}s}, Archive for Mathematical Logic \textbf{56} (2017), no.~3,
  187--213.

\bibitem[BV17b]{bv-survey-bfo}
\bysame, \emph{A survey on tame abstract elementary classes}, Beyond first
  order model theory (Jos{\'e} Iovino, ed.), CRC Press, 2017, pp.~353--427.

\bibitem[BV17c]{tame-frames-revisited-jsl}
\bysame, \emph{Tameness and frames revisited}, The Journal of Symbolic Logic
  \textbf{82} (2017), no.~3, 995--1021.

\bibitem[BV18]{counterexample-frame-afml}
\bysame, \emph{Good frames in the {H}art-{S}helah example}, Archive for
  Mathematical Logic \textbf{57} (2018), 687--712.

\bibitem[DS78]{dvsh65}
Keith~J. Devlin and Saharon Shelah, \emph{A weak version of ${\Diamond}$ which
  follows from $2^{\aleph_0} < 2^{\aleph_1}$}, Israel Journal of Mathematics
  \textbf{29} (1978), no.~2, 239--247.

\bibitem[Gro]{grossbergbook}
Rami Grossberg, \emph{A course in model theory {I}}, A book in preparation.

\bibitem[Gro02]{grossberg2002}
\bysame, \emph{Classification theory for abstract elementary classes},
  Contemporary Mathematics \textbf{302} (2002), 165--204.

\bibitem[GV06a]{tamenessthree}
Rami Grossberg and Monica VanDieren, \emph{Categoricity from one successor
  cardinal in tame abstract elementary classes}, Journal of Mathematical Logic
  \textbf{6} (2006), no.~2, 181--201.

\bibitem[GV06b]{tamenessone}
\bysame, \emph{Galois-stability for tame abstract elementary classes}, Journal
  of Mathematical Logic \textbf{6} (2006), no.~1, 25--49.

\bibitem[GV06c]{tamenesstwo}
\bysame, \emph{Shelah's categoricity conjecture from a successor for tame
  abstract elementary classes}, The Journal of Symbolic Logic \textbf{71}
  (2006), no.~2, 553--568.

\bibitem[GV17]{gv-superstability-jsl}
Rami Grossberg and Sebastien Vasey, \emph{Equivalent definitions of
  superstability in tame abstract elementary classes}, The Journal of Symbolic
  Logic \textbf{82} (2017), no.~4, 1387--1408.

\bibitem[GVV16]{gvv-mlq}
Rami Grossberg, Monica VanDieren, and Andr{\'e}s Villaveces, \emph{Uniqueness
  of limit models in classes with amalgamation}, Mathematical Logic Quarterly
  \textbf{62} (2016), 367--382.

\bibitem[Hru96]{mordell-lang-hrushovski}
Ehud Hrushovski, \emph{The {M}ordell-{L}ang conjecture for function fields},
  Journal of the American Mathematical Society \textbf{9} (1996), no.~3,
  667--690.

\bibitem[HS90]{hs-example}
Bradd Hart and Saharon Shelah, \emph{Categoricity over ${P}$ for first order
  ${T}$ or categoricity for $\phi \in {L}_{\omega_1, \omega}$ can stop at
  $\aleph_k$ while holding for $\aleph_0, \ldots, \aleph_{k - 1}$}, Israel
  Journal of Mathematics \textbf{70} (1990), 219--235.

\bibitem[Jec03]{jechbook}
Thomas Jech, \emph{Set theory}, 3rd ed., Springer-Verlag, 2003.

\bibitem[JS]{jash940-v1}
Adi Jarden and Saharon Shelah, \emph{Non forking good frames without local
  character}, Preprint. URL: \url{http://arxiv.org/abs/1105.3674v1}.

\bibitem[JS13]{jrsh875}
\bysame, \emph{Non-forking frames in abstract elementary classes}, Annals of
  Pure and Applied Logic \textbf{164} (2013), 135--191.

\bibitem[LR16]{ct-accessible-jsl}
Michael~J. Lieberman and Jiř{\'{\i}} Rosický, \emph{Classification theory for
  accessible categories}, The Journal of Symbolic Logic \textbf{81} (2016),
  no.~1, 151--165.

\bibitem[Mor65]{morley-cip}
Michael Morley, \emph{Categoricity in power}, Transactions of the American
  Mathematical Society \textbf{114} (1965), 514--538.

\bibitem[MS90]{makkaishelah}
Michael Makkai and Saharon Shelah, \emph{Categoricity of theories in
  ${L}_{\kappa,\omega}$, with $\kappa$ a compact cardinal}, Annals of Pure and
  Applied Logic \textbf{47} (1990), 41--97.

\bibitem[She74]{sh31}
Saharon Shelah, \emph{Categoricity of uncountable theories}, Proceedings of the
  Tarski symposium (Leon Henkin, John Addison, C.~C. Chang, William Craig, Dana
  Scott, and Robert Vaught, eds.), American Mathematical Society, 1974,
  pp.~187--203.

\bibitem[She83]{sh87a}
\bysame, \emph{Classification theory for non-elementary classes {I}: The number
  of uncountable models of $\psi \in {L}_{\omega_1, \omega}$. {P}art {A}},
  Israel Journal of Mathematics \textbf{46} (1983), no.~3, 214--240.

\bibitem[She87a]{sh88}
\bysame, \emph{Classification of non elementary classes {II}. {A}bstract
  elementary classes}, Classification Theory (Chicago, IL, 1985) (John~T.
  Baldwin, ed.), Lecture Notes in Mathematics, vol. 1292, Springer-Verlag,
  1987, pp.~419--497.

\bibitem[She87b]{sh300-orig}
\bysame, \emph{Universal classes}, Classification theory (Chicago, IL, 1985)
  (John~T. Baldwin, ed.), Lecture Notes in Mathematics, vol. 1292,
  Springer-Verlag, 1987, pp.~264--418.

\bibitem[She90]{shelahfobook}
\bysame, \emph{Classification theory and the number of non-isomorphic models},
  2nd ed., Studies in logic and the foundations of mathematics, vol.~92,
  North-Holland, 1990.

\bibitem[She99]{sh394}
\bysame, \emph{Categoricity for abstract classes with amalgamation}, Annals of
  Pure and Applied Logic \textbf{98} (1999), no.~1, 261--294.

\bibitem[She00]{sh702}
\bysame, \emph{On what {I} do not understand (and have something to say), model
  theory}, Math. Japonica \textbf{51} (2000), 329--377.

\bibitem[She09a]{shelahaecbook}
\bysame, \emph{Classification theory for abstract elementary classes}, Studies
  in Logic: Mathematical logic and foundations, vol.~18, College Publications,
  2009.

\bibitem[She09b]{shelahaecbook2}
\bysame, \emph{Classification theory for abstract elementary classes 2},
  Studies in Logic: Mathematical logic and foundations, vol.~20, College
  Publications, 2009.

\bibitem[She10]{sh922}
\bysame, \emph{Diamonds}, Proceedings of the American Mathematical Society
  \textbf{138} (2010), no.~6, 2151--2161.

\bibitem[She15]{sh893-lwb}
\bysame, \emph{{A}.{E}.{C}. with not too many models}, pp.~367--402, De
  Gruyter, 2015.

\bibitem[SK96]{kosh362}
Saharon Shelah and Oren Kolman, \emph{Categoricity of theories in
  $\mathbb{{L}}_{\kappa, \omega}$, when $\kappa$ is a measurable cardinal.
  {P}art {I}}, Fundamentae Mathematica \textbf{151} (1996), 209--240.

\bibitem[SVa]{multidim-v2}
Saharon Shelah and Sebastien Vasey, \emph{Categoricity and multidimensional
  diagrams}, Preprint. URL: \url{ https://arxiv.org/abs/1805.06291v2}.

\bibitem[SVb]{shvi648v1}
Saharon Shelah and Andr{\'e}s Villaveces, \emph{Categoricity may fail late},
  Preprint. URL: \url{http://arxiv.org/abs/math/0404258v1}.

\bibitem[SV99]{shvi635}
\bysame, \emph{Toward categoricity for classes with no maximal models}, Annals
  of Pure and Applied Logic \textbf{97} (1999), 1--25.

\bibitem[SV18]{aec-stable-aleph0-apal}
Saharon Shelah and Sebastien Vasey, \emph{Abstract elementary classes stable in
  $\aleph_0$}, Annals of Pure and Applied Logic \textbf{169} (2018), no.~7,
  565--587.

\bibitem[Van06]{vandierennomax}
Monica VanDieren, \emph{Categoricity in abstract elementary classes with no
  maximal models}, Annals of Pure and Applied Logic \textbf{141} (2006),
  108--147.

\bibitem[Van13]{nomaxerrata}
\bysame, \emph{Erratum to "{C}ategoricity in abstract elementary classes with
  no maximal models" [{A}nn. {P}ure {A}ppl. {L}ogic 141 (2006) 108-147]},
  Annals of Pure and Applied Logic \textbf{164} (2013), no.~2, 131--133.

\bibitem[Van16a]{vandieren-symmetry-apal}
\bysame, \emph{Superstability and symmetry}, Annals of Pure and Applied Logic
  \textbf{167} (2016), no.~12, 1171--1183.

\bibitem[Van16b]{vandieren-chainsat-apal}
\bysame, \emph{Symmetry and the union of saturated models in superstable
  abstract elementary classes}, Annals of Pure and Applied Logic \textbf{167}
  (2016), no.~4, 395--407.

\bibitem[Vas]{tame-succ-v5-toappear}
Sebastien Vasey, \emph{Tameness from two successive good frames}, Israel
  Journal of Mathematics, To appear. URL:
  \url{http://arxiv.org/abs/1707.09008v5}.

\bibitem[Vas16a]{indep-aec-apal}
\bysame, \emph{Building independence relations in abstract elementary classes},
  Annals of Pure and Applied Logic \textbf{167} (2016), no.~11, 1029--1092.

\bibitem[Vas16b]{ss-tame-jsl}
\bysame, \emph{Forking and superstability in tame {A}{E}{C}s}, The Journal of
  Symbolic Logic \textbf{81} (2016), no.~1, 357--383.

\bibitem[Vas16c]{sv-infinitary-stability-afml}
\bysame, \emph{Infinitary stability theory}, Archive for Mathematical Logic
  \textbf{55} (2016), 567--592.

\bibitem[Vas17a]{downward-categ-tame-apal}
\bysame, \emph{Downward categoricity from a successor inside a good frame},
  Annals of Pure and Applied Logic \textbf{168} (2017), no.~3, 651--692.

\bibitem[Vas17b]{uq-forking-mlq}
\bysame, \emph{On the uniqueness property of forking in abstract elementary
  classes}, Mathematical Logic Quarterly \textbf{63} (2017), no.~6, 598--604.

\bibitem[Vas17c]{categ-saturated-afml}
\bysame, \emph{Saturation and solvability in abstract elementary classes with
  amalgamation}, Archive for Mathematical Logic \textbf{56} (2017), 671--690.

\bibitem[Vas17d]{ap-universal-apal}
\bysame, \emph{Shelah's eventual categoricity conjecture in universal classes:
  part {I}}, Annals of Pure and Applied Logic \textbf{168} (2017), no.~9,
  1609--1642.

\bibitem[Vas17e]{categ-universal-2-selecta}
\bysame, \emph{Shelah's eventual categoricity conjecture in universal classes:
  part {I}{I}}, Selecta Mathematica \textbf{23} (2017), no.~2, 1469--1506.

\bibitem[VV17]{vv-symmetry-transfer-afml}
Monica VanDieren and Sebastien Vasey, \emph{Symmetry in abstract elementary
  classes with amalgamation}, Archive for Mathematical Logic \textbf{56}
  (2017), no.~3, 423--452.

\bibitem[Zil05]{zilber-pseudoexp}
Boris Zilber, \emph{Pseudo-exponentiation on algebraically closed fields of
  characteristic zero}, Annals of Pure and Applied Logic \textbf{132} (2005),
  67--95.

\bibitem[Ło54]{los-conjecture}
Jerzy Łoś, \emph{On the categoricity in power of elementary deductive systems
  and some related problems}, Colloquium Mathematicae \textbf{3} (1954), no.~1,
  58--62.

\end{thebibliography}

\end{document}